\documentclass[11pt,a4paper]{article}
\usepackage{comment,listings,framed,graphicx,url,xcolor,enumerate, hyperref, amsmath, float, multicol, multirow}
\usepackage{fullpage} 
\usepackage{todonotes}
\definecolor{mild}{rgb}{1,0.98,0.8}
\usepackage[utf8]{inputenc}    
\usepackage[T1]{fontenc}       
\usepackage{lmodern}           
\usepackage[margin=.8in]{geometry}
\usepackage{graphicx}          
\usepackage{amsmath,amssymb}   
\usepackage{booktabs}          
\usepackage{caption}           
\usepackage{subcaption}        
\usepackage{microtype}         
\usepackage{hyperref}          
\usepackage{xcolor}            
\usepackage{amsthm}
\usepackage{amsmath}
\DeclareMathOperator{\sech}{sech}
\usepackage[numbers,sort]{natbib}

\numberwithin{equation}{section}
\numberwithin{figure}{section}
\numberwithin{table}{section}

\newtheorem{theorem}{Theorem}[section]   
\newtheorem{lemma}[theorem]{Lemma}       

\theoremstyle{definition}

\theoremstyle{remark}
\newtheorem{remark}[theorem]{Remark}

\hypersetup{
  colorlinks=true,
  linkcolor=blue,
  citecolor=blue,
  urlcolor=blue
}

\title{Structure Preserving High-Order FDTD Methods for the Nonlinear Kerr-Debye Model with Lorentz Dispersion}
\author{%
   Emmanuel E. Oguadimma\thanks{Oregon State University, oguadime@oregonstate.edu},\, \, 
   Vrushali A. Bokil \thanks{Oregon State University, bokilv@oregonstate.edu}, \, \,
   Nathan L.~Gibson \thanks{Oregon State University, gibsonn@oregonstate.edu}
}
\date{\today}

\begin{document}

\maketitle

\begin{abstract}
We develop and analyze several families of fully discrete, high-order methods in the finite difference time domain (FDTD) framework for a nonlinear dispersive electromagnetic model for ultrashort pulse propagation in a nonlinear optical medium. The model consists of the time domain Maxwell's equations coupled to two ordinary differential equations; the nonlinear Kerr-Debye equation and the linear dispersive Lorentz equation, in one space dimension. The proposed methods combine staggered spatial discretizations of $2M$th-order, for integer $M>0$, with second-order explicit leapfrog or implicit trapezoidal time integrators. To handle the stiff relaxation in the nonlinear Kerr-Debye equation, we adopt a modified exponential integrator, introduced in \cite{PengJCP2020}, which ensures asymptotic-preserving (AP) behavior in the limit of vanishing relaxation time. We prove that the resulting schemes satisfy a discrete analogue of the continuous energy identity and preserve the positivity (PP) of the nonlinear susceptibility. For the schemes based on the leapfrog time integrator, we establish energy stability under a CFL condition, while the schemes based on the trapezoidal time integrator are unconditionally energy stable. High-order temporal accuracy in the trapezoidal time integrator based schemes is further achieved via Richardson extrapolation yielding $2M$th-order accurate methods in both space and time. Numerical experiments demonstrate the theoretical convergence rates and the ability of the proposed schemes to capture nonlinear wave phenomena such as soliton propagation, self-steepening, and odd harmonic generation. This work extends previous AP-PP and energy stable approaches in a discontinuous Galerkin framework to the FDTD setting, offering a computationally efficient, flexible and robust tool for simulating electromagnetic wave propagation in nonlinear dispersive optical media.

\end{abstract}


\section{Introduction}

Electromagnetic wave propagation in nonlinear materials plays an important role in fields such as plasma physics, modern microelectronics, nonlinear optics and photonics, laser technology and micromagnetics. The study of electromagnetic wave propagation in nonlinear optical media has been a subject of significant research for several decades \cite{gilles2000comparison,gilles1999electromagnetic,goorjian2002computational,hile1996numerical,taflove2000computational,ziolkowski1993full}. Nonlinear optics and photonics has grown in importance since the invention of the laser in 1960, which provided a ready source of high-intensity, near-monochromatic light. Among the key nonlinear effects encountered during wave propagation in such media are second and third harmonic generation, self-focusing, defocusing, and self-channeling \cite{agrawal2019nonlinear,boyd2019nonlinear}. Given the diverse ways in which nonlinearities influence wave dynamics, it is important to develop numerical methods that are accurate, stable, efficient, and capable of preserving key structural properties of the underlying mathematical model.

The propagation of electromagnetic waves in a nonlinear optical medium is governed by Maxwell's equations in the time domain, coupled with constitutive relations that describe the material response. 
In a linear dielectric medium, this response is characterized by the macroscopic polarization field, which represents the average dipole moment per unit volume. 
In a one-dimensional formulation, the linear polarization $P$ is given by \cite{joseph1991direct}
\begin{equation}
    P = \epsilon_0 \chi^{(1)} E,
\end{equation}
where $\epsilon_0$ is the permittivity of free space, $\chi^{(1)}$ is the linear or first order susceptibility, and $E$ is the propagating electric field. The emergence of ultrashort laser pulses capable of generating very high electric fields transformed our understanding of light matter interactions. Under such intense fields, the polarization response of a material is no longer strictly linear and can be expressed as a power series in the electric field \cite{joseph1997fdtd}:
\begin{equation}\label{eq:pol}
P = \epsilon_0 \left(\chi^{(1)}E + \chi^{(2)}E^2 + \chi^{(3)}E^3 + \chi^{(4)}E^4 + \cdots \right),
\end{equation}
where $\chi^{(n)}$ are the nonlinear susceptibilities of order $n\in \mathbb{N}, n \geq 2$. As a result, in addition to the well-understood linear behavior, intense electric fields can induce specific nonlinear effects in materials, depending on which term in the series in \eqref{eq:pol} is the dominant term. For specific types of susceptibilties, the relationship between the polarization and the electric field can be converted into linear or nonlinear ordinary differential equations (ODEs) \cite{banks2000electromagnetic,gilles2000comparison}. These ODEs are appended to Maxwell's equations to provide a coupled model for electromagnetic wave propagation in the nonlinear optical material. This approach is called the auxilliary differential equation method (ADE) and is a popular technique for constructing such models \cite{taflove2000computational} 

In this paper, we consider nonlinear optical materials that exhibit first order (linear) and third order nonlinear effects. The linear response of the nonlinear optical medium to the propagating electromagnetic waves in the one-dimensional case is modeled using the linear Lorentz model, an ordinary differential equation model for a harmonic oscillator driven by the electric field. The third order nonlinear response is modeled by the Kerr-Debye model, a nonlinear ordinary differential equation driven by the electric field. Similar to the electro-optical Kerr effect \cite{weinberger2008john}, where the refractive index of a material changes proportionally to the square of a static electric field, the presence of oscillating electric fields leads to a quadratic intensity-dependent change in the susceptibility. This phenomenon is known as the optical Kerr effect \cite{gilles2000comparison, PengJCP2020}. The corresponding cubic nonlinearity in the polarization is frequently
treated as an instantaneous response, which is a valid approximation when the pulse duration is very short
\cite{sorensen2002solitary,webb2004variational}. The nonlinear Kerr-Debye model discussed here includes a relaxation mechanism within the Kerr effect. This relaxation often occurs on a time scale of less than a femtosecond, and the standard instantaneous Kerr model is recovered when relaxation is neglected. However, in describing ultrashort pulses, the effect of a finite relaxation time can be important \cite{gilles1999electromagnetic, PengJCP2020,ziolkowski1993full}. For example, in \cite{gilles1999electromagnetic}, the authors analyzed the propagation of ultrashort pulses through a dispersive medium, where the linear response is modeled by three Lorentz dispersion terms and the nonlinear Kerr effect includes a finite Debye relaxation time. Notable consequences of this relaxation time include the suppression of phase matched harmonic generation and a slowdown in the self-steepening behavior of the electric field envelope \cite{gilles1999electromagnetic}.

The finite difference time domain (FDTD) method, also known as the Yee scheme \cite{yee1966numerical}, was first introduced by Kane Yee in 1966 for the numerical discretization of the 3D time domain Maxwell's equations in a linear medium like vacuum. It remains one of the most widely adopted numerical approaches to simulate electromagnetic wave propagation in linear and nonlinear materials. The original numerical method for linear media is nondissipative but dispersive \cite{taflove2000computational}, and it discretizes Maxwell's equations on a staggered spatial and temporal grid, resulting in second-order accuracy in both space and time. One of the primary advantages of the Yee scheme is its simplicity and explicit formulation, which enables direct time integration without the need for global matrix inversion. This characteristic also makes it highly amenable to efficient implementation on parallel computing platforms. As a result, it is especially well suited for modeling wave propagation in nonlinear media, where the material properties depend on the local electromagnetic field.

The Yee-FDTD method has been extended to discretize materials characterized as linear dispersive, nonlinear and nonlinear dispersive optical media \cite{kashiwa1990treatment,joseph1991direct,goorjian2002computational, hile1996numerical, sorensen2005kink, gilles1999electromagnetic, ziolkowski1994nonlinear, liu2022fdtd}. Recent developments have extended the second-order Yee-FDTD method to achieve arbitrary spatial accuracy for models in nonlinear optics. For example, several studies have explored higher order Yee-FDTD schemes applied to Maxwell's equations in both linear \cite{bokil2012analysis, BokilGibsonXXXX} and nonlinear  media \cite{bokil2018high,gilles2000comparison}. Beyond FDTD, a variety of other numerical techniques have been developed for electromagnetic simulations, including finite element methods (FEM) \cite{li2006analysis, li2007error}, high-order discontinuous Galerkin methods \cite{PengJCP2020, gedney2012discontinuous, huang2011interior, zhu2012hybrid,huang2017second}, finite volume methods \cite{de2000high}, and Godunov schemes \cite{AregbaDriolletHanouzet2011, carbou2009relaxation}.

Numerical simulation of wave propagation phenomenon typically introduces numerical dispersion and dissipation errors that need to be controlled. The nonlinear Kerr-Debye model introduces further challenges that require special consideration. In the Kerr-Debye framework, the relaxation of the Kerr nonlinearity is governed by an auxiliary differential equation (ADE) for the third order nonlinear susceptibility, $\chi :=\chi^{(3)}$ given as 
\begin{align}\label{debyeade}
    \partial_t \chi = -\frac{1}{\varepsilon}(\chi - E^2).
\end{align}
Here, the parameter $\varepsilon > 0$ denotes the relaxation time.
In the limiting case, as $\varepsilon$ approaches zero, $\chi$
converges to $E^2$, thus reducing the model to the classical Kerr formulation \cite{carbou2007relaxation,carbou2009relaxation}.

The nonlinear susceptibility $\chi$ is nonnegative, and it is desirable for a numerical scheme to preserve this property at the discrete level. This requirement is important because the nonnegativity of $\chi$ is closely related to the energy stability of the model. Numerical methods that enforce this discrete nonnegativity property are commonly referred to as positivity-preserving (PP) schemes \cite{PengJCP2020}. Additionally, equation \eqref{debyeade} becomes stiff when $\varepsilon$ is very small, so conventional numerical methods may be subject to a strict stability restriction, which typically requires the time step to scale as $\Delta t = \mathcal{O}(\varepsilon)$. To address both the stiffness and positivity requirements, we employ a modified exponential time integrator, as proposed in \cite{PengJCP2020}.

Our goal is to devise and analyze a family of high-order FDTD methods for Maxwell's equations in the nonlinear optical medium that preserves key properties of the nonlinear Kerr-Debye and linear Lorentz model at the discrete level. We aim to capture the effects of finite relaxation time and propose schemes that are asymptotic-preserving (AP), positivity-preserving (PP) and preserve a discrete analog of the continuous energy decay property. AP methods remain stable and accurate for a wide range of relaxation times $\varepsilon$, and as $\varepsilon \to 0$, they recover the limiting Kerr-Lorentz model without requiring fine resolution. 

To construct the proposed methods, we discretize Maxwell's equations in space using high-order staggered Yee type schemes and a staggered leapfrog time-stepping method, providing overall second-order accuracy. Following \cite{bokil2018high}, we also construct an implicit family of methods based on a second-order trapezoidal discretization in time while maintaining the high-order staggered discretizations in space. Both formulations allow arbitrarily high spatial accuracy, similar to the methods in \cite{bokil2018high}. We combine all discretizations for Maxwell's equations, explicit or implicit, with a trapezoidal time integrator for evolving the Kerr-Debye and Lorentz auxiliary differential equations (ADEs). Thus, the discretizations of the ordinary differential equations (ODEs) modeling the material response are implicit and second-order accurate in time. 

We further increase the temporal accuracy of the implicit time methods using Richardson extrapolation \cite{leveque1998finite, chen2025developing,ascoli2018limitations,bach2021effectiveness}. For the limiting model coupling Maxwell's equations with the nonlinear Kerr model and the linear Lorentz model, the temporal error expansion contains only even powers of the timestep, $\Delta t$, allowing the construction of $(2M,2M)$ schemes from $M$ successively refined solutions. For the model coupling Maxwell's equations with the nonlinear Kerr-Debye model and the linear Lorentz model, the modified exponential integrator introduces both even and odd powers of $\Delta t$ in the temporal error expansion. We therefore develop a modified Richardson extrapolation that uses $2M-1$ successively refined solutions to obtain $2M$th-order accuracy in time. We denote schemes that are second-order in time and $2M$th-order in space by $(2,2M)$ schemes, and schemes that are $2M$th-order accurate in both time and space by $(2M,2M)$ schemes. The underlying $(2,2M)$ schemes are asymptotic-preserving and positivity-preserving and satisfy discrete energy stability properties. For the Richardson extrapolated schemes, we additionally establish a uniform-in-time energy bound in the limiting case of Kerr nonlinearlity. The construction and analysis of these high-order extrapolated methods constitute some of the main contributions of this work.

The outline of this paper is as follows. In Section~\ref{sec:model}, we introduce the one-dimensional Maxwell-Kerr-Debye-Lorentz model and summarize its key continuous properties, including positivity of the nonlinear susceptibility and energy dissipation. We also present the limiting Maxwell-Kerr-Lorentz model obtained as $\varepsilon\to0$. In Section~\ref{subsec1.2}, we construct structure-preserving $(2,2M)$ FDTD schemes using staggered $2M$th-order spatial discretizations together with explicit leapfrog and implicit trapezoidal time discretizations. We establish asymptotic-preserving and positivity-preserving properties and derive discrete energy relations for these schemes. In Section~\ref{2M2M}, we construct $(2M,2M)$ extrapolated trapezoidal FDTD schemes for the limiting Maxwell-Kerr-Lorentz model using Richardson extrapolation and establish a uniform-in-time energy bound for the extrapolated solution. In Section~\ref{2M2MKDL}, we develop a modified Richardson extrapolation for the Maxwell-Kerr-Debye-Lorentz model that yields $(2M,2M)$ accuracy. Section~\ref{sec:implementation} describes the numerical implementation of the proposed methods. Numerical experiments demonstrating the theoretical convergence rates, energy behavior, and nonlinear wave propagation phenomena are presented in Section~\ref{sec:numerics}. Finally, conclusions and directions for future work are given in Section~\ref{sec:conclusion}.

\section{The Maxwell Kerr-Debye Model with Lorentz Dispersion}\label{sec:model}
Let $\Omega \subset \mathbb{R}^3$ and $T>0$. The model for light-matter interactions is given by Maxwell’s equations along with constitutive relationships describing the material response. Assuming no free charges or sources, Maxwell's equations are written as
\begin{subequations}\label{eq: mxk}
\begin{align}
\frac{\partial \mathbf{B}}{\partial t} &= -\nabla \times \mathbf{E}, \, \hspace{18mm} \text{ in } \Omega \times (0, T)\\
\frac{\partial \mathbf{D}}{\partial t}  &= \nabla \times \mathbf{H}, \, \hspace{20.5mm}  \text{ in } \Omega \times (0, T)\\
\nabla \cdot \mathbf{D} &=  \nabla \cdot \mathbf{B} = 0, \hspace{16mm} \text{ in } \Omega \times (0, T)\\
\mathbf{E}(\mathbf{x},0) &= \mathbf{E}_0 (\mathbf{x}); \,\,  \mathbf{H}(\mathbf{x},0) = \mathbf{H}_0(\mathbf{x}) \quad \text{ in }\Omega. \label{InitC}
\end{align}
\end{subequations}
The fields $\mathbf{D}, \mathbf{E}, \mathbf{B}$, and $\mathbf{H}$ denote the electric displacement, electric field, magnetic induction, and magnetic field intensity, respectively. All fields are functions of space and time. We supplement the system with the initial conditions \eqref{InitC}. Unless otherwise stated, we assume periodic boundary conditions throughout the analysis. In some of the numerical experiments presented in Section~\ref{KDLi}, other boundary conditions are employed as appropriate and are specified where they are introduced. 

Constitutive laws describe relationships between the different electromagnetic variables and depend on the material response. For a nonlinear optical material. these are generally given as \cite{jackson1977classical}: 
\begin{subequations}
    \begin{align}
        \mathbf{D} &= \epsilon_0 (\mathbf{E} + \mathbf{P}),\label{constlawfirst} \\
        \mathbf{B} &= \mu_0\mathbf{H},
    \end{align}
\end{subequations}
where $\mathbf{P}$ is the macroscopic polarization, $\epsilon_0 > 0$ is the permittivity of free space, and $\mu_0 > 0$ is the permeability of free space. The medium is nonmagnetic, i.e., the constitutive relationship for the magnetic field is the same as in a vacuum. 

The polarization incorporates both linear and nonlinear optical responses. We write 
\begin{align*}
    \mathbf{P}= \mathbf{P^L} + \mathbf{P^{NL}},
\end{align*}
where $\mathbf{P^{NL}}$ is the nonlinear polarization and $\mathbf{P^{L}}$ describes the linear polarization effects. In this paper, we consider media that have physical dispersion. Dispersion refers to the frequency-dependence of the speed of propagation of waves through a medium.  As a result, when a pulse of light contains a mix of frequencies, those components spread out as the pulse travels even in a homogeneous medium \cite{banks2000electromagnetic,BokilGibsonXXXX}. This spreading effect can change the shape of the pulse over time. In the next sections, we describe the specific linear and nonlinear optical responses that we will consider in this paper. 

\subsection{Linear response}
 We consider both instantaneous and delayed linear responses. For the delayed linear response we assume electronic polarization that is modeled by the Lorentz model \cite{banks2000electromagnetic,BokilGibsonXXXX}. The linear part of the polarization is given as
 \begin{equation}\label{linearpol}
   \mathbf{P^L} = \epsilon_\infty \mathbf{E} + \sum_{k=1}^{N_\omega}\mathbf{P}_k,
 \end{equation}
 where $\epsilon_\infty$ is the relative permittivity at infinite frequency of the material and the term $\epsilon_\infty \mathbf{E}$ captures the linear instantaneous response of the material. For $k=1,\dots,N_\omega$, the $k$-th Lorentz pole $\mathbf{P}_k$ is driven by the electric field and is modeled as a harmonic oscillator
    \begin{equation}
    \frac{\partial^2 \mathbf{P}_k}{\partial t^2}  + \frac{1}{\tau_k} \frac{\partial \mathbf{P}_k}{\partial t} + \omega_{0,k}^2\,\mathbf P_k \;=\; \,\omega_{p,k}^2\,\mathbf E,
    \qquad k=1,\dots,N_\omega,
    \label{eq:lorentz-ode}
    \end{equation}
     where $\tau_k, \omega_{0,k}$ and $\omega_{p,k}$ are the damping time, resonance frequency and plasma frequency, respectively, of the $k$th Lorentz pole.  
    The second-order differential equations are converted to first order form by introducing new variables $\mathbf{J}_k$ representing the $k$th polarization current density.  
    \begin{align}
    \partial_{t} \mathbf{P}_k&=\mathbf J_k, \label{lorsrc}\\
   \partial_{t} \mathbf{J}_k + \frac{1}{\tau_k} \mathbf J_k &+ \omega_{0,k}^2 \mathbf P_k = \,\omega_{p,k}^2\,\mathbf E \label{lorp}.
    \end{align}

\subsection{Nonlinear response}
The standard treatment of the optical Kerr effect \cite{weinberger2008john} assumes an instantaneous nonlinear response which means that the refractive index at any time depends only on the instantaneous intensity of the electromagnetic field, which is appropriate when field variations are slow compared to polarization dynamics \cite{BokilGibsonXXXX}. This assumption breaks down for ultrashort optical pulses on the order of femtoseconds or for materials with nonnegligible response times. In these cases, a finite relaxation process must be included to represent the material behavior accurately. To do so, we use the nonlinear Kerr-Debye model \cite{PengJCP2020}, which introduces relaxation into the nonlinear polarization through an auxiliary differential equation. This model accounts for both the intensity dependent nonlinearity and the finite time memory of the medium, yielding a more realistic description when the nonlinear susceptibility evolves in time. Consequently, the nonlinear refractive index does not respond instantaneously to changes in field intensity but instead evolves according to a characteristic relaxation time $\varepsilon$. 

Let $a>0$ be a coefficient of nonlinearity. We model $\mathbf{P^{NL}}$ as 
\begin{align}\label{NLPol}
	\mathbf{P^{NL}} = a\chi \mathbf{E}, 
\end{align}
where $\chi(\mathbf{x}, t)$ is the third-order nonlinear susceptibility. Choosing $\chi = |\mathbf{E}|^2$, \eqref{constlawfirst} becomes 
\begin{equation}
\label{eq:DKL}
\mathbf{D} = \mathbf{D}_{\mathrm{KL}} := \epsilon_0 \left(\epsilon_\infty + a|\mathbf{E}|^2\right)\mathbf{E}+ \epsilon_0\sum_{k=1}^{N_\omega}\mathbf{P}_k,
\end{equation}
incorporating both the instantaneous Kerr nonlinearity and the linear Lorentz dispersion. To capture the finite relaxation time in the nonlinear polarization response \eqref{NLPol}, we employ the nonlinear Kerr-Debye model which is given by 
\begin{align}\label{KerrDeb}
	\partial_t \chi = - \frac{1}{\varepsilon}\left(\chi - |\mathbf{E}|^2\right),
\end{align}
and \eqref{constlawfirst} becomes 
\begin{equation}\label{constlawKD}
\mathbf{D} = \mathbf{D}_{\mathrm{KDL}}:=\epsilon_0 \left(\epsilon_\infty + a\chi\right)\mathbf{E}+ \epsilon_0\sum_{k=1}^{N_\omega}\mathbf{P}_k,
\end{equation}
where $\varepsilon$ is the relaxation time of the nonlinear Kerr effect. 

In the next section, we present the complete nonlinear electromagnetic model. If the context is clear, we will drop the subscript on the electric displacement $\mathbf{D}$ in the constitutive laws 
\eqref{eq:DKL} and \eqref{constlawKD}.

\subsection{1D Maxwell-Kerr-Debye-Lorentz System}
We will consider two models in the sequel. The combination of the Maxwell system \eqref{eq: mxk} with \eqref{lorsrc}, \eqref{lorp} \eqref{KerrDeb} and \eqref{constlawKD} will be denoted the \textit{Maxwell Kerr-Debye-Lorentz} model. In this paper, we consider the one-dimensional version of the Maxwell-Kerr-Debye-Lorentz system by assuming an isotropic and homogeneous material in which electromagnetic plane waves are linearly polarized and propagate in the $x$ direction. Thus, the electric field is represented by one scalar component $E := E_z$, while the magnetic field is represented by one component $H := H_y$. All other variables are represented similarly by single scalar components. 

Let $\Omega \subset \mathbb{R}$. We consider a single Lorentz pole with $N_\omega = 1$. The one-dimensional Maxwell-Kerr-Debye-Lorentz model on $\Omega \times (0, T]$ is 
\begin{subequations}\label{eq:kerr-debyelorentz1da}
    \begin{align}
        \mu_0 \partial_t H &= \partial_x E, \label{faraday} \\
    \partial_t D &= \partial_x H, \label{ampere} \\
    \partial_t \chi &= -\frac{1}{\varepsilon} (\chi - E^2), \label{kerrdebye1d}\\
    \partial_t P &= J, \label{density} \\
    \partial_t J &+ \frac{1}{\tau} J + \omega_0^2 P = \omega_p^2 E, \label{lorentz}\\
     D &= \epsilon_0 \left( (\epsilon_\infty + a \chi) E + P \right). \label{constlaw}
    \end{align}
\end{subequations}
Formally, as the relaxation time $\varepsilon$ goes to $0$, \eqref{kerrdebye1d} becomes $\chi = E^2$, and we obtain the asymptotic limit of the system \eqref{eq:kerr-debyelorentz1da}, denoted as the {\it Maxwell-Kerr-Lorentz} model \cite{bokil2018high}:
\begin{subequations}\label{eq:kerr-debyelorentz1db}
    \begin{align}
        \mu_0 \partial_t H &= \partial_x E, \label{faradaya}\\
        \partial_t D &= \partial_x H, \label{amperea} \\
        \partial_t P &= J, \label{densitya} \\
        \partial_t J &+ \frac{1}{\tau} J + \omega_0^2 P = \omega_p^2 E, \label{lorentza}\\
        D &= \epsilon_0 \left( (\epsilon_\infty + a E^2) E + P \right). \label{constlawa}
    \end{align}
\end{subequations}

We add initial conditions to each of the models and assume periodic boundary conditions on the boundary of $\Omega$. We now state some known results for the continuous model without proof. We refer the reader to \cite{PengJCP2020} for a detailed proof of each.

\begin{theorem}[positivity-preserving property of the susceptibility $\chi$, \cite{PengJCP2020}]\label{ctchi}
For the system \eqref{eq:kerr-debyelorentz1da}, suppose $\chi(x,0)\ge0.$ Then for all $t\ge0$, $ \chi(x,t)\ge0.$
\end{theorem}

\begin{theorem}[Energy decay for Maxwell-Kerr-Debye-Lorentz system, \cite{PengJCP2020}]\label{ctenergy}
Suppose $\chi(x,0) \geq 0 $. Assume 
periodic boundary conditions on \(\partial\Omega\).  The energy for the Maxwell Kerr-Debye- Lorentz model \eqref{eq:kerr-debyelorentz1da}, defined as
\begin{equation}
   \mathcal{E}(t) =  \frac{1}{2} \int_\Omega \left(\mu_0 H^2 + \epsilon_0 \epsilon_\infty E^2 + a\epsilon_0 \chi E^2 + \frac{a\epsilon_0}{2} \chi^2 + \frac{\epsilon_0 \omega_0^2}{\omega_p^2}P^2 + \frac{\epsilon_0}{\omega_p^2} J^2\right) \, \rm dx, 
\end{equation}
satisfies the following energy identity
\begin{equation}\label{discreteenergy}
     \frac{\rm d \mathcal{E}(t)}{\rm d t} = - \int_\Omega \left(\frac{a\epsilon_0}{2\varepsilon} \left(\chi-E^2 \right)^2 + \frac{\epsilon_0}{\tau \omega^2_p} J^2 \right)\, \rm dx.
\end{equation}
\end{theorem}

\begin{theorem}[Energy decay for Maxwell-Kerr-Lorentz system, \cite{bokil2018high,PengJCP2020}]\label{ctenergylorentz}  Assume 
periodic boundary conditions on \(\partial\Omega\). the energy for the Maxwell Kerr-Lorentz model \eqref{eq:kerr-debyelorentz1db}, defined as 
\begin{equation}
\mathcal{E}_{\mathrm{Kerr}}(t) = \frac{1}{2} \int_{\Omega} \left( 
    \epsilon_0 \epsilon_{\infty} E^2
    + \frac{3\epsilon_0 a}{2} E^4
    + \mu_0 H^2
    + \frac{\epsilon_0 \omega_0^2}{\omega_p^2} P^2
    + \frac{\epsilon_0}{\omega_p^2} J^2
\right) \, \rm dx,
\end{equation}
satisfies the following energy identity
\begin{equation}
\frac{d \mathcal{E}_{\mathrm{Kerr}}(t)}{dt}
= - \frac{\epsilon_0}{\tau \omega_p^2} \int_{\Omega} J^2 \, \rm dx.
\end{equation}
\end{theorem}

\begin{remark}
Formally, as $\varepsilon \to 0$, $\chi \to E^2$ and 
$\mathcal{E} \to \mathcal{E}_{\mathrm{Kerr}}$.
\end{remark}

\section{Construction of Second Order Temporal Approximations}
In this section, we present the construction of two families of fully discrete schemes for the 1D Maxwell-Kerr-Debye-Lorentz system and its asymptotic limit, the Maxwell-Kerr-Lorentz model, presented in \eqref{eq:kerr-debyelorentz1da}, and \eqref{eq:kerr-debyelorentz1db}, respectively. Both families of schemes have high even order accuracy in space and are second order accurate in time. We refer to these schemes as having $(2,2M)$ accuracy. The first family uses staggered leapfrog discretization in time while the second family is based on a non-staggered trapezoidal time integrator. We first address the construction of $2M, M \in \mathbb{N}, M>0$ order accurate spatial discretizations. 
\subsection{Construction of High-Order Spatial Approximations}\label{subsec1.2}
In this section, we present the construction of high-order approximations of the first-derivative operator $\partial/\partial x$. Our development follows the framework in \cite{cohen2002higher, bokil2012analysis}. We introduce staggered $\ell^2$ spaces that will be used to state and verify the basic properties of these high-order derivative approximations.

Let $L>0$, $L \in \mathbb{R}$. Consider smooth functions $v$ and $u$ on $\mathbb{R}$ satisfying periodic boundary conditions, i.e., $v(x)= v(x + L)$ and $u(x)= u(x + L)$ for all $x \in \mathbb{R}$. Due to this periodic structure, it suffices to work on the computational domain $\Omega = [0,L]$. Let $\Delta x = h = L/J > 0$ for some positive integer $J$ denote a uniform mesh spacing. Define the grid points $x_j = jh$, $0 \le j \le J$, and the staggered points $x_{j+\frac12} = \left(j+\frac12\right)h$, $0 \le j \le J-1$. We then introduce two staggered grids on $\Omega$: a primal grid and a dual grid. The primal grid, $G_p$, is given by
\begin{equation}
G_p:\; x_0 < x_1 < x_2 < \cdots < x_{J-1} < x_J .
\end{equation}
The dual grid, $G_d$, on $[0,L]$ is defined as
\begin{equation}
G_d:\; x_{\frac12} < x_{\frac32} < x_{\frac52} < \cdots < x_{J-\frac12} .
\end{equation}
Let $v_\ell \approx v(x_\ell)$ and $u_{\ell+\frac12} \approx u(x_{\ell+\frac12})$ be degrees of freedom on the primal, and dual grids, respectively. On the primal grid $G_p$, we define the $\ell^2$ space of discrete grid functions $v_h := (v_\ell),\ 0 \le \ell \le J-1$ given by 
\begin{equation}
V_{0,h} := \left\{ v_h := (v_\ell),\ 0 \le \ell \le J-1,\ v_0 = v_J \ \Big|\  h\sum_{\ell=0}^{J-1} |v_{\ell}|^2 < \infty \right\},
\end{equation}
and on the dual grid $G_d$ we define the staggered $\ell^2$ space of discrete grid functions $u_h := (u_{\ell+\frac12}),\ 0 \le \ell \le J-1$ given by
\begin{equation}
V_{\frac12,h} := \left\{ u_h := (u_{\ell+\frac12}),\ 0 \le \ell \le J-1,\ 
u_{\frac12} = u_{J+\frac12} \ \Big| \  
h\sum_{\ell=0}^{J-1} |u_{\ell+\frac12}|^2 < \infty \right\}.
\end{equation}
These spaces are equipped with the discrete $\ell^2$ inner products
\begin{equation}
\langle v_h,w_h\rangle_{0,h}
:=
h\sum_{\ell=0}^{J-1} v_\ell w_\ell,
\qquad v_h,w_h \in V_{0,h},
\end{equation}
and
\begin{equation}
\langle u_h,z_h\rangle_{\frac12,h}
:=
h\sum_{\ell=0}^{J-1} u_{\ell+\frac12} z_{\ell+\frac12},
\qquad u_h,z_h \in V_{\frac12,h}.
\end{equation}
The corresponding discrete norms are then defined by
\begin{equation}
    \|v\|_{0,h} = \left( h\sum_{\ell=0}^{J-1} |v_\ell|^2 \right)^{1/2}, \quad \|u\|_{\frac{1}{2},h} = \left( h\sum_{\ell=0}^{J-1} |u_{\ell+\frac12}|^2 \right)^{1/2}.
\end{equation}
Operations on grid functions in the staggered $\ell^2$ spaces are typically performed pointwise at individual degrees of freedom. 
Next, for $p \in \mathbb{N}$, we define the discrete operators
\begin{equation}
    \mathcal{D}_{p,h}^{(2)} : V_{0,h} \rightarrow V_{\frac{1}{2},h}, \quad \left(\mathcal{D}_{p,h}^{(2)} u\right)_{\ell+\frac{1}{2}} = \frac{u_{\ell+p} - u_{\ell - p + 1}}{(2p - 1)h},
\end{equation}
\begin{equation}
   \widetilde{\mathcal{D}}_{p,h}^{(2)} : V_{\frac{1}{2},h} \rightarrow V_{0,h}, \quad \left(\widetilde{\mathcal{D}}_{p,h}^{(2)} u\right)_\ell = \frac{u_{\ell+p - \frac{1}{2}} - u_{\ell -p + \frac{1}{2}}}{(2p - 1)h}. 
\end{equation}
These are second-order discrete approximations to $\partial / \partial x$ with step-size $(2p - 1)h$. 
see \citep[p. 37]{cohen2002higher}. 
Similar to \cite{anne2000construction,cohen2002higher}, we construct finite difference approximations of order $2M$, for $M \in \mathbb{N}$, to the first derivative operator $\partial/\partial x$. Consider the operators
\begin{equation}\label{eq: 15}
    \mathcal{D}_{h}^{(2M)} : V_{0,h} \rightarrow V_{\frac{1}{2},h}; \qquad \widetilde{\mathcal{D}}_{h}^{(2M)} : V_{\frac{1}{2},h} \rightarrow V_{0,h},
\end{equation}
defined as linear combinations of second-order approximations of $\partial / \partial x$ computed with different space steps in the form
\begin{equation}\label{eq: 16}
    \mathcal{D}_{h}^{(2M)} = \sum_{p=1}^{M} \lambda_{2p-1}^{2M} \mathcal{D}_{p,h}^{(2)}, \qquad \widetilde{\mathcal{D}}_{h}^{(2M)} = \sum_{p=1}^{M} \lambda_{2p-1}^{2M} \widetilde{\mathcal{D}}_{p,h}^{(2)}.
\end{equation}

\begin{theorem}
For any $M \in \mathbb{N}$, the coefficients $\lambda_{2p-1}^{2M}$ of the linear combinations \eqref{eq: 16} are given by the explicit formula
\begin{equation}
    \lambda_{2p-1}^{2M} = \frac{2(-1)^{p-1}((2M - 1)!!)^2}{(2M + 2p - 2)!!(2M - 2p)!!(2p - 1)},
\end{equation}
where $1 \leq p \leq M$, and the double factorial is defined as
\begin{equation}
    n!! =
\begin{cases}
n \cdot (n - 2) \cdot (n - 4) \ldots 5 \cdot 3 \cdot 1, & n > 0, \text{ odd} \\
n \cdot (n - 2) \cdot (n - 4) \ldots 6 \cdot 4 \cdot 2, & n > 0, \text{ even} \\
1, & n = -1, 0.
\end{cases}
\end{equation}

\end{theorem}

\noindent\textit{Proof.} See \cite{bokil2012analysis,cohen2002higher}. 

\subsection{Modified Exponential Time Integrator}
\label{sec:modified-exponential}
In order to construct fully discrete FDTD schemes that are structure preserving; preserve energy stability, the positivity-preserving (PP) property and the asymptotic-preserving (AP) property, i.e., maintaining the asymptotic limit in the discrete scheme as $\epsilon \rightarrow 0$, we consider a modified exponential time discretization of
\eqref{kerrdebye1d} introduced in \cite{PengJCP2020} and inspired by \cite{crestetto2018particle}. As discussed in \cite{PengJCP2020}, to derive a second order in time approximation that satisfies AP, PP and energy stability properties, we will first derive a modified ODE and then discretize it appropriately. 

Let $\Delta t > 0$ be a given time step. We discretize the time interval $[0,T]$ using time grid points $t^n = n\Delta t$, for $n\in \mathbb{N}, 0\leq n\leq N$, with $T = N\Delta t$. We also define staggered time grid points $t^{n+\frac{1}{2}} = t^n+\frac{\Delta t}{2}$.
The derivation of the discrete approximation of \eqref{kerrdebye1d} is based on the estimate
\begin{align}
\left|
\frac{\Delta t}{\varepsilon}
e^{-\frac{\Delta t}{2\varepsilon}}
-
\left(1-e^{-\frac{\Delta t}{\varepsilon}}\right)
\right|
&=
\left|
\frac{\Delta t}{\varepsilon}
e^{-\frac{\Delta t}{2\varepsilon}}
-
\int_0^{\Delta t}
\frac{1}{\varepsilon}e^{-\frac{\tau}{\varepsilon}}\,d\tau
\right| =
\mathcal O\left(\frac{\Delta t^3}{\varepsilon^3}\right),
\label{eq:exp-coefficient-relation}
\end{align}
which follows by setting $r=\Delta t/\varepsilon$ and expanding the exponential terms in the expression on the left-hand side of \eqref{eq:exp-coefficient-relation} in a Taylor series about $r=0$. As $r\to\infty$, which corresponds to $\varepsilon\ll\Delta t$, one has
\[
\frac{\Delta t}{\varepsilon}e^{-\frac{\Delta t}{2\varepsilon}}
-\left(1-e^{-\frac{\Delta t}{\varepsilon}}\right)
\longrightarrow -1,
\]
so, the difference between the terms remains bounded.

Suppressing the spatial dependence of $\chi$ and $E$ for simplicity, and using the exact solution of \eqref{kerrdebye1d} over $[t^n,t^{n+1}]$ we get 
\begin{align}
\chi(t^{n+1})
&=
e^{-\frac{\Delta t}{\varepsilon}}\chi(t^n)
+
\int_0^{\Delta t}
\frac{1}{\varepsilon}
e^{\frac{\tau-\Delta t}{\varepsilon}}
E(t^n+\tau)^2\,d\tau.
\end{align}
Using midpoint quadrature to approximate the integral, we obtain
\begin{align}
\chi(t^{n+1})
&=
e^{-\frac{\Delta t}{\varepsilon}}\chi(t^n)
+
\frac{\Delta t}{\varepsilon}
e^{-\frac{\Delta t}{2\varepsilon}}
\left(E(t^{n+\frac12})\right)^2
+
\mathcal O(\Delta t^3).
\label{eq:modified-exp-midpoint}
\end{align}
Subtracting $\chi(t^n)$ from both sides of \eqref{eq:modified-exp-midpoint}, dividing by $\Delta t$ on both sides and then expanding about $t^{n+\frac12}$, yields the modified equation
\begin{align}
\left(
\frac{1}{2}
+
\frac{1}{2}e^{-\frac{\Delta t}{\varepsilon}}
\right)
\partial_t\chi(t^{n+\frac12})
&=
\frac{e^{-\frac{\Delta t}{\varepsilon}}-1}{\Delta t}
\chi(t^{n+\frac12})+
\frac{1}{\varepsilon}
e^{-\frac{\Delta t}{2\varepsilon}}
\left(E(t^{n+\frac12})\right)^2
+
\mathcal O(\Delta t^2).
\label{eq:modified-relaxation}
\end{align}
Applying centered differencing to each term in \eqref{eq:modified-relaxation} 
and rearranging the resulting equation gives
\begin{align}
\left(
\frac{1}{2}
+
\frac{1}{2}e^{-\frac{\Delta t}{\varepsilon}}
+
\frac{1-e^{-\frac{\Delta t}{\varepsilon}}}{2}
\right)
\chi(t^{n+1})
&=
e^{-\frac{\Delta t}{\varepsilon}}\chi(t^n)+
\frac{\Delta t}{2\varepsilon}
e^{-\frac{\Delta t}{2\varepsilon}}
E(t^{n+1})^2+
\frac{\Delta t}{2\varepsilon}
e^{-\frac{\Delta t}{2\varepsilon}}
E(t^n)^2
+
\mathcal O(\Delta t^3).
\label{eq:modified-exp-before-correction}
\end{align}
Using \eqref{eq:exp-coefficient-relation} on the third term on the left hand side and the second term on the right hand side in  \eqref{eq:modified-exp-before-correction}, we obtain 
\begin{align}\label{eq:modified-exp-relation}
\left(
\frac{1}{2}
+
\frac{1}{2}e^{-\frac{\Delta t}{\varepsilon}}
+
\frac{\Delta t}{2\varepsilon}
e^{-\frac{\Delta t}{2\varepsilon}}
\right)
\chi(t^{n+1})
&=
e^{-\frac{\Delta t}{\varepsilon}}\chi(t^n)+
\frac{1}{2}
\left(
1-e^{-\frac{\Delta t}{\varepsilon}}
\right)
E(t^{n+1})^2 \\
\notag 
&\hspace{8em}+
\frac{\Delta t}{2\varepsilon}
e^{-\frac{\Delta t}{2\varepsilon}}
E(t^n)^2
+
\mathcal O\left(\frac{\Delta t^3}{\varepsilon^3}\right).
\end{align}
In the next two sections, we will use \eqref{eq:modified-exp-relation} to construct second-order time discretizations for the equation \eqref{kerrdebye1d}. This is presented in the next two sections in equations \eqref{eq:chidiscreteLF} for the leapfrog schemes and in \eqref{eq:chidiscreteb} for the trapezoidal schemes. The AP and PP properties will hold for both types of discretizations due to the positivity of the exponential functions and their rapid decay property. 

The explicit $\varepsilon$-dependence in
\eqref{eq:exp-coefficient-relation} and \eqref{eq:modified-exp-relation} shows that the second-order consistency argument is not uniform in the relaxation parameter.  For fixed $\varepsilon>0$, the remainder vanishes as $\Delta t\to0$. However, when $\varepsilon=O(\Delta t)$, $\Delta t^3/\varepsilon^3 = \mathcal O(1)$. Thus, a loss of the designed temporal accuracy may occur in the transition regime $\Delta t/\varepsilon=O(1)$. For fixed $\varepsilon$, further refinement eventually yields $\Delta t\ll\varepsilon$, where the classical second-order behavior is recovered. 
 
\subsection{Fully Discrete $(2, 2M)$ Leapfrog FDTD Methods}\label{22MLF}
In this section, we derive fully discrete FDTD methods that use a second-order leapfrog time integrator for Maxwell's equations and couple it with second-order time discretizations for the material polarization equations; the resulting schemes are second-order in time and order $2M$ in space for $M\in\mathbb{N}$, and we refer to them as $(2,2M)$ leapfrog FDTD methods.

Let $N \in \mathbb{N}$. Given time step $\Delta t > 0$, we define primal and dual time grids \(t^n = n\Delta t\) for \(n=0,\dots,N\) with \(t^N = T\), and \(t^{n+\frac12} = \left(n+\tfrac12\right)\Delta t\) for \(n=0,\dots,N-1\). The discrete electric field is computed at primal space time nodes, and the magnetic field at dual space time nodes. For a field \(F(x,t)\), we write \(F_j^n \approx F(x_j,t^n)\) on the primal mesh and \(F_{j+\frac12}^{\,n+\frac12} \approx F\!\left(x_{j+\frac12},t^{n+\frac12}\right)\) on the dual mesh. The notation $F_h^n \in V_{0,h}$ indicates a grid function with values $F_h^n := (F_j^n), n \geq 0, 0 \leq j \leq J$. Similarly, the notation $F_h^{n+\frac{1}{2}} \in V_{\frac{1}{2},h}$ indicates a grid function with values $F_h^{n+\frac{1}{2}} := (F_{j+\frac{1}{2}}^{n+\frac{1}{2}}), n \geq 0, 0 \leq j \leq J-1$. 

For a grid function $F_h^n$, we define its time average at time $t^{n+\frac{1}{2}}$ by
\begin{equation}
\overline{F}_h^{n+\frac{1}{2}} := \frac{1}{2} \left( F_h^n + F_h^{n+1} \right). 
\end{equation}
and the discrete time derivative by
\begin{equation}
\delta_t F_h^{n+\frac{1}{2}} := \frac{F_h^{n+1} - F_h^n}{\Delta t}. 
\end{equation}
We discretize \eqref{faraday}-\eqref{ampere} with the staggered \(2M\) spatial scheme and a modified second-order leapfrog step in time. The ODEs \eqref{density}, \eqref{lorentz} and \eqref{constlaw} are discretized using the standard trapezoidal time integrator, while a second order approximation of the modified exponential time integrator derived in \eqref{eq:modified-exp-relation} is proposed for the stiff part of the ADE \eqref{kerrdebye1d}. The fully discrete \((2,2M)\) leapfrog FDTD schemes for the Maxwell-Kerr-Debye-Lorentz model are as follows. 
 
Assume $D_h^{n}, E_{h}^n, \chi_{h}^n, P_{h}^n , J_{h}^n \in V_{0,h}$ and $H_{h}^{n-\frac{1}{2}} \in V_{\frac{1}{2},h}$ are given. Then $\forall \, n \in \mathbb{N}$, with $0 < n \leq N$, find $D_h^{n+1}, E_{h}^{n+1}, \chi_{h}^{n+1}, P_{h}^{n+1}, J_{h}^{n+1} \in V_{0,h}$ and $H_{h}^{n+\frac{1}{2}} \in V_{\frac{1}{2},h}$ so that $\forall \, 0 \leq j \leq J-1$ we have

\begin{subequations} \label{leapfrog}
    \begin{align}
       \mu_0 \frac{H_{j+\frac{1}{2}}^{n+\frac{1}{2}} - H_{j+\frac{1}{2}}^{n-\frac{1}{2}}}{\Delta t} &= \left( \mathcal{D}_{h}^{(2M)} E_h^n \right)_{j+\frac{1}{2}}, \label{eq:faradaydiscreteLF}  \\
       \frac{D_j^{n+1} - D_j^n}{\Delta t} &= \left( \widetilde{\mathcal{D}}_{h}^{(2M)} H_h^{n+\frac{1}{2}} \right)_j, \label{eq:amperediscreteLF} \\
       \left( \frac{1}{2} + \frac{1}{2} e^{-\frac{\Delta t}{\varepsilon}} + \frac{\Delta t}{2\varepsilon} e^{-\frac{\Delta t}{2\varepsilon}}\right) \chi_j^{n+1} &= e^{-\frac{\Delta t}{\varepsilon}} \chi_j^n+ \frac{1}{2} \left( 1 - e^{-\frac{\Delta t}{\varepsilon}} \right) (E_j^{n+1})^2 + \frac{\Delta t}{2\varepsilon} e^{-\frac{\Delta t}{2\varepsilon}}
        ( E_j^n)^2 ,\label{eq:chidiscreteLF}\\
        \frac{P_j^{n+1} - P_j^n}{\Delta t} &= \frac{1}{2} \left( J_j^{n+1} + J_j^n \right),\label{eq:densitydiscreteLF}\\
        \frac{J_j^{n+1} - J_j^n}{\Delta t} & = -\frac{1}{2\tau} \left( J_j^{n+1} + J_j^n \right)- \frac{\omega_0^2}{2} \left(P_j^{n+1} + P_j^n \right)+ \frac{\omega_p^2}{2} \left( E_j^{n+1} + E_j^n \right),\label{eq:lorentzdiscreteLF}\\
        D^{n+1}_j &= \epsilon_0\left(\left(\epsilon_\infty +a\chi^{n+1}_j\right)E^{n+1}_j + P^{n+1}_j\right).\label{algcontlawLF}
    \end{align}
\end{subequations}
\begin{remark}
As $\varepsilon \to 0$, \eqref{eq:chidiscreteLF} becomes $\chi^{n+1}_j = (E^{n+1}_j)^2$. Hence, the correct asymptotic limit is captured. 
\end{remark}

The $(2, 2M)$ leapfrog FDTD methods satisfy the following discrete analogues of Theorem~\ref{ctchi}-\ref{ctenergylorentz}. 
\begin{theorem}[Discrete positivity-preserving property of the susceptibility $\chi^n_h$]\label{eq:positivitydiscrete}
Given $n \in \mathbb{N}, n \geq 0$, and a discrete susceptibility grid function $\chi_h^n \in V_{0,h}$. If $\chi_{j}^n \geq 0, \ 0\leq j \leq J$, then $\chi_{j}^{n+1} \geq 0 \ 0\leq j \leq J $.
\end{theorem}

\begin{proof} 
From \eqref{eq:chidiscreteLF}, we have
\begin{equation*}
\left( 
\frac{1}{2} + \frac{1}{2} e^{-\frac{\Delta t}{\varepsilon}} 
+ \frac{\Delta t}{2\varepsilon} e^{-\frac{\Delta t}{2\varepsilon}}
\right) \chi_{j}^{n+1}
= e^{-\frac{\Delta t}{\varepsilon}} \chi_{j}^n
+ \frac{1}{2} \left( 1 - e^{-\frac{\Delta t}{\varepsilon}} \right) ( E_{j}^{n+1})^2
+ \frac{\Delta t}{2\varepsilon} e^{-\frac{\Delta t}{2\varepsilon}} ( E_{j}^n)^2,
\end{equation*}
where $e^{-\frac{\Delta t}{\varepsilon}} > 0$, 
$\tfrac{1}{2}\left( 1 - e^{-\frac{\Delta t}{\varepsilon}} \right) > 0$, 
and $\tfrac{\Delta t}{2\varepsilon} e^{-\frac{\Delta t}{2\varepsilon}} > 0$. Since $\chi_{j}^n \geq 0$, we have $\chi_{j}^{n+1} \geq 0$.
\end{proof}
The following summation by parts lemma and the positive definiteness of the operator $\mathcal{A}_h$ will be used to establish discrete energy stability of the proposed scheme.
\begin{lemma}[Summation by Parts]\label{sumbypart} Consider the grid functions $U_h \in V_{0,h}$ and $W_h \in V_{\frac{1}{2},h}$. Then, the following summation by parts (SBP) identity holds
\begin{equation} 
\left\langle \mathcal{D}_{h}^{(2M)} U_h, W_h \right\rangle_{\frac{1}{2},h} = - \left\langle \widetilde{\mathcal{D}}_{h}^{(2M)} W_h, U_h \right\rangle_{0,h}.
\end{equation}
\end{lemma}

\begin{proof}
    See \cite{cohen2002higher}.
\end{proof}

\begin{theorem}[Positive Definiteness]\label{positiveA} \textit{Consider the operator $\mathcal{A}_h : V_{0,h} \rightarrow V_{0,h}$ defined as}
\begin{equation}
\mathcal{A}_h := I_h + \gamma^2 \widetilde{\mathcal{D}}_{h}^{(2M)} \circ \mathcal{D}_{h}^{(2M)}, 
\end{equation}
\textit{where $I_h : V_{0,h} \rightarrow V_{0,h}$ is the identity operator, the $2M, M \in \mathbb{N}$ order spatial discretizations are as defined in \eqref{eq: 15}, \eqref{eq: 16} and $\gamma = \frac{c_\infty \Delta t}{2}$ is a constant independent of the mesh step $h$, but depends on $\Delta t$ and the speed $c_\infty$. The operator $\mathcal{A}_h$ is positive definite if and only if $\gamma$ satisfies the condition}
\begin{align}\label{eq: cfl}
\displaystyle \frac{2 \gamma}{h} < \left[\sum_{\ell = 1}^M \frac{[(2\ell - 3)!!]^2}{(2\ell - 1)!}\right]^{-1}.
\end{align}
\end{theorem}

\begin{proof}
    See \cite{bokil2018high}.
\end{proof} 
We now prove the main energy stability result for the leapfrog schemes. To do so, we need an adjustment in the handling of the natural numerical constitutive law \eqref{algcontlawLF}, which does not yield a provable energy relation or guarantee energy stability. To address this, we introduce an alternative energy-based numerical constitutive law, following \cite{PengJCP2020}:
\begin{align}\label{eq:energybasedconstlaw}
    &\left(D^{n+1}_j - D^n_j\right)\left(E^{n+1}_j + E^n_j\right) - \frac{a\epsilon_0 \Delta t}{\varepsilon} \left(\frac{\chi^{n+1}_j +\chi^n_j}{2} - \frac{(E^{n+1}_j)^2 + (E^n_j)^2}{2}\right)^2 \notag\\
    &= \epsilon_0\epsilon_\infty\left((E_j^{n+1})^2-(E_j^n)^2\right)
    + \frac{\epsilon_0a}{2}\left((\chi_j^{n+1})^2-(\chi_j^n)^2\right) \notag\\
    &+ \epsilon_0a\left(\chi_j^{n+1}(E_j^{n+1})^2
    - \chi_j^n (E_j^n)^2\right) + \epsilon_0 \left(P^{n+1}_j - P^n_j\right)\left(E^{n+1}_j + E^n_j\right).
\end{align} 
\begin{remark} As discussed in \cite{PengJCP2020}, the above expression is a second-order in time approximation for the term  $(D^{n+1}_j - D^n_j)$ multipled by $(E^{n+1}_j + E^n_j)$.
\end{remark}
\begin{theorem}[Discrete energy decay]\label{energystabilityLF}
Let $H_h^{n+\frac{1}{2}} \in V_{\frac{1}{2},h}$, $D^n_h, E_h^n, \chi_h^n, P^n_h, J^n_h \in V_{0,h}$, and assume $\chi^n\geq 0$. Under the CFL condition
\begin{equation}\label{eq: 28}
\frac{c_\infty \Delta t}{h} < \left[ \sum_{\ell=1}^{M} \frac{[(2\ell - 3)!!]^2}{(2\ell - 1)!} \right]^{-1}, 
\end{equation}
where $c_\infty^2 = \frac{1}{\epsilon_0 \epsilon_\infty \mu_0}$, \textit{the $(2,2M)$ order leapfrog FDTD scheme satisfies the discrete energy identity}
\begin{equation}\label{eq:discretenergyrel}
\delta_t \mathcal{E}_{h, LF}^{n+\frac{1}{2}} = -\frac{a\epsilon_0}{2\varepsilon}\Biggl\|\frac{\chi_h^{n+1} + \chi_h^n}{2} -\frac{(E_h^{n+1})^2 + (E_h^n)^2}{2}\Biggr\|^2_{0,h} - \frac{\epsilon_0}{\omega_p^2 \tau} \Biggl\|\frac{J^{n+1}_h + J^n_h}{2}\Biggr\|^2_{0,h},
\end{equation}
where the discrete energy is defined for all $n \in \mathbb{N}$ as
\begin{align}\label{eq:discretenergyLF}
\mathcal{E}_{h, LF}^n := & \frac{1}{2}\Biggl(\mu_0 \| \overline{H_h}^{n+\frac{1}{2}} \|_{\frac{1}{2},h}^2 + \epsilon_0 \epsilon_\infty \left\langle E_h^n, \mathcal{A}_hE_h^n \right\rangle_{0,h} + a\epsilon_0\left\langle\chi_h^n,(E_h^n)^2\right\rangle_{0,h} + \frac{a\epsilon_0}{2}\bigl\|\chi_h^n\bigr\|^2_{0,h} \notag \\ 
& \hspace{9cm} + \frac{\epsilon_0 \omega_0^2}{\omega_p^2}\|P^n_h\|^2_{0,h} + \frac{\epsilon_0}{\omega_p^2}\|J^n_h\|^2_{0,h}\Biggr).
\end{align}
\end{theorem}

\begin{proof}
We take the average of the equation \eqref{eq:faradaydiscreteLF} at the time step $n$ and $n+1$, to get
\begin{equation}\label{eq: 31}
\mu_0 \frac{H_{j+\frac{1}{2}}^{n+\frac{3}{2}} - H_{j+\frac{1}{2}}^{n-\frac{1}{2}}}{2\Delta t}
= \frac{1}{2} \left( \left(\mathcal{D}_{h}^{(2M)} E_h^{n+1}\right)_{j+\frac{1}{2}} + \left(\mathcal{D}_{h}^{(2M)} E_h^n\right)_{j+\frac{1}{2}} \right).
\end{equation}
Multiplying the equation \eqref{eq: 31} by $H_{j+\frac{1}{2}}^{n+\frac{1}{2}}$, we get
\begin{equation}
\mu_0 \frac{H_{j+\frac{1}{2}}^{n+\frac{3}{2}} - H_{j+\frac{1}{2}}^{n-\frac{1}{2}}}{2\Delta t} \cdot H_{j+\frac{1}{2}}^{n+\frac{1}{2}} =
\frac{1}{2} \left( \left(\mathcal{D}_{h}^{(2M)} E_h^{n+1}\right)_{j+\frac{1}{2}} + \left(\mathcal{D}_{h}^{(2M)} E_h^n\right)_{j+\frac{1}{2}} \right) \cdot H_{j+\frac{1}{2}}^{n+\frac{1}{2}}. 
\end{equation}
Summing over all mesh points on the dual grid, multiplying by $h$ and simplifying yields 
\begin{equation}\label{eq: 33}
     \frac{1}{2\Delta t} \left(\mu_0\left\langle H^{n+\frac{3}{2}}_h , H^{n+\frac{1}{2}}_h\right\rangle_{\frac{1}{2}, h} - \mu_0\left\langle H^{n+\frac{1}{2}}_h , H^{n-\frac{1}{2}}_h\right\rangle_{\frac{1}{2}, h}\right)  = \frac{1}{2}\left\langle \left(\mathcal{D}_{h}^{(2M)}E_h^{n+1}
    +\mathcal{D}_{h}^{(2M)}E_h^{n}\right),\, H_h^{n+\frac12} \right\rangle_{\frac12,h}.
\end{equation}
Multiplying the equation \eqref{eq:amperediscreteLF} by $\overline{E_j}^{n+\frac{1}{2}}$, we obtain 
\begin{align}
 \frac{1}{2\Delta t}(D^{n+1}_j - D^n_j)(E^{n+1}_j + E^n_j)
= \left( \widetilde{\mathcal{D}}_{h}^{(2M)} H_{h} \right)^{n+\frac{1}{2}}_j \cdot \overline{E_j}^{n+\frac{1}{2}}.
\end{align}
Summing over all mesh points on the primal grid, multiplying by $h$ and simplifying gives 
\begin{align}\label{eq: 35}
 \frac{1}{2\Delta t}\left\langle D_h^{n+1}-D_h^n, E_h^{n+1}+E_h^n \right\rangle_{0,h}
= \frac{1}{2} \left\langle \widetilde{\mathcal{D}}_{h}^{(2M)} H_h^{n+\frac{1}{2}}, \, E_h^{n+1}+E_h^n \right\rangle_{0,h}.
\end{align}
Multiply \eqref{eq:densitydiscreteLF} by $P_h^{n+1} + P_h^n$, summing over all mesh points on the primal grid and multiplying by $h$, yields
\begin{equation}\label{eq:desnitydiscreteaLF}
    \|P_h^{n+1}\|^2_{0,h} - \|P_h^n\|^2_{0,h} = \frac{\Delta t}{2} \left\langle J_h^{n+1} + J_h^n, P_h^{n+1} + P_h^n \right\rangle_{0,h}.
\end{equation}
Multiply \eqref{eq:lorentzdiscreteLF} by $\left(J_h^{n+1} + J_h^n\right)/\omega_p^2$, sum over all mesh points on the primal grid, multiply by $h$, and rewrite the result by applying \eqref{eq:densitydiscreteLF} and \eqref{eq:desnitydiscreteaLF} to obtain 
\begin{align}\label{lorentzfin}
    \frac{1}{\omega_p^2}\left(\|J_h^{n+1}\|^2_{0,h} - \|J_h^n\|^2_{0,h}\right) + \frac{\Delta t}{2\omega_p^2 \tau}\|J_h^{n+1} + J_h^n\|^2_{0,h} &+ \frac{\omega_{0}^2}{\omega_p^2}\left(\|P_h^{n+1}\|^2_{0,h} - \|P_h^n\|^2_{0,h}\right) 
    \notag\\
    &\qquad \qquad = \left\langle E_h^{n+1} + E_h^n, P_h^{n+1} - P_h^n \right\rangle_{0,h}. 
\end{align}
Next, summing \(\eqref{eq:energybasedconstlaw}\) over all mesh points on the primal grid and multiplying by $h$, we obtain
\begin{align}\label{eq:constnorm}
&\left\langle D_h^{n+1}-D_h^n, E_h^{n+1}+E_h^n \right\rangle_{0,h}
=\frac{a\epsilon_0 \Delta t}{\varepsilon}
\left\|
\frac{\chi_h^{n+1}+\chi_h^n}{2}
-
\frac{(E_h^{n+1})^2+(E_h^n)^2}{2}
\right\|_{0,h}^2
\notag\\
&\qquad +
\epsilon_0\epsilon_\infty\left(\|E_h^{n+1}\|_{0,h}^2-\|E_h^n\|_{0,h}^2\right)
+ \frac{a\epsilon_0}{2}\left(\|\chi_h^{n+1}\|_{0,h}^2-\|\chi_h^n\|_{0,h}^2\right)
\notag\\
&\qquad\quad
+ a\epsilon_0\left(
\left\langle \chi_h^{n+1}, (E_h^{n+1})^2 \right\rangle_{0,h}
-
\left\langle \chi_h^n, (E_h^n)^2 \right\rangle_{0,h}
\right)
+ \epsilon_0 \left\langle P_h^{n+1}-P_h^n, E_h^{n+1}+E_h^n \right\rangle_{0,h}.
\end{align}
Combining equations \eqref{eq: 33}, \eqref{eq: 35}, \eqref{lorentzfin}, and \eqref{eq:constnorm} and using the summation by parts property in Lemma \ref{sumbypart}, we obtain 
\begin{equation} \label{eq: finalformLF}
\begin{aligned}
&\frac{1}{2\Delta t}\left(\mu_0\Bigg\langle H^{n+\frac{3}{2}}_h , H^{n+\frac{1}{2}}_h\right\rangle_{\frac{1}{2}, h} - \mu_0\left\langle H^{n+\frac{1}{2}}_h , H^{n-\frac{1}{2}}_h\right\rangle_{\frac{1}{2}, h} 
 + \epsilon_0 \epsilon_\infty \left( \|E_h^{n+1}\|_{0,h}^2 - \|E_h^n\|_{0,h}^2 \right) \\
   &\quad + a\epsilon_0 \left( 
        \left\langle \chi_h^{n+1}, (E_h^{n+1})^2 \right\rangle_{0,h} 
        - \left\langle \chi_h^n, (E_h^n)^2 \right\rangle_{0,h} 
    \right)
 + \frac{a\epsilon_0}{2} \left( 
        \|\chi_h^{n+1}\|_{0,h}^2 - \|\chi_h^n\|_{0,h}^2 
    \right) \\
& \qquad \qquad \qquad + \frac{\epsilon_0}{\omega_p^2}\Big(\|J_h^{n+1}\|_{0,h}^2 - \|J_h^n\|_{0,h}^2\Big)
    + \frac{\epsilon_0 \omega_0^2}{\omega_p^2}\Big(\|P_h^{n+1}\|_{0,h}^2 - \|P_h^n\|^2_{0,h}\Big)\Bigg)
    \\
    & = - \frac{a\epsilon_0}{2\varepsilon} \left\| 
        \frac{\chi_h^{n+1} + \chi_h^n}{2} 
        - \frac{(E_h^{n+1})^2 + (E_h^n)^2}{2} 
    \right\|_{0,h}^2 - \frac{\epsilon_0}{\omega_p^2\tau}\Biggl\|\frac{J_h^{n+1} + J_h^n}{2}\Biggr\|^2_{0,h}.
\end{aligned}
\end{equation}
We define a discrete quantity, $\widetilde{\mathcal{E}}_{h, LF}^n$, for $n \in \mathbb{N}, n\geq 0$, which is an alternate version of the discrete energy defined in \eqref{eq:discretenergyLF}, as, 
\begin{align}\label{eq:energyidentityLF}
\widetilde{\mathcal{E}}_{h,LF}^n &:= \frac{1}{2}\Biggl(\mu_0 \left\langle H_h^{n-\frac{1}{2}}, H_h^{n+\frac{1}{2}} \right\rangle_{\frac{1}{2},h}
+ \epsilon_0 \epsilon_\infty \| E_h^n \|_{0,h}^2 + a\epsilon_0\left\langle \chi_h^n, (E_h^n)^2 \right\rangle_{0,h} 
+ \frac{a\epsilon_0}{2}\|\chi_h^n\|_{0,h}^2 \notag \\
&\hspace{9cm} + \frac{\epsilon_0 \omega_0^2}{\omega_p^2}\|P^n_h\|^2_{0,h}+ \frac{\epsilon_0}{\omega_p^2}\|J^n_h\|^2_{0,h}\Biggr).
\end{align}
Using the definition \eqref{eq:energyidentityLF}, the equation \(\eqref{eq: finalformLF}\) can be simplified as \begin{align}\label{eq:discreteenergy}
\frac{\widetilde{\mathcal{E}}_{h,LF}^{n+1}
-\widetilde{\mathcal{E}}_{h,LF}^{n}}{\Delta t}
&=
-\frac{a\epsilon_0}{2\varepsilon}
\Biggl\|
\frac{\chi_h^{n+1}+\chi_h^n}{2}
-
\frac{(E_h^{n+1})^2+(E_h^n)^2}{2}
\Biggr\|_{0,h}^2
-
\frac{\epsilon_0}{\omega_p^2\tau}
\Biggl\|
\frac{J_h^{n+1}+J_h^n}{2}
\Biggr\|_{0,h}^2
\le 0.
\end{align}
\noindent
which immediately shows \(\widetilde{\mathcal{E}}_{h, LF}^{n+1}\le \widetilde{\mathcal{E}}_{h, LF}^n\).
It remains to show $\widetilde{\mathcal{E}}_{h,LF}^n$ is nonnegative and equivalent to the definition \eqref{eq:discretenergyLF}. 

By applying the parallelogram law and \eqref{eq:faradaydiscreteLF} we have
\begin{align}\label{eq: 39}
\left\langle H_h^{n-\frac{1}{2}}, H_h^{n+\frac{1}{2}} \right\rangle_{\frac{1}{2},h}
&= \frac{1}{4} \left[ \| H_h^{n+\frac{1}{2}} + H_h^{n-\frac{1}{2}} \|_{\frac{1}{2},h}^2 - \| H_h^{n+\frac{1}{2}} - H_h^{n-\frac{1}{2}} \|_{\frac{1}{2},h}^2 \right] \notag \\
&= \| \overline{H_h}^{n} \|_{\frac{1}{2},h}^2 - \frac{\Delta t^2}{4 \mu_0^2} \| \mathcal{D}_{h}^{(2M)} E_h^n \|_{\frac{1}{2},h}^2. 
\end{align}
Using equation \eqref{eq: 39}, and the summation by parts property in Lemma \ref{sumbypart}, one then obtains that 
\begin{align}\label{eq: 40}
\mu_0 \langle H_h^{n-\frac{1}{2}}, H_h^{n+\frac{1}{2}} \rangle_{\frac{1}{2},h} 
&+ \epsilon_0 \epsilon_\infty \| E_h^n \|_{0,h}^2 
= \mu_0 \| \overline{H_h}^{n} \|_{\frac{1}{2},h}^2 
- \frac{\Delta t^2}{4 \mu_0} \| \mathcal{D}_{h}^{(2M)} E_h^n \|_{\frac{1}{2},h}^2 
+ \epsilon_0 \epsilon_\infty \| E_h^n \|_{0,h}^2 \notag \\
&= \mu_0 \| \overline{H_h}^{n} \|_{\frac{1}{2},h}^2 
+ \epsilon_0 \epsilon_\infty \Bigg(\left\langle E_h^n, E_h^n \right\rangle_{0,h} 
- \frac{\Delta t^2}{4 \epsilon_0 \epsilon_\infty \mu_0} \left\langle \mathcal{D}_{h}^{(2M)} E_h^n, \mathcal{D}_{h}^{(2M)} E_h^n \right\rangle_{\frac{1}{2},h}\Bigg) \notag \\
&= \mu_0 \| \overline{H_h}^{n} \|_{\frac{1}{2},h}^2 
+ \epsilon_0 \epsilon_\infty\Bigg( \left\langle E_h^n, E_h^n \right\rangle_{0,h} 
+ \frac{\Delta t^2}{4 \epsilon_0 \epsilon_\infty \mu_0} \left\langle  E_h^n, \widetilde{\mathcal{D}}_{h}^{(2M)}\mathcal{D}_{h}^{(2M)} E_h^n \right\rangle_{0,h} \Bigg)\notag \\
&= \mu_0 \| \overline{H_h}^{n} \|_{\frac{1}{2},h}^2 
+ \epsilon_0 \epsilon_\infty \left\langle E_h^n, \left( I + \frac{c_\infty^2 \Delta t^2}{4} \widetilde{\mathcal{D}}_{h}^{(2M)} \circ \mathcal{D}_{h}^{(2M)} \right) E_h^n \right\rangle_{0,h}\notag\\
&= \mu_0 \| \overline{H_h}^{n+\frac{1}{2}} \|_{\frac{1}{2},h}^2 
+ \epsilon_0 \epsilon_\infty \left\langle E_h^n,\mathcal{A}_h E_h^n \right\rangle_{0,h}.
\end{align} 
Substituting equation \eqref{eq: 40} into definition \eqref{eq:energyidentityLF}, equation \eqref{eq:discreteenergy} becomes 
\begin{align}\label{eq:discreteenergy2}
\frac{\mathcal{E}_{h,LF}^{n+1}
-\mathcal{E}_{h,LF}^{n}}{\Delta t}
&=
-\frac{a\epsilon_0}{2\varepsilon}
\Biggl\|
\frac{\chi_h^{n+1}+\chi_h^n}{2}
-
\frac{(E_h^{n+1})^2+(E_h^n)^2}{2}
\Biggr\|_{0,h}^2
-
\frac{\epsilon_0}{\omega_p^2\tau}
\Biggl\|
\frac{J_h^{n+1}+J_h^n}{2}
\Biggr\|_{0,h}^2
\le 0, 
\end{align}
which is the discrete energy identity \eqref{eq:discretenergyrel} with the discrete quantity $\mathcal{E}_{h, LF}^n$ as defined in \eqref{eq:discretenergyLF}. It only remains to show that the discrete quantity $\mathcal{E}_{h, LF}^n$ does define an energy. Under the CFL condition \eqref{eq: 28}, Theorem \ref{positiveA} provides the positive definiteness of the operator $\mathcal{A}_h$, while Theorem \ref{eq:positivitydiscrete} provides the positivity of the susceptibility  $\chi_{j}^{n+1}$ under the assumption of positivity of $\chi_{j}^{n}$. Thus, the definition \eqref{eq:discretenergyLF} provides a discrete energy and the identity \eqref{eq:discretenergyrel} (also \eqref{eq:discreteenergy2}) is a discrete energy identity for the fully discrete $(2,2M)$ leapfrog FDTD scheme, compatible with the continuous energy definition and energy identity given in Theorem \ref{ctenergy}.  
\end{proof}
\begin{remark} 
As $\varepsilon \to 0$, \eqref{eq:chidiscreteLF} becomes $\chi^{n+1}_h = (E^{n+1}_h)^2$. Thus, the energy estimate \eqref{eq:discretenergyrel} becomes 
\begin{align}\label{eq:discreteenergy2Kerr}
\frac{\mathcal{E}_{h,LF,\mathrm{Kerr}}^{n+1}
-\mathcal{E}_{h,LF,\mathrm{Kerr}}^{n}}{\Delta t}
&=
-
\frac{\epsilon_0}{\omega_p^2\tau}
\Biggl\|
\frac{J_h^{n+1}+J_h^n}{2}
\Biggr\|_{0,h}^2
\le 0, 
\end{align}
with the discrete energy for the Maxwell Kerr-Lorentz model defined as 
\begin{align}\label{eq:discretenergyLFKerr}
\mathcal{E}_{h, LF,\mathrm{Kerr}}^n := & \frac{1}{2}\Biggl(\mu_0 \| \overline{H_h}^{n} \|_{\frac{1}{2},h}^2 + \epsilon_0 \epsilon_\infty \left\langle E_h^n, \mathcal{A}_hE_h^n \right\rangle_{0,h} + \frac{3a\epsilon_0}{2}\|(E_h^n)^2\|^2_{0,h} \notag \\
& \qquad+ \frac{\epsilon_0 \omega_0^2}{\omega_p^2}\|P^n_h\|^2_{0,h} + \frac{\epsilon_0}{\omega_p^2}\|J^n_h\|^2_{0,h}\Biggr).
\end{align}
Hence, the correct asymptotic limit is captured as discussed in \cite{bokil2018high}, where this energy estimate is derived as a special case of the Maxwell Kerr-Lorentz model in 1D. This limiting energy for the Kerr limit is compatible with the continuous energy definition and energy identity given
in Theorem \ref{ctenergylorentz}. Thus, formally, as $\varepsilon \to 0$, we have $\chi_h^n \to (E_h^n)^2$ and 
$\mathcal{E}^n_{h,LF} \to \mathcal{E}^n_{h,LF,\mathrm{Kerr}}$.
\end{remark}

\subsection{Fully Discrete  $(2, 2M)$ Trapezoidal FDTD Methods}\label{22MTraps}
In this section, we construct \((2, 2M)\) order trapezoidal FDTD schemes for the discretization of the system \eqref{eq:kerr-debyelorentz1da}. The primary objective is to develop schemes that are unconditionally stable. These \((2, 2M)\) schemes combine the high-order \(2M\) spatial discretizations with the second-order trapezoidal time integrator. 

The trapezoidal schemes do not stagger the electric and magnetic fields in time as the leapfrog schemes do. However, as in the case of the $(2,2M)$ leapfrog FDTD schemes, the magnetic field lies in \(V_{\frac{1}{2},h}\) and is therefore staggered in space relative to the electric field in \(V_{0,h}\). Unlike the \((2,2M)\) leapfrog FDTD schemes which are conditionally stable under a stability condition, we will show that the $(2,2M)$ trapezoidal FDTD schemes are unconditionally stable due to the implicit nature of the time discretization. 

The fully discrete \((2, 2M)\) trapezoidal FDTD schemes for the Maxwell-Kerr-Debye-Lorentz system are given as follows:

Assume $D_h^{n}, E_h^n, \chi_h^n, P^n_h, J^n_h \in V_{0,h}$ and $H_{h}^{n} \in V_{\frac{1}{2},h}$ are given and $\chi_h^n \geq 0$. Then $\forall \, n \in \mathbb{N}$, with $0 < n \leq N$, find $D_{h}^{n+1}, E_{h}^{n+1}, \chi_h^{n+1}, P^{n+1}_h, J^{n+1}_h \in V_{0,h}$ and $H_{h}^{n+1} \in V_{\frac{1}{2},h}$ so that $\forall \, 0 < j < J$ we have

\begin{subequations} \label{eq:discretelorentzT}
    \begin{align}
       \mu_0 \frac{H_{j+\frac{1}{2}}^{n+1} - H_{j+\frac{1}{2}}^{n}}{\Delta t} &= \frac{1}{2} \left( \left( \mathcal{D}_{h}^{(2M)} E_h^{n+1} \right)_{j+\frac{1}{2}} + \left( \mathcal{D}_{h}^{(2M)} E_h^n \right)_{j+\frac{1}{2}} \right), \label{eq:faradayb}  \\
       \frac{D_j^{n+1} - D_j^n}{\Delta t} &= \frac{1}{2} \left( \left( \widetilde{\mathcal{D}}_{h}^{(2M)} H_h^{n+1} \right)_j + \left( \widetilde{\mathcal{D}}_{h}^{(2M)} H_h^n \right)_j \right), \label{eq:ampereb} \\
       \left( \frac{1}{2} + \frac{1}{2} e^{-\frac{\Delta t}{\varepsilon}} + \frac{\Delta t}{2\varepsilon} e^{-\frac{\Delta t}{2\varepsilon}}\right) \chi_j^{n+1} &= e^{-\frac{\Delta t}{\varepsilon}} \chi_j^n+ \frac{1}{2} \left( 1 - e^{-\frac{\Delta t}{\varepsilon}} \right) (E_j^{n+1})^2 + \frac{\Delta t}{2\varepsilon} e^{-\frac{\Delta t}{2\varepsilon}}
        ( E_j^n)^2 ,\label{eq:chidiscreteb}\\
        \frac{P_j^{n+1} - P_j^n}{\Delta t} &= \frac{1}{2} \left( J_j^{n+1} + J_j^n \right),\label{eq:densitydiscretebTP}\\
        \frac{J_j^{n+1} - J_j^n}{\Delta t} & = -\frac{1}{2\tau} \left( J_j^{n+1} + J_j^n \right)- \frac{\omega_0^2}{2} \left(P_j^{n+1} + P_j^n \right)+ \frac{\omega_p^2}{2} \left( E_j^{n+1} + E_j^n \right),\label{eq:lorentzdiscreteTP}\\
         D^{n+1}_j &= \epsilon_0\left(\left(\epsilon_\infty +a\chi^{n+1}_j\right)E^{n+1}_j + P^{n+1}_j\right).\label{algcontlawT}
    \end{align}
\end{subequations}
\begin{remark}
As was done for the leapfrog schemes, a second order approximation of the modified exponential time integrator derived in \eqref{eq:modified-exp-relation} is used to approximate the Kerr-Debye equation in \eqref{eq:chidiscreteb}. As $\varepsilon \to 0$, \eqref{eq:chidiscreteb} becomes $\chi^{n+1}_j = (E^{n+1}_j)^2$ and the correct asymptotic limit is captured.
\end{remark}
 
Replacing the constitutive law \eqref{algcontlawT} with \eqref{eq:energybasedconstlaw}, we obtain $(2,2M)$ trapezoidal FDTD methods that are unconditionally stable. The theorem below provides a proof of the unconditional stability of these schemes using the energy method. 
\begin{theorem}
Let $H_h^{n} \in V_{\frac{1}{2},h}$, $D^n_h, E_h^n, \chi_h^n, P^n_h, J^n_h \in V_{0,h}$, and assume $\chi^n\geq 0$.  The $(2,2M)$ trapezoidal FDTD schemes \eqref{eq:discretelorentzT} are unconditionally stable and satisfy the discrete energy identity
\begin{equation}\label{energydscreteKD}
\delta_t \mathcal{E}_{h, \text{Trap}}^{n+\frac{1}{2}} = -\,\frac{a\epsilon_0}{2\,\varepsilon}\Biggl\|\frac{\chi_h^{n+1}+\chi_h^n}{2}-\frac{(E_h^{n+1})^2+(E_h^n)^2}{2}\Biggr\|^2_{0,h} - \frac{\epsilon_0}{\omega_p^2 \tau} \Biggl\|\frac{J^{n+1}_h + J^n_h}{2}\Biggr\|^2_{0,h},
\end{equation}
where the discrete energy is defined for all $n\in \mathbb{N}$ as
\begin{align}\label{eq: 30}
\mathcal{E}_{h, \text{Trap}}^n := & \frac{1}{2}\Biggl( \mu_0 \| H_h^{n} \|_{\frac{1}{2},h}^2 + \epsilon_0 \epsilon_\infty \|E^n_h\|^2_{0,h} + a\epsilon_0\left\langle\chi_h^n,(E_h^n)^2\right\rangle_{0,h} + \frac{a\epsilon_0}{2}\bigl\|\chi_h^n\bigr\|^2_{0,h} \notag \\ 
& \quad\quad + \frac{\epsilon_0 \omega_0^2}{\omega_p^2}\|P^n_h\|^2_{0,h} + \frac{\epsilon_0}{\omega_p^2}\|J^n_h\|^2_{0,h}\Biggr).
\end{align}    
\end{theorem}
\begin{proof}
Multiplying the equation \eqref{eq:faradayb} by $\overline{H}_{j+\frac{1}{2}}^{n+\frac{1}{2}}$, we get
\begin{equation}
\mu_0 \frac{H_{j+\frac{1}{2}}^{n+1} - H_{j+\frac{1}{2}}^{n}}{\Delta t} \cdot \overline{H}_{j+\frac{1}{2}}^{n+\frac{1}{2}} =
\left(\mathcal{D}_{h}^{(2M)} \overline{E}_h^{n+\frac{1}{2}}\right)_{j+\frac{1}{2}} \cdot \overline{H}_{j+\frac{1}{2}}^{n+\frac{1}{2}}. 
\end{equation}
Summing over all mesh points on the dual grid, multiplying by $h$ and simplifying yields 
\begin{equation}\label{eq: 49i}
     \frac{1}{2\Delta t} \left(\mu_0\|H^{n+1}_h\|^2_{\frac{1}{2}, h} - \mu_0\|H^{n}_h\|^2_{\frac{1}{2}, h}\right)  = \left\langle
\mathcal{D}_{h}^{(2M)}\overline{E}_h^{n+\frac{1}{2}},\,
\overline{H}_h^{\,n+\frac12}
\right\rangle_{\frac12,h}
\end{equation}
Multiplying the equation \eqref{eq:ampereb} by $\overline{E_j}^{n+\frac{1}{2}}$, we obtain
\begin{align}
 \frac{1}{2\Delta t}(D^{n+1}_j - D^n_j)(E^{n+1}_j + E^n_j)
= \left( \widetilde{\mathcal{D}}_{h}^{(2M)} \overline{H}_h^{n+\frac{1}{2}} \right)_j  \cdot \overline{E_j}^{n+\frac{1}{2}}.
\end{align}
Summing over all mesh points on the primal grid, multiplying by $h$ and simplifying gives 
\begin{align}\label{eq: 49ii}
 \frac{1}{2\Delta t}\left\langle D_h^{n+1}-D_h^n, E_h^{n+1}+E_h^n \right\rangle_{0,h}
= \left\langle
\widetilde{\mathcal{D}}_{h}^{(2M)} \overline{H}_h^{n+\frac{1}{2}},
\,
\overline{E}_h^{\,n+\frac12}
\right\rangle_{0,h}
\end{align}
Multiply \eqref{eq:densitydiscretebTP} by $P_h^{n+1} + P_h^n$, summing over all mesh points on the primal grid, multiplying by $h$ and then simplifying yields 
\begin{equation}\label{eq:desnitydiscretea}
    \|P_h^{n+1}\|^2_{0,h} - \|P_h^n\|^2_{0,h} = \frac{\Delta t}{2} \left\langle J_h^{n+1} + J_h^n, P_h^{n+1} + P_h^n \right\rangle_{0,h}
\end{equation}
Multiplying \eqref{eq:lorentzdiscreteTP} by $\left(J_h^{n+1} + J_h^n\right)/\omega_p^2$, summing over all mesh points on the primal grid, multiplying by $h$ and rewriting the result by applying \eqref{eq:densitydiscretebTP} and \eqref{eq:desnitydiscretea} we obtain 
\begin{align}\label{lorentzfinb}
    \frac{1}{\omega_p^2}\left(\|J_h^{n+1}\|^2 - \|J_h^n\|^2\right) + \frac{\Delta t}{2\omega_p^2 \tau}\|J_h^{n+1} + J_h^n\|^2 &+ \frac{\omega_{0}^2}{\omega_p^2}\left(\|P_h^{n+1}\|^2 - \|P_h^n\|^2\right) 
    \notag\\
    &\qquad \qquad = \left\langle E_h^{n+1} + E_h^n, P_h^{n+1} - P_h^n \right\rangle_{0,h} 
\end{align}
Next, summing \(\eqref{eq:energybasedconstlaw}\) over all mesh points on the primal grid and multiplying by $h$, we obtain
\begin{align}\label{eq:constnormtrap}
&\left\langle D_h^{n+1}-D_h^n, E_h^{n+1}+E_h^n \right\rangle_{0,h}
=\frac{a\epsilon_0 \Delta t}{\varepsilon}
\left\|
\frac{\chi_h^{n+1}+\chi_h^n}{2}
-
\frac{(E_h^{n+1})^2+(E_h^n)^2}{2}
\right\|_{0,h}^2
\notag\\
&\qquad +
\epsilon_0\epsilon_\infty\left(\|E_h^{n+1}\|_{0,h}^2-\|E_h^n\|_{0,h}^2\right)
+ \frac{a\epsilon_0}{2}\left(\|\chi_h^{n+1}\|_{0,h}^2-\|\chi_h^n\|_{0,h}^2\right)
\notag\\
&\qquad\quad
+ a\epsilon_0\left(
\left\langle \chi_h^{n+1}, (E_h^{n+1})^2 \right\rangle_{0,h}
-
\left\langle \chi_h^n, (E_h^n)^2 \right\rangle_{0,h}
\right)
+ \epsilon_0 \left\langle P_h^{n+1}-P_h^n, E_h^{n+1}+E_h^n \right\rangle_{0,h}.
\end{align}
Combining equations \eqref{eq: 49i}, \eqref{eq: 49ii} and \eqref{lorentzfinb}, and \eqref{eq:constnormtrap} and using the summation by parts property in Lemma \ref{sumbypart}, we obtain  
\begin{align} \label{eq: finalform}
&\frac{1}{2\Delta t}\Bigg(\mu_0\|H^{n+1}_h\|^2_{\frac{1}{2}, h} - \mu_0\|H^{n}_h\|^2_{\frac{1}{2}, h} +
    \epsilon_0 \epsilon_\infty \left( \|E_h^{n+1}\|_{0,h}^2 - \|E_h^n\|_{0,h}^2 \right) 
    + \frac{a\epsilon_0}{2} \left( 
        \|\chi_h^{n+1}\|_{0,h}^2 - \|\chi_h^n\|_{0,h}^2 
    \right) \notag \\
&+ a\epsilon_0 \left( 
        \left\langle \chi_h^{n+1}, (E_h^{n+1})^2 \right\rangle_{0,h} 
        - \left\langle \chi_h^n, (E_h^n)^2 \right\rangle_{0,h}\right) + \frac{\epsilon_0}{\omega_p^2}\Big(\|J_h^{n+1}\|_{0,h}^2 - \|J_h^n\|_{0,h}^2\Big)
    + \frac{\epsilon_0 \omega_0^2}{\omega_p^2}\Big(\|P_h^{n+1}\|_{0,h}^2 - \|P_h^n\|^2_{0,h}\Big)
    \Bigg) \notag \\
    & \quad = -  \frac{a\epsilon_0}{2\varepsilon} \left\| 
        \frac{\chi_h^{n+1} + \chi_h^n}{2} 
        - \frac{(E_h^{n+1})^2 + (E_h^n)^2}{2} 
    \right\|_{0,h}^2 - \frac{\epsilon_0}{\omega_p^2\tau}\|\overline{J}_h^{n+\frac{1}{2}}\|^2_{0,h}.
\end{align}
Using the definition of the discrete energy \eqref{eq: 30}, the equation \eqref{eq: finalform} can be rewritten as 
\begin{align}\label{eq:discreteenergyCN}
  \frac{\mathcal{E}_{h, \text{Trap}}^{n+1} - \mathcal{E}_{h,\text{Trap}}^n}{\Delta t} &= -\,\frac{a\epsilon_0}{2\,\varepsilon}\Biggl\|\frac{\chi_h^{n+1}+\chi_h^n}{2}-\frac{(E_h^{n+1})^2+(E_h^n)^2}{2}\Biggr\|^2 - \frac{\epsilon_0}{\omega_p^2 \tau} \Biggl\|\overline{J}_h^{n+\frac{1}{2}}\Biggr\|^2 \leq 0
\end{align}
which proves the unconditional discrete energy decay result \eqref{energydscreteKD} and shows \(\mathcal{E}_{h, \text{Trap}}^{n+1}\le\mathcal{E}_{h, \text{Trap}}^n\). Theorem \ref{eq:positivitydiscrete} provides the positivity of the susceptibility  $\chi_{j}^{n+1}$ under the assumption of positivity of $\chi_{j}^{n}$. Thus, the definition \eqref{eq: 30} provides a discrete energy and the identity \eqref{energydscreteKD} (also \eqref{eq:discreteenergyCN}) is a discrete energy identity for the fully discrete $(2,2M)$ trapezoidal FDTD scheme, compatible with the continuous energy definition and energy identity given in Theorem \ref{ctenergy}.   
\end{proof}
\begin{remark} 
As $\varepsilon \to 0$, \eqref{eq:chidiscreteLF} becomes $\chi^{n+1}_h = (E^{n+1}_h)^2$. Thus, the energy estimate \eqref{energydscreteKD} becomes 
\begin{align}\label{eq:discreteenergyTrapKerr}
\frac{\mathcal{E}_{h,\mathrm{Trap},\mathrm{Kerr}}^{n+1}
-\mathcal{E}_{h,\mathrm{Trap},\mathrm{Kerr}}^{n}}{\Delta t}
&=
-
\frac{\epsilon_0}{\omega_p^2\tau}
\Biggl\|
\frac{J_h^{n+1}+J_h^n}{2}
\Biggr\|_{0,h}^2
\le 0, 
\end{align}
with the discrete energy for the $(2,2M)$ trapezoidal FDTD schemes for the Maxwell Kerr-Lorentz model defined as 
\begin{align}\label{eq:discretenergyTrapKerr}
\mathcal{E}_{h, \mathrm{Trap},\mathrm{Kerr}}^n := & \frac{1}{2}\Biggl(\mu_0 \| \overline{H_h}^{n} \|_{\frac{1}{2},h}^2 + \epsilon_0 \epsilon_\infty \left\|E_h^n\right\|^2_{0,h} + \frac{3a\epsilon_0}{2}\|(E_h^n)^2\|^2_{0,h} + \frac{\epsilon_0 \omega_0^2}{\omega_p^2}\|P^n_h\|^2_{0,h} + \frac{\epsilon_0}{\omega_p^2}\|J^n_h\|^2_{0,h}\Biggr).
\end{align}
Hence, the correct asymptotic limit is captured as discussed in \cite{bokil2018high}, where this energy estimate is derived as a special case of the Maxwell Kerr-Lorentz model in 1D. This limiting energy for the Kerr limit is compatible with the continuous energy definition and energy identity given
in Theorem \ref{ctenergylorentz}. Thus, formally, as $\varepsilon \to 0$, we have $\chi_h^n \to (E_h^n)^2$ and 
$\mathcal{E}^n_{h,\mathrm{Trap}} \to \mathcal{E}^n_{h,\mathrm{Trap},\mathrm{Kerr}}$.
\end{remark}

\section{Construction of High-Order Temporal Approximations}
In Section~\ref{22MTraps}, we introduced a family of $(2,2M)$ trapezoidal FDTD schemes that are second-order accurate in time and $2M$th-order accurate in space. For these methods, the global truncation error takes the form $\mathcal{O}(\Delta t^{2} + h^{2M})$. In practice, achieving the full spatial accuracy requires a strong coupling between the time step $\Delta t$ and the mesh size $h$. Specifically, to prevent the temporal error from dominating, one must ensure $\Delta t^{2} \leq h^{2M}$, which implies that $\Delta t = \mathcal{O}(h^{M})$. For example, the $(2,4)$ scheme requires $\Delta t = \mathcal{O}(h^{2})$, and the $(2,6)$ scheme requires $\Delta t = \mathcal{O}(h^{3})$. This condition becomes increasingly restrictive as $M$ increases.

To avoid this severe time--step restriction and attain high-order accuracy in both time and space, we construct extrapolated methods in time that allow for arbitrarily high temporal and spatial accuracy with the same accuracy in time and space. We use Richardson extrapolation \cite[p.~183]{BurdenFairesBurden2015}, \cite{burg2009application,fornberg2003some} to extend temporal accuracy from second order to $2M$th order in time, for $M\in \mathbb{N}, M > 1$. We refer to this new class of  methods based on Richardson extrapolation as the $(2M,2M)$ extrapolated trapezoidal FDTD schemes.  

Richardson extrapolation is an important and powerful technique to generate high order temporal approximations for partial differential equations. However, there can be issues with stability and structure preserving properties when applying this technique \cite{burg2009application}. We address these issues in later sections. 
 
 We will construct fully discrete extrapolated trapezoidal FDTD methods first for the limiting Maxwell-Kerr-Lorentz model presented in \eqref{eq:kerr-debyelorentz1db}, and then for the Maxwell Kerr-Debye-Lorentz model presented in \eqref{eq:kerr-debyelorentz1da}. 

\subsection{Fully discrete $(2M, 2M)$ Extrapolated Trapezoidal FDTD Methods for Maxwell Kerr-Lorentz System}\label{2M2M}
In this section, we develop $(2M,2M)$ extrapolated trapezoidal FDTD chemes for the Maxwell--Kerr--Lorentz model using Richardson extrapolation

 We note that each component of the fully discrete solution for the $(2,2M)$ trapezoidal FDTD schemes at a fixed grid point in space and time has a local truncation error expansion in time in which only even powers of $\Delta t$ are present and terms in odd powers of $\Delta t$ vanish . This is due to the trapezoidal approximation of the time derivatives and time averages of lower order terms. For each $M \in \mathbb{N}$, we apply the Richardson extrapolation method \cite{ascoli2018limitations,bach2021effectiveness} successively $M-1$ times to the $(2,2M)$ trapezoidal FDTD schemes described in Section~\ref{22MTraps} to generate methods that have $2M$ order accuracy in time.

We compute the solution of the $(2,2M)$ trapezoidal FDTD schemes for the Maxwell-Kerr-Lorentz model in \eqref{eq:kerr-debyelorentz1db} $M$ times using successively refined time steps $\Delta t, \Delta t/2, \cdots , \Delta t/2^{M-1}$. For $n \geq 0$, let the solution of the $(2,2M)$ trapezoidal FDTD schemes at time grid point $t^n$, computed with time step $\Delta t/2^{k-1}$ be denoted as  $U_{h}^{n,k} = [E_{h}^{n,k}, D_{h}^{n,k}, H_{h}^{n,k}, P_{h}^{n,k}, J_{h}^{n,k}]^T$ and let $U(t^n) = [E(t^n), D(t^n), H(t^n), P(t^n), J(t^n)]^T$ denote the exact solution of the  Maxwell-Kerr-Lorentz model presented in \eqref{eq:kerr-debyelorentz1db} at the time grid point $t^n$. For $k = 1, \ldots, M$, assuming sufficient regularity of the exact solution, the global error in the numerical solution at $t^n$ is given via 
\begin{align}
\label{RE1}
    U_{h}^{n,k}
    &= U(t^{n})
       + \sum_{\ell = 1}^{M-1} C_{h,\ell}
         \left(\Delta t^2 \, 2^{-2(k-1)}\right)^{\ell} + \mathcal{O}(\Delta t^{2M}) +\mathcal{O}(h^{2M}) \notag \\
     &=U(t^{n})
       + \sum_{\ell=1}^{M-1} C_{h,\ell} \,\Delta t^{2\ell} \, s_k^{\,\ell}+ \mathcal{O}(\Delta t^{2M})  +\mathcal{O}(h^{2M}),
\end{align}
where for $\ell = 1,\cdots, M-1, C_{h,\ell} = [C_{E,h,\ell}, C_{D,h,\ell}, C_{H,h,\ell}, C_{P,h,\ell}, C_{J,h,\ell}, C_{\chi,h,\ell}]^T$ depends on time derivatives of the components of the exact solution at $t^n$ and spatial grid points, and where $s_k$ is given by
\[
    s_k = \left(\frac{1}{2^{\,k-1}}\right)^{\!2}
         \;=\; 4^{-(k-1)},\qquad k=1,\dots,M.
\]
The application of Richardson extrapolation $M-1$ times involves seeking coefficients $\{\alpha_k\}_{k=1}^{M}$ such that the extrapolated solution denoted as $U_{h,\mathrm{RE}}^{n}$ satisfies 
\begin{align}
\label{RE2}
    U_{h,\mathrm{RE}}^{n}
      \;=\; \sum_{k=1}^{M}\alpha_k\,
            U_h^{n,k}
      \;=\; U(t^{n}) + \mathcal{O}(\Delta t^{2M})+ \mathcal{O}(h^{2M}),
\end{align}
i.e., the first $M-1$ even-order terms in $\Delta t$ cancel. We note that the presentation of Richardson extrapolation as done above is effectively the same as an application of the general {\it undetermined coefficients method} \cite{cohen2002higher}. Combining \eqref{RE1} and \eqref{RE2}, we can show that these coefficients must satisfy the conditions
\begin{equation}
\sum_{k=1}^{M}\alpha_k\,s_k^{\,\ell} = \delta_{\ell 0} = 
\begin{cases}
1, & \ell=0,\\
0, & \ell=1,2,\dots,M-1,
\end{cases}
\label{eq:moment-conds}
\end{equation}
which can be written as the square linear system $A\boldsymbol{\alpha} = \mathbf{e}_1$ with a Vandermonde matrix
\(A\in\mathbb{R}^{M\times M}\),
\[
A \;=\;
\begin{bmatrix}
1 & 1 & \cdots & 1\\
s_1 & s_2 & \cdots & s_M\\
s_1^{2} & s_2^{2} & \cdots & s_M^{2}\\
\vdots & \vdots & \ddots & \vdots\\
s_1^{M-1} & s_2^{M-1} & \cdots & s_M^{M-1}
\end{bmatrix},
\qquad
\alpha=\begin{bmatrix}\alpha_1\\ \alpha_2\\ \vdots\\ \alpha_M\end{bmatrix},
\qquad
e_1=\begin{bmatrix}1\\0\\ \vdots\\ 0\end{bmatrix},
\qquad
s_k=4^{-(k-1)}.
\]
Since the \(s_1,\dots,s_{M}\) values are distinct, \(A\) is nonsingular and the solution $\boldsymbol{\alpha}$ is unique. 

As $M$ increases, the Vandermonde system $A\boldsymbol{\alpha} = \mathbf{e}_1$ becomes ill-conditioned. Because of this, we derive a closed form solution for the components $\alpha_j$ of the vector $\boldsymbol{\alpha}$. Following the exposition presented in \cite{ascoli2018limitations,bach2021effectiveness}, we observe that 
\[
    A\boldsymbol{\alpha} = \mathbf{e}_1 \implies \boldsymbol{\alpha}^\top A^\top = \mathbf{e}_1^\top \quad \implies 
\boldsymbol{\alpha}^\top = \mathbf{e}_1^\top (A^\top)^{-1},
\]
so \(\boldsymbol{\alpha}^\top\) is the first row of \((A^\top)^{-1}\). To compute this first row, without computing the inverse of $A^\top$, we use Lagrange interpolation as done in \cite{ascoli2018limitations,bach2021effectiveness}. For any polynomial $P$, 
\[
    P(z) = a_0 + a_1 z + \cdots + a_{M-1} z^{M-1},
\]
we can collect its coefficients into the vector $\mathbf{a} = (a_0, \dots, a_{M-1})^\top$. 
Then, we have 
\[
A^\top \mathbf{a} =
\begin{bmatrix}
P(s_1)\\ P(s_2)\\ \vdots\\ P(s_M)
\end{bmatrix}.
\]
Rearranging terms, we get 
\[
    \boldsymbol{a} = (A^\top)^{-1}
    \begin{bmatrix}
        P(s_1)\\ P(s_2)\\ \vdots\\ P(s_M)
    \end{bmatrix}.
\]
For each $i,\,  1\leq i\leq M$, we choose \(P\) = \(P_i\) so that
\[
P_i(s_i) = 1,\qquad P_i(s_k) = 0 \quad (k \ne i),
\]
i.e., $P_i(s_k), 1\leq i,k\leq M$ satisfies the delta property. With this choice of the polynomials we have that $A^\top \mathbf{a}_i = \mathbf{e}_i$ where $\mathbf{a}_i = (a_{0i}, a_{1i}\dots, a_{(M-1)i})^\top$ is the coefficient vector for polynomial $P_i$. Thus,
\[
    \mathbf{a}_i = (A^\top)^{-1}\mathbf{e}_i,
\]
and $a_{0i}, 1\leq i\leq M$ is the element in the first row and $i$th column of the matrix $(A^\top)^{-1}$. In particular, it is also equal to the element $\alpha_{i}$ of the vector $\boldsymbol{\alpha}$. For \(i = 1\), using Lagrange interpolation, the polynomial $P_i$ is written explicitly as 
\[
P_1(z) = \frac{(z - s_2)(z - s_3)\cdots(z - s_M)}
{(s_1 - s_2)(s_1 - s_3)\cdots(s_1 - s_M)}.
\]
Since \((A^\top)^{-1} \mathbf{e}_1 = \mathbf{a}_{1}\) is the coefficient vector of \(P_1\), its constant term $\mathbf{a}_{01}$ is
\[
P_1(0) = a_{01} =  \alpha_1 = \frac{(-s_2)(-s_3)\cdots(-s_M)}
{(s_1 - s_2)(s_1 - s_3)\cdots(s_1 - s_M)}.
\]
By generalizing this idea, the $\alpha_k, 1\leq k\leq M$, coefficients can be computed as 
\begin{equation}
\label{alphas}
\alpha_k = P_k(0) = \prod_{\substack{m = 1 \\ m \ne k}}^M \frac{s_m}{s_m - s_k}=\prod_{\substack{m = 1 \\ m \ne k}}^M \frac{1}{1-4^{m-k}},
\end{equation}
with $s_k=4^{-(k-1)}$. 

Pairing Richardson extrapolation in \eqref{RE2} and \eqref{alphas} with a spatial discretization of order $2M$ yields a scheme that has $2M$ order accuracy in space and time and we denote such a scheme as the $(2M,2M)$ extrapolated trapezoidal FDTD scheme or the limiting Maxwell-Kerr-Lorentz model. 

\subsection{Fully Discrete $(2M, 2M)$ Extrapolated Trapezoidal FDTD Methods for Maxwell Kerr-Debye-Lorentz System}\label{2M2MKDL}

The local truncation error associated with the exponential time integrator \eqref{eq:chidiscreteb} for the nonlinear Debye equation \eqref{kerrdebye1d} does not, in general, contain only even powers of $\Delta t$. Consequently, the even-power temporal error expansion used to derive the Richardson extrapolation coefficients in \eqref{RE2} and \eqref{alphas} for the limiting Maxwell-Kerr-Lorentz system does not hold for the Maxwell-Kerr-Debye-Lorentz system. Building a fully discrete $(2M,2M)$ system in the case when the nonlinear Debye equation is included requires a modified Richardson extrapolation that takes into consideration both the even and odd powers of $\Delta t$, which we develop in this section. 

For each $M\in\mathbb{N}$, we apply Richardson extrapolation to the $(2,2M)$ trapezoidal FDTD schemes for the Maxwell-Kerr-Debye-Lorentz system. In contrast to the Maxwell-Kerr-Lorentz system, the temporal error expansion generally contains both even and odd powers of $\Delta t$. 

We compute the solution of the $(2,2M)$ trapezoidal FDTD schemes for the Maxwell-Kerr-Debye-Lorentz model $2M-1$ times using successively refined time steps $\Delta t,\Delta t/2,\ldots,\Delta t/2^{2M-2}$. For $n\geq0$, let the solution at the common time grid point $t^n$, computed with time step $\Delta t/2^{k-1}$, be denoted by $U_h^{n,k}=[E_h^{n,k},D_h^{n,k},H_h^{n,k},P_h^{n,k},J_h^{n,k},\chi_h^{n,k}]^\top$, for $k=1,\ldots,2M-1$. Thus, the computation with time step $\Delta t/2^{k-1}$ requires $2^{k-1}n$ time steps to reach the common physical time $t^n=n\Delta t$. Let $U(t^n)=[E(t^n),D(t^n),H(t^n),P(t^n),J(t^n),\chi(t^n)]^\top$ denote the exact solution of the Maxwell-Kerr-Debye-Lorentz model at $t^n$.

For $k=1,\ldots,2M-1$, assuming sufficient regularity of the exact solution and a full integer-power temporal error expansion, the global error in the numerical solution at $t^n$ is given by
\begin{align}
U_h^{n,k}
&=
U(t^n)
+
\sum_{\ell=2}^{2M-1}
C_{h,\ell}
\left(
\Delta t\,2^{-(k-1)}
\right)^\ell
+
\mathcal{O}(\Delta t^{2M})
+
\mathcal{O}(h^{2M})
\notag\\
&=
U(t^n)
+
\sum_{\ell=2}^{2M-1}
C_{h,\ell}\Delta t^\ell r_k^\ell
+
\mathcal{O}(\Delta t^{2M})
+
\mathcal{O}(h^{2M}),
\label{RE1-general}
\end{align}
where $r_k=2^{-(k-1)}$ for $k=1,\ldots,2M-1$, and $C_{h,\ell}=[C_{E,h,\ell},C_{D,h,\ell},C_{H,h,\ell},C_{P,h,\ell},C_{J,h,\ell},C_{\chi,h,\ell}]^T$ for $\ell=2,\ldots,2M-1$. The coefficients \(C_{h,\ell}\) depend on time derivatives of the components of the exact solution at \(t^n\) and on the spatial grid points, but are independent of \(\Delta t\) and \(k\). 

Following the standard Richardson extrapolation procedure
\cite[p.~183]{BurdenFairesBurden2015} and the coefficient construction
in \cite{ascoli2018limitations,bach2021effectiveness}, we seek coefficients
$\{\tilde{\alpha}_k\}_{k=1}^{2M-1}$ such that the extrapolated solution, denoted by
\begin{align}
\label{RE2-general}
U_{h,\mathrm{RE}}^{n}
&=
\sum_{k=1}^{2M-1}\tilde{\alpha}_k U_h^{n,k}
=
U(t^n)
+
\mathcal{O}(\Delta t^{2M})
+
\mathcal{O}(h^{2M}),
\end{align}
i.e., the temporal error terms of orders $\Delta t^2,\Delta t^3,\ldots,\Delta t^{2M-1}$ cancel. Combining \eqref{RE1-general} and \eqref{RE2-general}, we can show that these coefficients must satisfy
\begin{equation}
\sum_{k=1}^{2M-1}\tilde{\alpha}_k r_k^\ell
=
\delta_{\ell0}
=
\begin{cases}
1, & \ell=0,\\
0, & \ell=2,3,\ldots,2M-1,
\end{cases}
\label{eq:moment-conds-general}
\end{equation}
which gives the square linear system $A\boldsymbol{\tilde{\alpha}}=\mathbf{e}_1$, where $A\in\mathbb{R}^{(2M-1)\times(2M-1)}$ is a generalized Vandermonde matrix
\[
A=
\begin{bmatrix}
1 & 1 & \cdots & 1\\
r_1^2 & r_2^2 & \cdots & r_{2M-1}^2\\
r_1^3 & r_2^3 & \cdots & r_{2M-1}^3\\
\vdots & \vdots & \ddots & \vdots\\
r_1^{2M-1} & r_2^{2M-1} & \cdots & r_{2M-1}^{2M-1}
\end{bmatrix},
\qquad
\boldsymbol{\tilde{\alpha}}
=
\begin{bmatrix}
\tilde{\alpha}_1\\
\tilde{\alpha}_2\\
\vdots\\
\tilde{\alpha}_{2M-1}
\end{bmatrix},
\qquad
\mathbf{e}_1
=
\begin{bmatrix}
1\\
0\\
\vdots\\
0
\end{bmatrix},
\qquad
r_k=2^{-(k-1)}.
\]
The matrix $A$ differs from an ordinary Vandermonde matrix because the row containing the first powers of $r_k$ is absent. Since the values $r_1,\ldots,r_{2M-1}$ are distinct and positive, $A$ is nonsingular and the solution $\boldsymbol{\tilde{\alpha}}$ is unique.

As $M$ increases, the generalized Vandermonde system \cite[p.~99]{Gantmacher2000}, $A\boldsymbol{\tilde{\alpha}}=\mathbf{e}_1$ becomes increasingly ill-conditioned. Because of this, we derive a closed form for the components $\tilde{\alpha}_i$ of the vector $\boldsymbol{\tilde{\alpha}}$. In a manner similar to the previous section, we observe that
\[
A\boldsymbol{\tilde{\alpha}}=\mathbf{e}_1
\quad\Longrightarrow\quad
\boldsymbol{\tilde{\alpha}}^{\top}A^{\top}=\mathbf{e}_1^{\top}
\quad\Longrightarrow\quad
\boldsymbol{\tilde{\alpha}}^{\top}
=
\mathbf{e}_1^{\top}(A^{\top})^{-1},
\]
so $\boldsymbol{\tilde{\alpha}}^{\top}$ is the first row of $(A^{\top})^{-1}$. To determine this row without explicitly computing the inverse of $A^{\top}$, we consider polynomials of the form
\[
P(z)
=
a_0+a_2z^2+a_3z^3+\cdots+a_{2M-1}z^{2M-1},
\]
and collect their coefficients into the vector
$\mathbf{a}=(a_0,a_2,a_3,\ldots,a_{2M-1})^{\top}$. Then
\[
A^{\top}\mathbf{a}
=
\begin{bmatrix}
P(r_1)\\
P(r_2)\\
\vdots\\
P(r_{2M-1})
\end{bmatrix},
\]
and hence
\[
\mathbf{a}
=
(A^{\top})^{-1}
\begin{bmatrix}
P(r_1)\\
P(r_2)\\
\vdots\\
P(r_{2M-1})
\end{bmatrix}.
\]

For each $i$, $1\leq i\leq2M-1$, we choose $P=P_i$ from the polynomial space
$\mathcal P=\operatorname{span}\{1,z^2,z^3,\ldots,z^{2M-1}\}$ such that
\[
P_i(r_i)=1,
\qquad
P_i(r_k)=0
\quad (k\neq i),
\]
so that $P_i(r_k)=\delta_{ik}$ for $1\leq i,k\leq2M-1$. With this choice, $A^{\top}\mathbf{a}_i=\mathbf{e}_i$, where
$\mathbf{a}_i=(a_{0i},a_{2i},a_{3i},\ldots,a_{(2M-1)i})^{\top}$ is the coefficient vector of $P_i$. Thus,
\[
\mathbf{a}_i=(A^{\top})^{-1}\mathbf{e}_i.
\]
Hence, $a_{0i}$ is the element in the first row and $i$th column of $(A^{\top})^{-1}$. In particular,
\begin{align}
\label{coeff-general}
\tilde{\alpha}_i=a_{0i}=P_i(0),
\qquad
i=1,\ldots,2M-1.
\end{align}
To construct \(P_i\), we first define the ith Lagrange basis polynomial by
\begin{align}\label{Llang}
L_i(z)
=
\prod_{\substack{m=1\\m\neq i}}^{2M-1}
\frac{z-r_m}{r_i-r_m}.
\end{align}
This polynomial satisfies the Kronecker delta property, $L_i(r_k) = \delta_{ik},\, 1\leq i,k\leq 2M-1$.  However, \(L_i \notin \mathcal P\) since it generally contains a nonzero linear term. We therefore modify \(L_i\) while preserving its values at the nodes. Define the nodal polynomial
\begin{align}\label{Qlang}
Q(z)
=
\prod_{m=1}^{2M-1}(z-r_m).
\end{align}
Since $Q(r_k)=0,\, k=1,\ldots,2M-1,$ the polynomial
$
P_i(z)=L_i(z)+c_iQ(z)
$
satisfies the same interpolation conditions as \(L_i\), regardless of
the value of \(c_i\). Indeed,
\[
P_i(r_k)
=
L_i(r_k)+c_iQ(r_k)
=
L_i(r_k)
=
\delta_{ik}.
\]
It remains to choose \(c_i\) so that \(P_i\) has no linear term. This is
equivalent to requiring $P_i'(0)=0.$ With $c_i =- L'_i(0)/Q'_i(0)$, the required polynomial is
\begin{equation}
P_i(z)
=
L_i(z)
-
\frac{L_i'(0)}{Q'(0)}Q(z).
\label{eq:modified-lagrange}
\end{equation}
From \eqref{coeff-general}, one has 
\[
\tilde{\alpha}_i
=
P_i(0)
=
L_i(0)
-
\frac{L_i'(0)}{Q'(0)}Q(0).
\]
Using logarithmic differentiation on \eqref{Llang} and \eqref{Qlang}, and algebraic manipulation, we obtain  
\begin{equation}
\tilde{\alpha}_i
=
\frac{2^{i-1}}{2^{2M-1}-1}
\prod_{\substack{m=1\\m\neq i}}^{2M-1}
\frac{1}{1-2^{m-i}},
\qquad
i=1,\ldots,2M-1.
\label{eq:general-alphas}
\end{equation}
Pairing Richardson extrapolation in \eqref{RE2-general} and \eqref{eq:general-alphas} with a spatial discretization of order $2M$ yields a scheme that has $2M$th-order accuracy in space and time, and we denote such a scheme by the $(2M,2M)$ FDTD scheme for the Maxwell-Kerr-Debye-Lorentz model.

\subsection{Energy Bounds for Richardson Extrapolation}
\label{sec:energybounds}
In this section, we explore the stability of the extrapolated schemes. While we are not able to establish energy decay properties, we establish a uniform-in-time energy bound \cite{chen2025developing} for the $(2M,2M)$ extrapolated trapezoidal FDTD scheme applied to the limiting Maxwell-Kerr-Lorentz model as developed in Section~\ref{2M2M}.
Using the energy decay property of the underlying $(2,2M)$ trapezoidal FDTD scheme on the coarse mesh and the $M-1$ mesh refinements, the extrapolated solution inherits an unconditional energy boundedness estimate with a constant independent of the timestep, $\Delta t$. Here, by ``energy'', we mean a non-negative functional of the extrapolated solutions that has a similar form as the discrete energy of the underlying $(2,2M)$ trapezoidal FDTD methods.

For $k=1,\dots,M$, let $U^{n,k}_h = \bigl[E_h^{n,k},D_h^{n,k}, H_h^{n,k}, P_h^{n,k},J_h^{n,k}\bigr]$ denote the discrete solution of the $(2,2M)$ trapezoidal FDTD scheme applied to the limiting Maxwell-Kerr-Lorentz model at time $t^n$ computed with step size $\Delta t/2^{k-1}$. Using Richardson extrapolation in \eqref{RE2} and \eqref{alphas} the $(2M,2M)$ extrapolated solution components at time $t^{n+1}$ are given by
\[
\begin{aligned}
H_{h,RE}^{n+1} &= \sum_{k=1}^M \alpha_k\, H_h^{n+1,k}, 
&\quad E_{h,RE}^{n+1} &= \sum_{k=1}^M \alpha_k\, E_h^{n+1,k}, 
&\quad R_{h,RE}^{n+1}&= \sum_{k=1}^M \alpha_k\,\bigl(E_h^{n+1,k}\bigr)^2,\\
P_{h,RE}^{n+1} &= \sum_{k=1}^M \alpha_k\, P_h^{n+1,k}, 
&\quad J_{h,RE}^{n+1} &= \sum_{k=1}^M \alpha_k\, J_h^{n+1,k},
\end{aligned}
\]
with the $\alpha_k$ coefficients defined in \eqref{alphas}. We note that the extrapolation of the squares of electric field values for each refinement are not the same as the square of $E_{h,RE}^{n}$ and thus, we define a new extrapolated field, $R^n_{h,RE}$ defined at any time level $n \in\mathbb{N}$ as 
\begin{equation}
\label{eq:Esquare}
R^n_{h,RE} = \sum_{k=1}^M \alpha_k\,\bigl(E_h^{n,k}\bigr)^2 
\end{equation}. 

Next, we show that the $(2M,2M)$ extrapolated trapezoidal FDTD schemes for the Maxwell Kerr-Lorentz model satisfy a uniform-in-time energy bound independent of $\Delta t$ for all values of $M > 1$. 
\begin{theorem}
Let $\mathcal{E}_{h,\mathrm{Trap, Kerr}}^{n,k}$ be the discrete energy for the $(2,2M)$ trapezoidal FDTD scheme computed with time step $\Delta t/2^{k-1}$, $k=1,\dots,M$, applied to the limiting Maxwell-Kerr-Lorentz model defined in~\eqref{eq:discretenergyTrapKerr} and let
$\mathcal{E}_{h,\mathrm{RE}}^{n}$ be the energy of the extrapolated
$(2M,2M)$ solution at time $t^n$ defined as 
\begin{align}\label{REEng}
   \hspace*{-8pt} \mathcal{E}_{h,\mathrm{RE}}^{n} := \frac{1}{2}\Biggl(\mu_0 \| H_{h,RE}^{n} \|_{\frac{1}{2},h}^2 + \epsilon_0 \epsilon_\infty \left\|E_{h,RE}^n\right\|^2_{0,h} + \frac{3a\epsilon_0}{2}\|R_{h,RE}^n\|^2_{0,h} + \frac{\epsilon_0 \omega_0^2}{\omega_p^2}\|P^n_{h,RE}\|^2_{0,h} + \frac{\epsilon_0}{\omega_p^2}\|J^n_{h,RE}\|^2_{0,h}\Biggr).
\end{align}
Assume that \eqref{eq:discretenergyTrapKerr} satisfies $ \mathcal{E}_{h,\mathrm{Trap, Kerr}}^{n+1,k}
\le \mathcal{E}_{h,\mathrm{Trap, Kerr}}^{n,k},\, \forall \, n\ge 0, k=1,\ldots M$. Then, for all $n\ge0$, we have 
\begin{equation}
0 \le \mathcal{E}_{h,\mathrm{RE}}^{n}
\le C_0 \sum_{k=1}^M \mathcal{E}_{h,\mathrm{Trap, Kerr}}^{0,k},
\end{equation}
with 
\begin{align}\label{C0}
    C_0 = \sum_{k=1}^M \alpha_k^2, \quad \alpha_k = \prod_{\substack{m = 1 \\ m \ne k}}^M \frac{s_m}{s_m - s_k},
\quad s_k = 4^{-(k-1)},
\end{align}
depending only on $M$.  

\end{theorem}

\begin{proof}
For $n \geq 0, M > 1$, let the solution of the $(2,2M)$ trapezoidal FDTD schemes at time grid point $t^n$, computed with time step $\Delta t/2^{k-1}$ be denoted as  $U_{h}^{n,k} = [E_{h}^{n,k}, D_{h}^{n,k}, H_{h}^{n,k}, P_{h}^{n,k}, J_{h}^{n,k}]^T$. Let $u_h^{n,k}$, $1\leq k\leq M$, denote any solution component in $U_h^{n,k}$ and $u_{h,RE}^n$ denote the corresponding extrapolated solution on either the primal or dual grid. For $s \in \{0,\frac{1}{2}\}$, by the triangle inequality and the Cauchy--Schwarz inequality,
\[
\|u_{h,RE}^n\|_{s,h}
\le \sum_{k=1}^M |\alpha_k|\,\|u_h^{n,k}\|_{s,h}
\le \Bigl(\sum_{k=1}^M \alpha_k^2\Bigr)^{\!\frac12}
    \Bigl(\sum_{k=1}^M \|u_h^{n,k}\|_{s,h}^2\Bigr)^{\!\frac12}
\implies
\|u_{h,RE}^n\|_{s,h}^2
\le C_0\sum_{k=1}^M \|u_h^{n,k}\|_{s,h}^2.
\]
where $C_0$ is defined as in \eqref{C0}. We now apply this generic estimate to each extrapolated field at time level $n$. By construction,
\[
\begin{aligned}
H_{h,RE}^{n}
&= \sum_{k=1}^M \alpha_k\, H_h^{n,k}
\implies
\|H_{h,RE}^{n}\|_{\frac{1}{2},h}^2
\le C_0\sum_{k=1}^M \bigl\|H_h^{n,k}\bigr\|_{\frac{1}{2},h}^2,\\
E_{h,RE}^{n}
&= \sum_{k=1}^M \alpha_k\, E_h^{n,k}
\implies
\|E_{h,RE}^{n}\|_{0,h}^2
\le C_0\sum_{k=1}^M \bigl\|E_h^{n,k}\bigr\|_{0,h}^2,\\
P_{h,RE}^{n}
&= \sum_{k=1}^M \alpha_k\, P_h^{n,k}
\implies
\|P_{h,RE}^{n}\|_{0,h}^2
\le C_0\sum_{k=1}^M \bigl\|P_h^{n,k}\bigr\|_{0,h}^2,\\
J_{h,RE}^{n}
&= \sum_{k=1}^M \alpha_k\, J_h^{n,k}
\implies
\|J_{h,RE}^{n}\|_{0,h}^2
\le C_0\sum_{k=1}^M \bigl\|J_h^{n,k}\bigr\|_{0,h}^2,\\
R_{h,RE}^{n}
&= \sum_{k=1}^M \alpha_k\,\bigl(E_h^{n,k}\bigr)^2
\implies
\|R_{h,RE}^{n}\|_{0,h}^2\le C_0\sum_{k=1}^M 
        \bigl\|\bigl(E_h^{n,k}\bigr)^2\bigr\|_{0,h}^2
= C_0\sum_{k=1}^M \bigl\|E_h^{n,k}\bigr\|_{0,h}^4.
\end{aligned}
\]
Substituting these into \eqref{REEng}, we obtain
\begin{align*}
\mathcal{E}_{h,\mathrm{RE}}^{n} 
&= \frac{1}{2}\Biggl(
\mu_0 \| H_{h,RE}^{n} \|_{\frac{1}{2},h}^2 
+ \epsilon_0 \epsilon_\infty \|E_{h,RE}^{n}\|^2_{0,h} 
+ \frac{3\epsilon_0 a}{2}\bigl\|R_{h,RE}^{n}\bigr\|_{0,h}^2  
+ \frac{\epsilon_0 \omega_0^2}{\omega_p^2}\|P_{h,RE}^{n}\|^2_{0,h}
+ \frac{\epsilon_0}{\omega_p^2}\|J_{h,RE}^{n}\|^2_{0,h}
\Biggr) \\[0.5em]
&\le C_0 \sum_{k=1}^M
\frac{1}{2}\Biggl[
\mu_0 \bigl\|H_h^{n,k}\bigr\|_{\frac{1}{2},h}^2
+ \epsilon_0 \epsilon_\infty \bigl\|E_h^{n,k}\bigr\|^2_{0,h} 
+ \frac{3\epsilon_0 a}{2}\bigl\|E_h^{n,k}\bigr\|_{0,h}^4  
+ \frac{\epsilon_0 \omega_0^2}{\omega_p^2}\bigl\|P_h^{n,k}\bigr\|^2_{0,h}
+ \frac{\epsilon_0}{\omega_p^2}\bigl\|J_h^{n,k}\bigr\|^2_{0,h}
\Biggr] \\[0.5em]
&= C_0 \sum_{k=1}^M \mathcal{E}_{h,\mathrm{Trap, Kerr}}^{n,k}.
\end{align*}
Using equation \eqref{eq:discretenergyTrapKerr} for each $k=1,\dots,M$, we finally obtain
\begin{align*}
0 \leq \mathcal{E}_{h,\mathrm{RE}}^{n} 
&\leq C_0\sum_{k=1}^M \mathcal{E}_{h,\mathrm{Trap, Kerr}}^{n,k}
\leq C_0 \sum_{k=1}^M \mathcal{E}_{h,\mathrm{Trap, Kerr}}^{n-1,k}
\le C_0 \sum_{k=1}^M \mathcal{E}_{h,\mathrm{Trap, Kerr}}^{0,k}.
\end{align*}
\end{proof}

\begin{remark}
We note that while the energy bound established above does not imply
monotone energy decay, the numerical results in Fig.~\ref{fig:energy}(c)
show that the energy of the $(2M,2M)$ extrapolated trapezoidal FDTD schemes, as defined in \eqref{REEng}, decays
monotonically in time, with behavior nearly indistinguishable from that
of the $(2,2M)$ trapezoidal FDTD schemes. 

The Richardson extrapolation coefficients and the corresponding
values of $C_0$ for $M=1,\dots,5$ are given in Table~\ref{tab:richardson-C0}.
\begin{table}
\caption{The coefficient $C_0$ for varying values of $M$.}
\label{tab:richardson-C0}
\begin{equation}
\begin{array}{c|c|c}
M & \{\alpha_k\}_{k=1}^M & C_0 \\ [0.3em] \hline
1
& \{1\}
& 1
\\[0.3em]

2
& \left\{-\dfrac{1}{3},\dfrac{4}{3}\right\}
& 1.8889
\\[0.7em]

3
& \left\{\dfrac{1}{45},-\dfrac{4}{9},\dfrac{64}{45}\right\}
&  2.2207
\\[0.7em]

4
& \left\{-\dfrac{1}{2835},\dfrac{4}{135},
-\dfrac{64}{135},\dfrac{4096}{2835}\right\}
& 2.3131
\\[0.7em]

5
& \left\{\dfrac{1}{722925},-\dfrac{4}{8505},
\dfrac{64}{2025},-\dfrac{4096}{8505},
\dfrac{1048576}{722925}\right\}
& 2.3368
\\[0.7em]

\end{array}
\end{equation}
\end{table}
Thus, $C_0$ increases with $M$ for the values shown, but the increase
becomes progressively smaller as $M$ grows. Moreover, $C_0$ depends only
on $M$ and is independent of both $\Delta t$ and the time level $n$.

For the Maxwell--Kerr--Debye--Lorentz system, we also investigated
energy bounds for the corresponding $(2M,2M)$ extrapolated trapezoidal FDTD schemes constructed in Section \ref{2M2MKDL}. In contrast to the uniform in time energy bound for the limiting Maxwell-Kerr-Lorentz case, the nonlinear energy term for the Debye term, $\langle \chi_h,(E_h)^2\rangle_{0,h}$, requires a discrete
inverse estimate to produce a bound on a corresponding energy, which introduces a factor depending on the spatial
mesh size $h$. Consequently, the resulting bound on the energy is not uniform with
respect to the mesh parameters, and we do not present it here.
Nevertheless, the numerical experiments in Fig.~\ref{fig:energy}(c)
show monotone energy decay for the $(2M,2M)$ extrapolated trapezoidal FDTD schemes
with the Debye relaxation retained.
\end{remark}

\subsection{Asymptotic and Positivity Preservation Using Richardson Extrapolation}
In this section, we discuss what asymptotic preservation (AP) means for the high order $(2M,2M)$ extrapolated trapezoidal FDTD schemes based on Richardson extrapolation and we discuss positivity preservation (PP) in these schemes. We note that the extrapolation methods developed in Sections \ref{2M2M} and \ref{2M2MKDL} are different. More precisely, they use a different number of refinements: $M-1$ for the Kerr-Lorentz model and $2(M-1)$ for the Kerr-Debye-Lorentz model. Thus,  they have different coefficients $\alpha_k$ in the extrapolation and we cannot directly compare a limiting version of the schemes for the Maxwell Kerr-Debye-Lorentz model in Section \ref{2M2MKDL} to those constructed in Section \ref{2M2M} for the Maxwell Kerr-Lorentz model.  

Since, the $(2,2M)$ trapezoidal FDTD schemes for the Maxwell-Kerr-Debye-Lorentz system satisfy the asymptotic-preserving property, as $\varepsilon \to 0$, the update equation \eqref{eq:chidiscreteb} reduces to give $\chi_h^{n}\to (E_h^{n})^2$, so that each discrete solution $\chi_h^{n,k}$ for $k = 1,\ldots, 2M-1$, appearing in the extrapolation $\chi_{h,RE}^{n} = \sum_{k=1}^{2M-1} \tilde{\alpha}_k \chi_{h}^{n,k}$ in \eqref{RE2-general} for $\chi$ approaches the corresponding discrete solution $(E_h^{n,k})^2$ of the limiting Maxwell--Kerr--Lorentz system. Since the extrapolation coefficients $\tilde{\alpha}_k$ defined in \eqref{eq:general-alphas} are independent of $\varepsilon$, the limit $\varepsilon\to0$ may be taken term by term in \eqref{RE2-general} and we get $\chi_{h,RE}^{n} \rightarrow \sum_{k=1}^{2M-1} \tilde{\alpha}_k (E_{h}^{n,k})^2$. This is analogous to the quantity defined as $R^{n}_{h,RE}$ in \eqref{eq:Esquare} in Section \ref{sec:energybounds} which sums $M$ terms. If a corresponding $(2M,2M)$ extrapolated trapezoidal FDTD scheme based on Richardson extrapolation using $2(M-1)$ mesh refinements was constructed for the limiting Maxwell-Kerr-Lorentz model, $\chi^n_{h,RE} \rightarrow \tilde{R}^n_{h,RE}:=\sum_{k=1}^{2M-1} \tilde{\alpha}_k (E_{h}^{n,k})^2$ as $\varepsilon \rightarrow 0$. In this sense, the extrapolated methods defined in Section \ref{2M2MKDL} are asymptotic preserving. 

The positivity-preserving property of the underlying $(2,2M)$ schemes may not, in general, be inherited by the Richardson extrapolated $(2M,2M)$ scheme. By Theorem~\ref{eq:positivitydiscrete}, discrete susceptibilities computed on any mesh refinement satisfies $\chi_h^{n,k}\geq0$ provided the initial susceptibility is nonnegative. However, the extrapolated susceptibility in \eqref{RE2-general} is a linear combination  $\chi_{h,\mathrm{RE}}^n=\sum_{k=1}^{2M-1}\tilde{\alpha}_k\chi_h^{n,k}$, in which the coefficients $\tilde{\alpha}_k$ defined in \eqref{eq:general-alphas} are not all nonnegative for $M\geq2$. Consequently, even when $\chi_h^{n,k}\geq0$ for every underlying solution, the extrapolated susceptibility $\chi_{h,\mathrm{RE}}^n$ may not be nonnegative. Thus, positivity preservation cannot, in general, be guaranteed for the $(2M,2M)$ extrapolated schemes. Nevertheless, we observe positivity preservation of the extrapolated schemes in some of the numerical experiments presented in Section~\ref{sec:numerics}.

\section{Numerical Implementation}\label{sec:implementation}
In this section, we provide a detailed numerical implementation of the $(2,2M)$ methods developed in Sections~\ref{22MLF}, \ref{22MTraps}, and the $(2M,2M)$ methods developed in Sections \ref{2M2M} and \ref{2M2MKDL}.

\subsection{Implementation of $(2,2M)$ Leapfrog FDTD Methods}

In this section, we discuss the implementation of the fully discrete $(2,2M)$ leapfrog FDTD methods presented in \eqref{leapfrog} which are applied to a non-dimensional version of the systems \eqref{eq:kerr-debyelorentz1da} and \eqref{eq:kerr-debyelorentz1db}. All the coefficients appearing in the equations below are collected together and defined below in \eqref{coeff}.

First, we compute $H_{j+\frac{1}{2}}^{n+\frac{1}{2}}$, and $D^{n+1}_j$ from \eqref{eq:faradaydiscreteLF}, and \eqref{eq:amperediscreteLF}, respectively. Next, we solve \eqref{eq:densitydiscreteLF} for $P^{n+1}_j$ and substitute the result into \eqref{eq:lorentzdiscreteLF} to obtain an equation for $J^{n+1}_j$. Next, we solve for $\chi^{n+1}_j$ in \eqref{eq:chidiscreteLF}. Substituting $\chi^{n+1}_j$ and $P^{n+1}_j$ into \eqref{algcontlawLF} produces a discrete nonlinear equation for $E^{n+1}_j$. We define
\begin{align}
    A_{0,j}^n
    &= D_j^{n+1} - b_1 P_j^n - b_2 J_j^n
    - \frac{\Delta t a_5}{2} E_j^n, \quad
    A_{1,j}^n
    = b_0 + a a_0 \chi_j^n + a a_2(E_j^n)^2,
\end{align}
where $A_{0,j}^n$ and $A_{1,j}^n$ contain quantities that are known when solving for $E_j^{n+1}$. The nonlinear equation can then be written as the root finding problem
\begin{align}
     F\left(E^{n+1}_j\right)
     &= A_{0,j}^n
     - A_{1,j}^n E_j^{n+1}
     - a a_1\left(E_j^{n+1}\right)^3
     =0.
     \label{residualLF}
\end{align}
The derivative of the residual is
\begin{equation}
    F'\left(E_j^{n+1}\right)
    =
    -A_{1,j}^n
    -3a a_1\left(E_j^{n+1}\right)^2.
\end{equation}
Employing Newton's method, let $E_j^{n+1,(k)}$ denote the $k$th Newton iterate for $E_j^{n+1}$. We use $E_j^{n+1,(0)}=E_j^n$ as the initial iterate and compute
\begin{equation}\label{newtonLF}
     E_j^{n+1, (k+1)} = E_j^{n+1, (k)} - \frac{F\left(E_j^{n+1, (k)}\right)}{F'\left(E_j^{n+1, (k)}\right)}.
\end{equation}
If $K_j \in \mathbb{N}$ denotes the index of the final Newton iterate, the iteration is terminated when
\[
\left|
\frac{E^{n+1,(K_j)}_j-E^{n+1,(K_j-1)}_j}
{10^{-5}+|E^{n+1,(K_j-1)}_j|}
\right|<10^{-8}.
\]
The converged solution is then assigned as $E_j^{n+1}=E_j^{n+1,(K_j)}$ and is used in the subsequent updates of $J_j^{n+1}$, $P_j^{n+1}$, and $\chi_j^{n+1}$.

Assuming that the grid functions $H^{n-\frac{1}{2}}_{j+\frac{1}{2}}, D^n_j, E^n_j, J^n_j, P^n_j, \text{ and } \chi^n_j$ have been computed, the update equations for the leapfrog scheme proceed as follows:
\begin{subequations}
    \begin{align}
        H_{j+\frac{1}{2}}^{n+\frac{1}{2}} &= H_{j+\frac{1}{2}}^{n-\frac{1}{2}} + \Delta t \left( \mathcal{D}_{h}^{(2M)} E_h^n\right)_{j+\frac{1}{2}}, \\
        D^{n+1}_j &= D^{n}_j + \Delta t \left( \widetilde{\mathcal{D}}_{h}^{(2M)} H_h^{n+\frac{1}{2}} \right)_j,\\
        E_j^{n+1,(0)} &= E_j^n,\\
        E_j^{n+1, (k+1)} &= E_j^{n+1, (k)} - \frac{F\left(E_j^{n+1, (k)}\right)}{F'\left(E_j^{n+1, (k)}\right)}, 0\leq k\leq K_j-1\\
        E_j^{n+1} &= E_j^{n+1,(K_j)},\\
        J^{n+1}_j &= a_3 \,J^{n}_j + a_4 \, P^n_j + a_5 \left(E^{n+1}_j + E^n_j\right),\\
        P^{n+1}_j &= P_j^n + \frac{\Delta t}{2}\left(J^{n+1}_j + J^n_j\right), \\
        \chi^{n+1}_j &= a_0 \, \chi^{n}_j + a_1 (E^{n+1}_j)^2 + a_2 (E^n_j)^2.
    \end{align}
\end{subequations}
The coefficients in the preceding equations are explicitly given by 
\begin{align}
    a_0 &= \frac{e^{- \frac{\Delta t}{\varepsilon}}}{\frac{1}{2} + \frac{1}{2} e^{-\frac{\Delta t}{\varepsilon}} + \frac{\Delta t}{2\varepsilon} e^{-\frac{\Delta t}{2\varepsilon}}}, \qquad
    a_1 = \frac{\frac{1}{2} \left( 1 - e^{-\frac{\Delta t}{\varepsilon}} \right)}{\frac{1}{2} + \frac{1}{2} e^{-\frac{\Delta t}{\varepsilon}} + \frac{\Delta t}{2\varepsilon} e^{-\frac{\Delta t}{2\varepsilon}}}, \qquad
    a_2 = \frac{\frac{\Delta t}{2\varepsilon} e^{-\frac{\Delta t}{2\varepsilon}}}{\frac{1}{2} + \frac{1}{2} e^{-\frac{\Delta t}{\varepsilon}} + \frac{\Delta t}{2\varepsilon} e^{-\frac{\Delta t}{2\varepsilon}}} \notag \\
    a_3 &= \frac{4\tau - 2\Delta t - \omega_0^2 \Delta t^2 \tau}{4 \tau + 2\Delta t + \omega_0^2 \Delta t^2 \tau}, \;\; \qquad
    a_4 = -\frac{4\omega_0^2 \Delta t \tau}{4 \tau + 2\Delta t + \omega_0^2 \Delta t^2 \tau}, \qquad
    a_5 = \frac{2\omega_p^2\Delta t \tau }{4 \tau + 2\Delta t + \omega_0^2 \Delta t^2 \tau} \notag\\
    b_0 &= \epsilon_\infty + \frac{\Delta t a_5}{2}, \hspace{7em}
    b_1 = 1 + \frac{\Delta t a_4}{2}, \hspace{7em}
    b_2 = \frac{\Delta t}{2} + \frac{\Delta t a_3}{2}.
    \label{coeff}
\end{align}
In our simulations, for every $M$ considered, the Newton solver converged within at most four iterations.

\subsection{Implementation of $(2,2M)$ Trapezoidal FDTD Methods}
In this section, we discuss the implementation of the fully discrete $(2,2M)$ trapezoidal FDTD methods presented in \eqref{eq:discretelorentzT} which are applied to a non-dimensional version of the systems \eqref{eq:kerr-debyelorentz1da} and \eqref{eq:kerr-debyelorentz1db}. All the coefficients appearing in the equations below are collected together and defined in \eqref{coeff}.

From \eqref{eq:faradayb} and \eqref{eq:ampereb}, we obtain 
\begin{align}\label{eq:DeqnTrap}
    D_{j}^{n+1} &= D_{j}^{n} + \Delta t \left( \widetilde{\mathcal{D}}_{h}^{(2M)} H_h^n \right)_j + \frac{\Delta t^2}{4} \left(\widetilde{\mathcal{D}}_{h}^{(2M)} \mathcal{D}_{h}^{(2M)}\left(E_h^{n+1} + E_h^{n}\right)\right)_j,
\end{align}
with $\widetilde{\mathcal{D}}_{h}^{(2M)} \mathcal{D}_{h}^{(2M)} = \dfrac{\partial^2}{\partial x^2} + \mathcal{O}(h^{2M})$. We next solve \eqref{eq:densitydiscretebTP} for $P^{n+1}_j$ and substitute the result into \eqref{eq:lorentzdiscreteTP} to obtain an equation for $J^{n+1}_j$. We also solve for $\chi^{n+1}_j$ in \eqref{eq:chidiscreteb}. Substituting $\chi^{n+1}_j$ and $P^{n+1}_j$ into \eqref{algcontlawT}, and combining the result with \eqref{eq:DeqnTrap}, provides a discrete nonlinear system for the grid function $E^{n+1}_h$.

We define
\begin{align}
    A_{0,j}^n
    &= D_j^n
    + \Delta t \left(\widetilde{\mathcal{D}}_{h}^{(2M)}H_h^n\right)_j
    + \frac{\Delta t^2}{4}
    \left(\widetilde{\mathcal{D}}_{h}^{(2M)}
    \mathcal{D}_{h}^{(2M)}E_h^n\right)_j
    - b_1P_j^n
    - b_2J_j^n
    - \frac{\Delta t a_5}{2}E_j^n, \notag\\
    A_{1,j}^n
    &= b_0 + a a_0\chi_j^n + a a_2(E_j^n)^2,
\end{align}
where $A_{0,j}^n$ and $A_{1,j}^n$ contain quantities that are known when solving for the grid function $E_h^{n+1}$. The $j$th component of the nonlinear system for $E_h^{n+1}$ can then be written as the root finding problem
\begin{align}
    F_j\left(E_h^{n+1}\right)
    &=
    A_{0,j}^n
    + \frac{\Delta t^2}{4}
    \left(
    \widetilde{\mathcal{D}}_{h}^{(2M)}
    \mathcal{D}_{h}^{(2M)}E_h^{n+1}
    \right)_j
    - A_{1,j}^n E_j^{n+1}
    - a a_1\left(E_j^{n+1}\right)^3
    =0.
    \label{newtonstep}
\end{align}
The Jacobian matrix $JF\left(E_h^{n+1}\right)$ is defined componentwise by
\begin{align}\label{eq:jacobi}
    \left[JF\left(E_h^{n+1}\right)\right]_{ij}
    &=
    \begin{cases}
    \displaystyle
    \frac{\partial F_i\left(E_j^{n+1}\right)}
    {\partial E_j^{n+1}},
    & \text{if } |i-j|\le M, \\[1em]
    0,
    & \text{otherwise,}
    \end{cases}
\end{align}
and is banded with stencil width $2M+1$.

Employing Newton's method, let $E_h^{n+1,(k)}$ denote the $k$th Newton iterate for $E_h^{n+1}$. We use $E_h^{n+1,(0)}=E_h^n$ as the initial iterate and compute
\begin{align}\label{eq:Eupdate}
    E_h^{n+1,(k+1)}
    =
    E_h^{n+1,(k)}
    -
    \left(JF\left(E_h^{n+1,(k)}\right)\right)^{-1}
    F\left(E_h^{n+1,(k)}\right).
\end{align}
If $K \in \mathbb{N}$ denotes the index of the final Newton iterate, the iteration is terminated when
\begin{align}\label{stopcrit}
\max_j
\left|
\frac{
E_j^{n+1,(K)}-E_j^{n+1,(K-1)}
}{
10^{-5}+|E_j^{n+1,(K-1)}|
}
\right|
<10^{-8}.
\end{align}
The converged solution is then assigned as $E_h^{n+1}=E_h^{n+1,(K)}$ and is used in the subsequent updates of $H_h^{n+1}$, $D_h^{n+1}$, $J_h^{n+1}$, $P_h^{n+1}$, and $\chi_h^{n+1}$.

Assuming that the grid functions $E^n_h, H^{n}_h, D^n_h, J^n_h, P^n_h, \text{ and } \chi^n_h$ have been computed, the update equations for the scheme proceed as follows.
\begin{subequations}\label{imptrap}
    \begin{align}
        E_h^{n+1,(0)} &= E_h^n,\\
        E_h^{n+1,(k+1)}
        &= E_h^{n+1,(k)}
        - \left(JF\left(E_h^{n+1,(k)}\right)\right)^{-1}
        F\left(E_h^{n+1,(k)}\right), 0\leq k\leq K-1\\
        E_h^{n+1} &= E_h^{n+1,(K)},\\
        H_{j+\frac{1}{2}}^{n+1}
        &= H_{j+\frac{1}{2}}^{n}
        + \frac{\Delta t}{2}
        \left(
        \left(\mathcal{D}_{h}^{(2M)}E_h^{n+1}\right)_{j+\frac{1}{2}}
        +
        \left(\mathcal{D}_{h}^{(2M)}E_h^n\right)_{j+\frac{1}{2}}
        \right),\\
        D^{n+1}_j
        &= D_{j}^{n}
        + \frac{\Delta t}{2}
        \left(
        \left(\widetilde{\mathcal{D}}_{h}^{(2M)}H_h^{n+1}\right)_j
        +
        \left(\widetilde{\mathcal{D}}_{h}^{(2M)}H_h^n\right)_j
        \right),\\
        J^{n+1}_j
        &= a_3 J^{n}_j
        + a_4 P^n_j
        + a_5\left(E^{n+1}_j+E^n_j\right),\\
        P^{n+1}_j
        &= P_j^n
        + \frac{\Delta t}{2}
        \left(J^{n+1}_j+J^n_j\right),\\
        \chi^{n+1}_j
        &= a_0\chi^{n}_j
        + a_1(E^{n+1}_j)^2
        + a_2(E^n_j)^2.
    \end{align}
\end{subequations}
The coefficients are defined as in \eqref{coeff}. In our simulations, for every $M$ considered, the Newton solver converged within at most three iterations.

\subsection{Implementation of $(2M,2M)$ Extrapolated Trapezoidal FDTD Methods}

We now describe the implementation of the $(2M,2M)$ extrapolated trapezoidal FDTD schemes for both the limiting Maxwell-Kerr-Lorentz system (denoted by KL), and the Maxwell-Kerr-Debye-Lorentz system (denoted by KDL), developed in Sections \ref{2M2M}, and \ref{2M2MKDL}, respectively. 

For the $k$th refinement level, let $\Delta t_k=\Delta t/2^{k-1}$. Over the coarse time interval $[t^n,t^{n+1}]$, the $(2,2M)$ scheme with time step $\Delta t_k$ is advanced through $2^{k-1}$ consecutive substeps. We denote the corresponding substep times by $t_{k,\ell}^n=t^n+\ell\Delta t_k$, for $\ell=0,\ldots,2^{k-1}$, so that $t_{k,0}^n=t^n$ and $t_{k,2^{k-1}}^n=t^{n+1}$. Thus, all refinement levels reach the same physical time $t^{n+1}$.

At each refined substep, the nonlinear electric-field equation is solved using Newton's method. Let $F_{S,k}$ denote the nonlinear residual for the model $S\in\{\mathrm{KL},\mathrm{KDL}\}$ corresponding to the time step $\Delta t_k$, and let $E_{S,h}^{n,k,\ell+1,(r)}$ denote the $r$th Newton iterate for the electric field at the $(\ell+1)$st substep. The Newton iteration is
\begin{align}
E_{S,h}^{n,k,\ell+1,(r+1)}
&=
E_{S,h}^{n,k,\ell+1,(r)}
-
\left[
JF_{S,k}\left(
E_{S,h}^{n,k,\ell+1,(r)}
\right)
\right]^{-1}
F_{S,k}\left(
E_{S,h}^{n,k,\ell+1,(r)}
\right).
\label{eq:RE-Newton}
\end{align}
The initial iterate at each substep is taken to be the electric field from the preceding substep, $E_{S,h}^{n,k,\ell+1,(0)}=E_{S,h}^{n,k,\ell}$, and the same relative-increment stopping criterion \eqref{stopcrit} is used. If $K_{S,k,\ell+1}$ denotes the index of the final Newton iterate, we set $E_{S,h}^{n,k,\ell+1}=E_{S,h}^{n,k,\ell+1,(K_{S,k,\ell+1})}$. This converged electric field is used to update the remaining discrete fields at the $(\ell+1)$st substep, and the resulting solution becomes the initial state for the next refined substep.

For the Maxwell-Kerr-Debye-Lorentz system, $F_{\mathrm{KDL},k}$, and the Jacobian $JF_{\mathrm{KDL},k}$, is obtained from the residual \eqref{newtonstep}, and the definition \eqref{eq:jacobi}, respectively, by replacing $\Delta t$ with $\Delta t_k$. Accordingly, all coefficients in \eqref{coeff} that depend on the time step are recomputed using $\Delta t_k$. Once the Newton iteration has converged, the discrete $H$, $D$, $J$, $P$, and $\chi$ fields are updated using \eqref{imptrap}, again with $\Delta t$ replaced by $\Delta t_k$. After completing all $2^{k-1}$ substeps, the resulting solution at the common time $t^{n+1}$ is denoted by
\[
U_{\mathrm{KDL},h}^{n+1,k}
=
[E_{\mathrm{KDL},h}^{n+1,k},
D_{\mathrm{KDL},h}^{n+1,k},
H_{\mathrm{KDL},h}^{n+1,k},
P_{\mathrm{KDL},h}^{n+1,k},
J_{\mathrm{KDL},h}^{n+1,k},
\chi_{\mathrm{KDL},h}^{n+1,k}]^\top.
\]
For the limiting Maxwell-Kerr-Lorentz system, the corresponding nonlinear residual $F_{\mathrm{KL},k}$ and the Jacobian $JF_{\mathrm{KL},k}$ are obtained from the limiting form of \eqref{newtonstep}, and the corresponding revised definition \eqref{eq:jacobi}, respectively. Equivalently this can be done by taking $a_0=0$, $a_1=1$, and $a_2=0$ in the Debye coefficients and replacing $\Delta t$ by $\Delta t_k$. After the Newton iteration has converged, the discrete $H$, $D$, $J$, and $P$ fields are advanced using the limiting form of the $(2,2M)$ trapezoidal FDTD scheme. After completing all $2^{k-1}$ substeps, the resulting solution at $t^{n+1}$ is denoted by
\[
U_{\mathrm{KL},h}^{n+1,k}
=
[E_{\mathrm{KL},h}^{n+1,k},
D_{\mathrm{KL},h}^{n+1,k},
H_{\mathrm{KL},h}^{n+1,k},
P_{\mathrm{KL},h}^{n+1,k},
J_{\mathrm{KL},h}^{n+1,k}]^\top.
\]
For the Maxwell-Kerr-Lorentz system, we compute $M$ such approximations using $k=1,\ldots,M$. The $(2M,2M)$ solution is then obtained from the Richardson extrapolation
\begin{align}
U_{\mathrm{KL},h,RE}^{n+1}
&=
\sum_{k=1}^{M}
\alpha_k\,
U_{\mathrm{KL},h}^{n+1,k},\qquad
\alpha_k
=
\prod_{\substack{m = 1 \\ m \ne k}}^M
\frac{1}{1-4^{m-k}}.
\label{eq:implementation-RE-KL}
\end{align}
For the Maxwell-Kerr-Debye-Lorentz system, we instead compute $2M-1$ approximations using $k=1,\ldots,2M-1$. The $(2M,2M)$ solution is given by
\begin{align}
U_{\mathrm{KDL},h,RE}^{n+1}
&=
\sum_{k=1}^{2M-1}
\tilde{\alpha}_k\,
U_{\mathrm{KDL},h}^{n+1,k}, \qquad
\tilde{\alpha}_k
=
\frac{2^{k-1}}{2^{2M-1}-1}
\prod_{\substack{m=1\\m\ne k}}^{2M-1}
\frac{1}{1-2^{m-k}}.
\label{eq:implementation-RE-KDL}
\end{align}
In our simulations, for every $M$ considered, the Newton solver converged within at most four iterations.

\section{Numerical Simulations}\label{sec:numerics}
In this section, we report several numerical experiments to demonstrate the theoretical results developed so far, and to evaluate the performance of the proposed schemes. 
 We consider an extended form of the Maxwell--Kerr--Debye system with $N_\omega$ linear Lorentz responses, where $N_\omega \geq 1, N_\omega \in \mathbb{N}$.
\begin{subequations}\label{KDLi}
\begin{align}
\mu_0 \partial_t H &= \partial_x E, \\
\partial_t D &= \partial_x H,  \\
D &= \epsilon_0 \left( (\epsilon_\infty + a \chi) E + \sum_{k=1}^{N_\omega} P_k \right), \\
\partial_t \chi &= -\frac{1}{\varepsilon} (\chi - E^2), \\
\partial_t P_k &= J_k, \quad k = 1, 2, \ldots, N_\omega, \\
\partial_t J_k + \frac{1}{\tau_k} J_k + \omega_k^2 P_k &= (\omega_p^{(k)})^2 E, \quad k = 1, 2, \ldots, N_\omega.
\end{align}
\end{subequations}
Following work done in \cite{PengJCP2020}, we let $\bar{\tau}$ be the reference time scale, $c \bar{\tau}$ be the reference space scale, where $c = 1/\sqrt{\mu_0 \epsilon_0}$ is the speed of light, and let $\bar{E}$ be the reference electric field. By setting
\begin{align}\label{Para}
    t/\bar{\tau} \rightarrow t, \quad x / (c \bar{\tau}) &\rightarrow x, \quad \omega_k \bar{\tau} \rightarrow \omega_k, \quad (1/\tau_k) \bar{\tau} \rightarrow 1/\tau_k, \quad \bar{\tau} \omega_p^{(k)} \rightarrow \omega_p^{(k)}, \quad \frac{1}{\varepsilon} \bar{\tau} \rightarrow \frac{1}{\varepsilon}, \notag\\
    \left(H / \bar{E} \right) \sqrt{\mu_0 / \epsilon_0} &\rightarrow H, \quad D / (\epsilon_0 \bar{E}) \rightarrow D, \quad P / \bar{E} \rightarrow P, \left(J / \bar{E} \right) \bar{\tau} \rightarrow J, \quad
    E / \bar{E} \rightarrow E, \notag\\
    \chi / &\bar{E}^2 \rightarrow \chi, \quad
    a \bar{E}^2 \rightarrow a,
\end{align}
one obtains the dimensionless system,
\begin{subequations}\label{KDLaa}
\begin{align}
\partial_t H &= \partial_x E, \label{fardim}\\
\partial_t D &= \partial_x H, \label{ampdim}  \\
D &= (\epsilon_\infty + a \chi) E + \sum_{k=1}^{N_\omega} P_k,  \\
\partial_t \chi &= -\frac{1}{\varepsilon} (\chi - E^2),  \\
\partial_t P_k &= J_k, \quad k = 1, 2, \ldots, N_\omega, \\
\partial_t J_k + \frac{1}{\tau_k} J_k + \omega_k^2 P_k &= (\omega_p^{(k)})^2 E, \quad k = 1, 2, \ldots, N_\omega.\label{lordim}
\end{align}
\end{subequations}
For simplicity, we use the same notation for both the scaled and the original variables. 

In the following sections, we demonstrate the implementation of the proposed schemes using $N_\omega = 1$ in Sections \ref{sim1}, \ref{sim2}, \ref{soliton} and $N_\omega = 3$ in Section \ref{sim4}, with $M \in \mathbb{N}$.

\subsection{Accuracy test: Manufactured Solution}
\label{sim1}
We use manufactured solutions to demonstrate the accuracy of the numerical schemes presented in sections \ref{22MLF}, \ref{22MTraps}, \ref{2M2M} and \ref{2M2MKDL}. The exact solutions are taken from \cite{PengJCP2020} and are defined as follows: 
\begin{equation}\label{eq: exactsoln}
    \begin{aligned}
        H(x,t) &= \sin(x)\sin(t),\\
        \chi(x,t) &= \frac{1}{2}t - \frac{1}{4}\sin(2t) + \sin^2(x) + \frac{1}{2}, \\
        E(x,t) &=  \sqrt{\chi(x,t) + \varepsilon \sin^2(t)}, \\
        P(x,t) &= \cos(x)\sin(t),\\
        J(x,t) &= \cos(x)\cos(t),\\
        D(x,t) &= \left(\epsilon_\infty +\chi(x,t)\right)E(x,t) + P(x,t).
        \end{aligned}
\end{equation}
For $N_\omega = 1$, we add forcing functions to \eqref{fardim}, \eqref{ampdim} and \eqref{lordim}  to make \eqref{KDLaa} hold on $\Omega = [0, 2\pi]$, with periodic boundary conditions. 

For the $(2, 2M)$ schemes, the time step is chosen as $\Delta t = \nu \Delta x^{M}$ with $\nu = 0.9$ and $M = 1,2,3$. For the $(2M, 2M)$ schemes, the time step is chosen as $\Delta t = \nu \Delta x$ with $\nu = 0.9$. 

We list the relative $\ell^2$ and $\ell^\infty$ errors, along with the observed
orders of spatial accuracy, for the electric field $E$ in
Tables~\ref{tab:full-comparison} and~\ref{tab:implicit-44-66}. The errors at
the final time $T$ are computed by comparing the numerical solution with the
exact solution $E(t^N)$ and are given by
\begin{align}
\mathrm{Rel}\,\ell_2(h,T)
&:=
\frac{
\left(\sum_{\ell=0}^{J-1} h\,\bigl|E_\ell(t^N)-E_\ell^{N}\bigr|^2\right)^{1/2}
}{
\left(\sum_{\ell=0}^{J-1} h\,\bigl|E_\ell(t^N)\bigr|^2\right)^{1/2}
},  \\
\mathrm{Rel}\,\ell_\infty(h,T)
&:=
\frac{
\max_{0\le \ell \le J-1}\bigl|E_\ell(t^N)-E_\ell^{N}\bigr|
}{
\max_{0\le \ell \le J-1}\bigl|E_\ell(t^N)\bigr|
}.
\end{align}
In the above, we choose $T = 1$. 

In Tables~\ref{tab:full-comparison}--\ref{tab:implicit-44-66}, we present the $\ell_2$ and $\ell_\infty$ errors and convergence rates of the electric field for the proposed $(2,2M)$ and $(2M,2M)$ schemes for $\varepsilon=0$, $\varepsilon=10^{-12}$, and $\varepsilon=1$. Here, $\varepsilon=0$ corresponds to the limiting Maxwell-Kerr-Lorentz model, $\varepsilon=10^{-12}$ represents a stiff near-limit regime, and $\varepsilon=1$ represents a finite-relaxation regime. While only the errors in the electric field are shown, the remaining discrete fields $H$, $D$, $P$, $J$, and $\chi$ are expected to attain the same designed convergence order, since they are advanced as components of the same fully coupled scheme using discretizations of the corresponding temporal and spatial orders. From Table~\ref{tab:full-comparison}, we observe clear second-, fourth-, and sixth-order convergence for the $(2,2)$, $(2,4)$, and $(2,6)$ schemes, respectively, for all three values of $\varepsilon$.

\begin{table}
\centering
\caption{Convergence comparison of $(2,2M)$ leapfrog and trapezoidal FDTD schemes for discretizations $(2,2)$, $(2,4)$, and $(2,6)$, for various values of $\varepsilon$ and number of nodal points, $J$.}
\label{tab:full-comparison}
\resizebox{\textwidth}{!}{
\centering
\begin{tabular}{ccr
                rr rr
                rr rr}
\toprule
& & & \multicolumn{4}{c}{Leapfrog FDTD scheme} & \multicolumn{4}{c}{Trapezoidal FDTD scheme} \\
\cmidrule(r){4-7} \cmidrule(r){8-11}
Scheme & $\varepsilon$ & $J$
  & Rel $\ell_2$ error & Rate
  & Rel $\ell_\infty$ error & Rate
  & Rel $\ell_2$ error & Rate
  & Rel $\ell_\infty$ error & Rate \\
\midrule

\multirow{15}{*}{$(2,2)$}
& \multirow{5}{*}{$0$}
& 15  & 4.696e-03 & --    & 8.647e-03 & --    & 4.381e-03 & --    & 1.014e-02 & --    \\
& & 30  & 1.077e-03 & 2.124 & 2.154e-03 & 2.005 & 1.102e-03 & 1.992 & 2.437e-03 & 2.057 \\
& & 60  & 2.745e-04 & 1.972 & 5.589e-04 & 1.946 & 2.839e-04 & 1.956 & 5.995e-04 & 2.023 \\
& & 120 & 6.743e-05 & 2.026 & 1.370e-04 & 2.028 & 7.099e-05 & 2.000 & 1.495e-04 & 2.004 \\
& & 240 & 1.690e-05 & 1.997 & 3.450e-05 & 1.990 & 1.784e-05 & 1.992 & 3.723e-05 & 2.006 \\
\cmidrule{2-11}

& \multirow{5}{*}{$10^{-12}$}
& 15  & 4.696e-03 & --    & 8.647e-03 & --    & 4.381e-03 & --    & 1.014e-02 & --    \\
& & 30  & 1.077e-03 & 2.124 & 2.154e-03 & 2.005 & 1.102e-03 & 1.992 & 2.437e-03 & 2.057 \\
& & 60  & 2.745e-04 & 1.972 & 5.589e-04 & 1.946 & 2.839e-04 & 1.956 & 5.995e-04 & 2.023 \\
& & 120 & 6.743e-05 & 2.026 & 1.370e-04 & 2.028 & 7.099e-05 & 2.000 & 1.495e-04 & 2.004 \\
& & 240 & 1.690e-05 & 1.997 & 3.450e-05 & 1.990 & 1.784e-05 & 1.992 & 3.723e-05 & 2.006 \\
\cmidrule{2-11}

& \multirow{5}{*}{$1$}
& 15  & 2.820e-03 & --    & 3.390e-03 & --    & 3.118e-03 & --    & 5.086e-03 & --    \\
& & 30  & 6.819e-04 & 2.048 & 8.979e-04 & 1.917 & 7.906e-04 & 1.980 & 1.010e-03 & 2.332 \\
& & 60  & 1.441e-04 & 2.242 & 2.001e-04 & 2.166 & 1.774e-04 & 2.156 & 3.140e-04 & 1.686 \\
& & 120 & 3.530e-05 & 2.030 & 4.798e-05 & 2.060 & 4.435e-05 & 2.000 & 7.760e-05 & 2.017 \\
& & 240 & 8.622e-06 & 2.034 & 1.348e-05 & 1.832 & 1.100e-05 & 2.011 & 2.050e-05 & 1.920 \\

\midrule

\multirow{15}{*}{$(2,4)$}
& \multirow{5}{*}{$0$}
& 15  & 7.595e-04 & --    & 1.524e-03 & --    & 7.689e-04 & --    & 1.960e-03 & --    \\
& & 30  & 5.124e-05 & 3.890 & 1.184e-04 & 3.686 & 5.242e-05 & 3.875 & 1.145e-04 & 4.098 \\
& & 60  & 3.298e-06 & 3.958 & 7.756e-06 & 3.933 & 3.382e-06 & 3.954 & 7.418e-06 & 3.948 \\
& & 120 & 2.079e-07 & 3.988 & 4.904e-07 & 3.983 & 2.133e-07 & 3.987 & 4.681e-07 & 3.986 \\
& & 240 & 1.302e-08 & 3.997 & 3.074e-08 & 3.996 & 1.336e-08 & 3.997 & 2.933e-08 & 3.996 \\
\cmidrule{2-11}

& \multirow{5}{*}{$10^{-12}$}
& 15  & 7.595e-04 & --    & 1.524e-03 & --    & 7.689e-04 & --    & 1.960e-03 & --    \\
& & 30  & 5.124e-05 & 3.890 & 1.184e-04 & 3.686 & 5.242e-05 & 3.875 & 1.145e-04 & 4.098 \\
& & 60  & 3.298e-06 & 3.958 & 7.756e-06 & 3.933 & 3.382e-06 & 3.954 & 7.418e-06 & 3.948 \\
& & 120 & 2.079e-07 & 3.988 & 4.904e-07 & 3.983 & 2.133e-07 & 3.987 & 4.681e-07 & 3.986 \\
& & 240 & 1.302e-08 & 3.997 & 3.074e-08 & 3.996 & 1.336e-08 & 3.997 & 2.933e-08 & 3.996 \\
\cmidrule{2-11}

& \multirow{5}{*}{$1$}
& 15  & 3.903e-04 & --    & 6.704e-04 & --    & 4.248e-04 & --    & 5.138e-04 & --    \\
& & 30  & 2.219e-05 & 4.137 & 3.994e-05 & 4.069 & 2.452e-05 & 4.115 & 3.924e-05 & 3.711 \\
& & 60  & 1.403e-06 & 3.983 & 2.499e-06 & 3.998 & 1.556e-06 & 3.978 & 2.690e-06 & 3.866 \\
& & 120 & 8.827e-08 & 3.991 & 1.616e-07 & 3.951 & 9.795e-08 & 3.990 & 1.719e-07 & 3.968 \\
& & 240 & 5.527e-09 & 3.997 & 1.010e-08 & 4.000 & 6.133e-09 & 3.997 & 1.080e-08 & 3.992 \\

\midrule

\multirow{15}{*}{$(2,6)$}
& \multirow{5}{*}{$0$}
& 15  & 2.665e-04 & --    & 6.635e-04 & --    & 2.730e-04 & --    & 7.451e-04 & --    \\
& & 30  & 5.531e-06 & 5.591 & 1.078e-05 & 5.944 & 5.609e-06 & 5.605 & 1.051e-05 & 6.148 \\
& & 60  & 9.630e-08 & 5.844 & 1.971e-07 & 5.773 & 9.731e-08 & 5.849 & 1.922e-07 & 5.773 \\
& & 120 & 1.553e-09 & 5.955 & 3.156e-09 & 5.965 & 1.568e-09 & 5.956 & 3.079e-09 & 5.964 \\
& & 240 & 2.443e-11 & 5.990 & 4.932e-11 & 6.000 & 2.469e-11 & 5.989 & 4.871e-11 & 5.982 \\
\cmidrule{2-11}

& \multirow{5}{*}{$10^{-12}$}
& 15  & 2.665e-04 & --    & 6.635e-04 & --    & 2.730e-04 & --    & 7.451e-04 & --    \\
& & 30  & 5.531e-06 & 5.591 & 1.078e-05 & 5.944 & 5.609e-06 & 5.605 & 1.051e-05 & 6.148 \\
& & 60  & 9.630e-08 & 5.844 & 1.971e-07 & 5.773 & 9.731e-08 & 5.849 & 1.922e-07 & 5.773 \\
& & 120 & 1.553e-09 & 5.955 & 3.156e-09 & 5.964 & 1.568e-09 & 5.956 & 3.079e-09 & 5.964 \\
& & 240 & 2.447e-11 & 5.988 & 4.946e-11 & 5.996 & 2.473e-11 & 5.987 & 4.885e-11 & 5.978 \\
\cmidrule{2-11}

& \multirow{5}{*}{$1$}
& 15  & 8.257e-05 & --    & 1.580e-04 & --    & 8.619e-05 & --    & 1.333e-04 & --    \\
& & 30  & 2.006e-06 & 5.363 & 3.457e-06 & 5.514 & 2.044e-06 & 5.398 & 3.175e-06 & 5.392 \\
& & 60  & 3.521e-08 & 5.832 & 8.477e-08 & 5.350 & 3.576e-08 & 5.837 & 7.911e-08 & 5.327 \\
& & 120 & 5.740e-10 & 5.939 & 1.423e-09 & 5.897 & 5.822e-10 & 5.940 & 1.349e-09 & 5.874 \\
& & 240 & 9.014e-12 & 5.993 & 2.254e-11 & 5.980 & 9.201e-12 & 5.984 & 2.167e-11 & 5.960 \\

\bottomrule
\end{tabular}
}
\end{table}

Table~\ref{tab:implicit-44-66} shows the errors obtained with the $(2M,2M)$ extrapolated trapezoidal FDTD schemes. For $\varepsilon=0$, the discrete solution of the limiting Maxwell-Kerr-Lorentz model described in Section~\ref{2M2M} is used, while for $\varepsilon=10^{-12}$ and $\varepsilon=1$, the discrete solution of the Maxwell-Kerr-Debye-Lorentz model developed in Section~\ref{2M2MKDL} is applied. The $(4,4)$ method exhibits approximately fourth-order convergence in both the $\ell_2$ and $\ell_\infty$ norms for all three values of $\varepsilon$. Similarly, the $(6,6)$ method exhibits approximately sixth-order convergence, with the observed rates approaching six as the mesh is refined. The results for $\varepsilon=0$ and $\varepsilon=10^{-12}$ are nearly indistinguishable, showing that the Maxwell-Kerr-Debye-Lorentz discretization approaches the limiting Maxwell-Kerr-Lorentz discretization in the stiff-relaxation limit. For $\varepsilon=1$, the expected fourth- and sixth-order convergence rates are also recovered, demonstrating the high-order accuracy of the modified Richardson extrapolation when the nonlinear Debye relaxation is retained. In both cases, it is possible to obtain higher-order accuracy by simply increasing the value of M.

\begin{table}[H]
\setlength{\tabcolsep}{6pt}
\caption{Convergence of $(2M,2M)$ extrapolated trapezoidal FDTD schemes for discretizations $(4,4)$ and $(6,6)$, for various values of $\varepsilon$ and number of nodal points, $J$.}
\centering
\begin{tabular}{cccrrrr}
\toprule
Scheme & $\varepsilon$ & $J$
  & $\mathrm{Rel}\,\ell_2$ error & Rate
  & $\mathrm{Rel}\,\ell_\infty$ error & Rate \\
\midrule

\multirow{15}{*}{$(4,4)$}
& \multirow{5}{*}{$0$}
& 15  & 5.348e-04 & --    & 1.329e-03 & --    \\
& & 30  & 3.457e-05 & 3.951 & 6.757e-05 & 4.298 \\
& & 60  & 2.387e-06 & 3.856 & 4.922e-06 & 3.779 \\
& & 120 & 1.506e-07 & 3.987 & 3.104e-07 & 3.987 \\
& & 240 & 9.591e-09 & 3.973 & 1.997e-08 & 3.958 \\
\cmidrule{2-7}

& \multirow{5}{*}{$10^{-12}$}
& 15  & 5.351e-04 & --    & 1.314e-03 & --    \\
& & 30  & 3.464e-05 & 3.949 & 6.719e-05 & 4.289 \\
& & 60  & 2.390e-06 & 3.857 & 4.884e-06 & 3.782 \\
& & 120 & 1.508e-07 & 3.987 & 3.080e-07 & 3.987 \\
& & 240 & 9.598e-09 & 3.973 & 1.981e-08 & 3.959 \\
\cmidrule{2-7}

& \multirow{5}{*}{$1$}
& 15  & 2.355e-04 & --    & 4.383e-04 & --    \\
& & 30  & 1.561e-05 & 3.915 & 2.616e-05 & 4.067 \\
& & 60  & 1.136e-06 & 3.780 & 1.901e-06 & 3.782 \\
& & 120 & 7.147e-08 & 3.991 & 1.193e-07 & 3.995 \\
& & 240 & 4.630e-09 & 3.948 & 7.781e-09 & 3.938 \\

\midrule

\multirow{15}{*}{$(6,6)$}
& \multirow{5}{*}{$0$}
& 15  & 2.696e-04 & --    & 6.800e-04 & --    \\
& & 30  & 5.329e-06 & 5.661 & 8.856e-06 & 6.263 \\
& & 60  & 9.331e-08 & 5.836 & 1.665e-07 & 5.733 \\
& & 120 & 1.506e-09 & 5.953 & 2.674e-09 & 5.960 \\
& & 240 & 2.372e-11 & 5.988 & 4.207e-11 & 5.990 \\
\cmidrule{2-7}

& \multirow{5}{*}{$10^{-12}$}
& 15  & 2.695e-04 & --    & 6.800e-04 & --    \\
& & 30  & 5.328e-06 & 5.661 & 8.856e-06 & 6.263 \\
& & 60  & 9.331e-08 & 5.835 & 1.665e-07 & 5.733 \\
& & 120 & 1.506e-09 & 5.953 & 2.674e-09 & 5.960 \\
& & 240 & 2.373e-11 & 5.988 & 4.220e-11 & 5.986 \\
\cmidrule{2-7}

& \multirow{5}{*}{$1$}
& 15  & 8.177e-05 & --    & 1.660e-04 & --    \\
& & 30  & 1.919e-06 & 5.413 & 2.749e-06 & 5.916 \\
& & 60  & 3.398e-08 & 5.819 & 7.400e-08 & 5.215 \\
& & 120 & 5.551e-10 & 5.936 & 1.256e-09 & 5.881 \\
& & 240 & 8.776e-12 & 5.983 & 2.029e-11 & 5.952 \\

\bottomrule
\end{tabular}
\label{tab:implicit-44-66}
\end{table}
The convergence rates for the remaining variables $H$, $P$, $J$, and $D$, as well as the nonlinear susceptibility $\chi$ for $\varepsilon>0$, exhibit similar behavior, but are omitted for brevity. In Table~\ref{tab:cn_22_eps}, we consider $\varepsilon=10^{-3}$ and $\nu=0.9$, with refinement taking the time step from $\Delta t\gg\varepsilon$ through $\Delta t=\mathcal{O}(\varepsilon)$ and eventually to $\Delta t\ll\varepsilon$. The expected second-order convergence deteriorates as $\Delta t=\mathcal{O}(\varepsilon)$. With further refinement, so that $\Delta t\ll \varepsilon$, the convergence rate approaches second-order, recovering the expected accuracy of the $(2,2)$ trapezoidal scheme.

\begin{table}[H]
\centering
\caption{Convergence of Trapezoidal scheme for the $(2,2)$ discretization with $\varepsilon = 10^{-3}$, $\nu=0.9$.}
\label{tab:cn_22_eps}
\renewcommand{\arraystretch}{1.15}
\begin{tabular}{c c c c c c c}
\toprule
$\varepsilon$ & $J$ & $\Delta t$
& Rel $\ell_2$ error & Rate
& Rel $\ell_\infty$ error & Rate \\
\midrule

\multirow{12}{*}{$10^{-3}$}
& 15    & $3.769911\times10^{-1}$ & $4.484\times10^{-3}$ & --     & $1.028\times10^{-2}$ & --     \\
& 30    & $1.884956\times10^{-1}$ & $1.211\times10^{-3}$ & 1.888  & $2.582\times10^{-3}$ & 1.993  \\
& 60    & $9.424778\times10^{-2}$ & $4.041\times10^{-4}$ & 1.584  & $7.382\times10^{-4}$ & 1.807  \\
& 120   & $4.712389\times10^{-2}$ & $2.061\times10^{-4}$ & 0.971  & $2.886\times10^{-4}$ & 1.355  \\
& 240   & $2.356194\times10^{-2}$ & $1.596\times10^{-4}$ & 0.369  & $1.734\times10^{-4}$ & 0.735  \\
& 480   & $1.178097\times10^{-2}$ & $9.507\times10^{-5}$ & 0.748  & $9.468\times10^{-5}$ & 0.873  \\
& 960   & $5.890486\times10^{-3}$ & $6.003\times10^{-5}$ & 0.663  & $5.472\times10^{-5}$ & 0.791  \\
& 1920  & $2.945243\times10^{-3}$ & $6.034\times10^{-5}$ & -0.007 & $5.587\times10^{-5}$ & -0.030 \\
& 3840  & $1.472622\times10^{-3}$ & $2.340\times10^{-5}$ & 1.366  & $2.171\times10^{-5}$ & 1.364  \\
& 7680  & $7.363108\times10^{-4}$ & $7.000\times10^{-6}$ & 1.741  & $6.497\times10^{-6}$ & 1.740  \\
& 15360 & $3.681554\times10^{-4}$ & $1.898\times10^{-6}$ & 1.883  & $1.762\times10^{-6}$ & 1.882  \\
& 30720 & $1.840777\times10^{-4}$ & $4.933\times10^{-7}$ & 1.944  & $4.580\times10^{-7}$ & 1.944  \\

\bottomrule
\end{tabular}
\end{table}

\subsection{Energy Decay}
\label{sim2}
For the purpose of illustrating the energy decay property of the proposed schemes, we use the following parameter values  \cite{sakkaplangkul2017multiscale}: 
\begin{align*}
    \epsilon_0 &= 8.854 \times 10^{-12}, \quad \epsilon_\infty = 2.25, \quad \epsilon_s = 5.25, \quad  \omega_0 = 5.84, \quad \omega_p = \omega_0\sqrt{\epsilon_s - \epsilon_\infty},\\
    \mu_0 &= 1.2566 \times 10^{-6} , \quad 1/\tau = 1.168 \times 10^{-5} , \quad a = 1\times 10^{-2}, \quad \varepsilon = 1\times 10^{-6}
\end{align*}
All fields are initially set to zero. We impose on the left boundary \(x=0\) an incident electric field of the form $ E(0,t) = E_0\, g(ft), \, t \ge 0,$ where $g$ is the \textit{Ricker pulse} defined by \cite{sakkaplangkul2017multiscale}
\begin{align*}
    g(s) =
\begin{cases}
\displaystyle
\frac{2k}{11} \left(2k(s - 1.35)^2 - 1\right) e^{-k(s - 1.35)^2}, & \text{if } s < 3.48/f \\
0, & \text{if } s \geq 3.48/f
\end{cases}
\end{align*}
with $k = \left(\pi/{1.31}\right)^2, \,
f = 4 \times 10^7 \ \text{Hz}, \,
E_0 = 4.0 \times 10^6$.

  \begin{figure}[H]
    \centering
    \includegraphics[width=.5\linewidth]{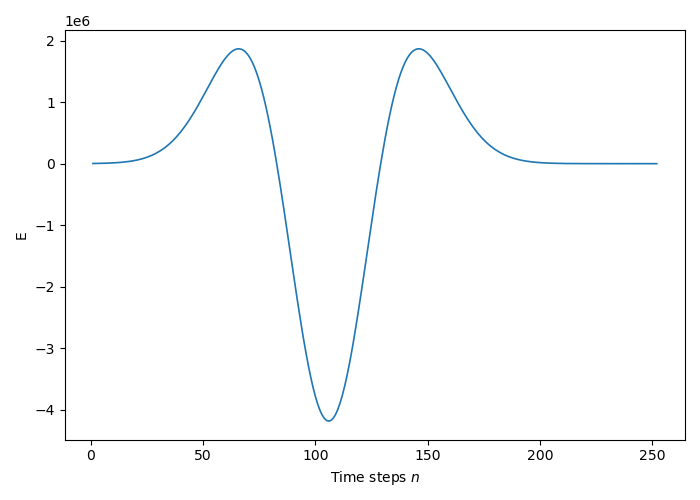}
    \caption{Time trace of electric field at $x=0$}
  \end{figure}

We take $\Delta x = 1/10$ and choose the time step as $\Delta t = \nu \Delta x$, where the CFL number is set to $\nu = 2/\pi$. This is the limiting value in the CFL condition \eqref{eq: cfl}, as derived in \cite{bokil2012analysis}, for which the proposed leapfrog schemes remain stable. While the trapezoidal FDTD schemes are unconditionally stable, we adopt the same time-stepping condition for consistency and to facilitate a fair comparison with the leapfrog schemes.

In Fig.~\ref{fig:energy}, (a) shows the time trace of the electric field $E$ at $x=0$. The plots (b) and (c) show the time evolution of the discrete energy for the leapfrog and trapezoidal schemes, respectively. In both cases, the discrete energy decays monotonically in time, in agreement with the theoretical results. Moreover, the energy curves for the different approximation orders are nearly indistinguishable within each scheme family, indicating that all methods capture the same dissipative behavior.
\begin{figure}[H]
  \centering

  \begin{subfigure}[t]{0.45\linewidth}
    \centering
    \includegraphics[width=\linewidth]{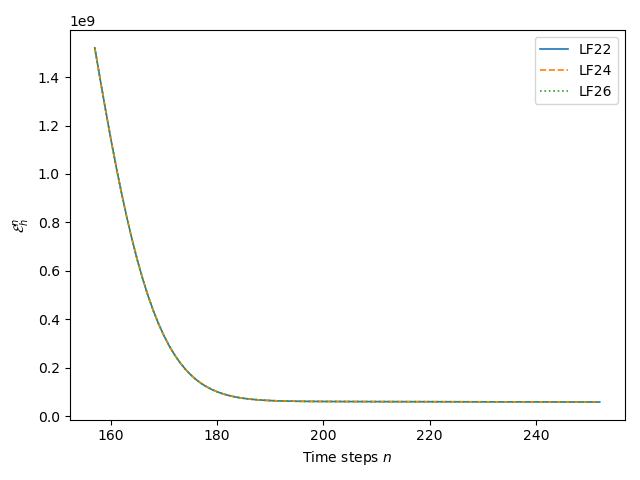}
    \caption{Leapfrog FDTD Schemes}
  \end{subfigure}
  \hfill
  \begin{subfigure}[t]{0.45\linewidth}
    \centering
    \includegraphics[width=\linewidth]{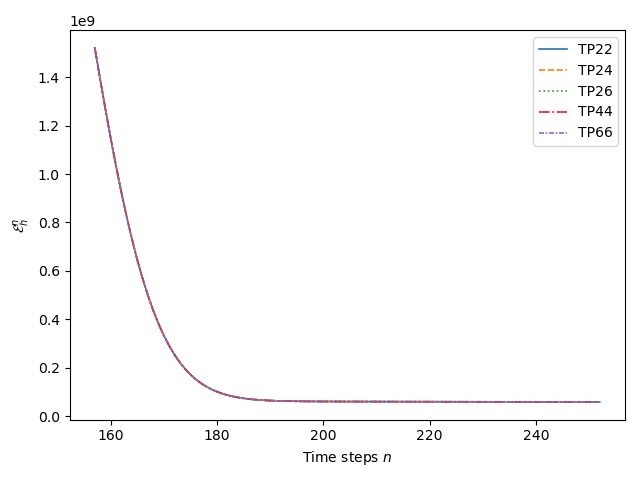}
    \caption{Trapezoidal \& extrapolated trapezoidal FDTD Schemes}
  \end{subfigure}

  \caption{Discrete energy decay of the leapfrog and trapezoidal FDTD schemes}
  \label{fig:energy}
\end{figure}

\subsection{Soliton-like propagation}\label{soliton}
We consider a soliton-like propagation problem with setup similar as considered in \cite{PengJCP2020}. The parameters are taken as 
\begin{align*}
    \epsilon_0 = 1, &\quad \mu_0 = 1, \quad \epsilon_\infty = 2.25, \quad \epsilon_s = 5.25, \quad  \omega_0 = 5.84, \quad \omega_p = \omega_0\sqrt{\epsilon_s - \epsilon_\infty},\\
    &\quad 1/\tau = 1.168 \times 10^{-5} , \quad a = 0.0504, \quad c = 1/\sqrt{\mu_0\epsilon_0\epsilon_\infty}.
\end{align*}
Initially, all the fields are zero. From the left boundary $x = 0$, a sech signal will enter the domain, 
\begin{align}
    E(x = 0, t) = \sech(t-20)\cos(\Omega_0 t), \quad \Omega_0 = 12.57.
\end{align}
As shown in \cite{gilles2000comparison}, the left boundary condition for the magnetic field can be obtained from Maxwell's equations and the inverse Fourier transform of $H$ as 
\begin{align}
    H(0,t) \approx \frac12\left(\sum_{m=0}^{2}\frac{(-i)^m}{m!}\,
    \left(\frac{1}{Z}\right)^{(m)}(\omega_c)\Bigg|_{\omega_c = \Omega_0}\,f^{(m)}(t)\right)e^{i\Omega_0 t} +\text{c.c.}
\end{align}
where c.c stands for the complex conjugate of the first term, $f^{(m)}$ is the m-th derivative of $f(t)$ and $(1/Z)^{(m)}$ is the m-th derivative of $(1/Z)$ with respect to $\omega$. 
In our simulation, we take the computational domain as $\Omega = [0,50]$, and the final time as $T = 40$ or $T = 80$. With these choices, the signal will not reach the right boundary by the final time, thus zero boundary conditions are used for $E$ and $H$ at $x = 50$. 

We use the one-sided finite difference approximations derived in~\cite{yefet2001staggered}, which are at least fourth-order accurate, in order to preserve the overall fourth-order spatial accuracy of the numerical scheme. For a generic field variable \( V \), these approximations are given by:
\begin{subequations}\label{fourthonesided}
    \begin{align}
\left. \frac{\partial V}{\partial x} \right|_{ \frac{1}{2} } &= \frac{1}{24\Delta x} \left( -22V_0 + 17V_1 + 9V_2 - 5V_3 + V_4 \right),  \\
\left. \frac{\partial V}{\partial x} \right|_{1} &= \frac{1}{24\Delta x} \left( -23V_{\frac{1}{2}} + 21V_{\frac{3}{2}} + 3V_{\frac{5}{2}} - V_{\frac{7}{2}} \right),  \\
\left. \frac{\partial V}{\partial x} \right|_{J-1} &= \frac{1}{24\Delta x} \left( 23V_{J-\frac{1}{2}} - 21V_{J-\frac{3}{2}} - 3V_{J-\frac{5}{2}} + V_{J-\frac{7}{2}} \right),  \\
\left. \frac{\partial V}{\partial x} \right|_{J-\frac{1}{2}} &= \frac{1}{24\Delta x} \left( 22V_J - 17V_{J-1} - 9V_{J-2} + 5V_{J-3} - V_{J-4} \right). 
\end{align}

\end{subequations}
For \( M > 3 \), appropriate \( 2M \)-th order one-sided approximations must be defined at the boundaries to maintain the overall \( 2M \)-th order spatial accuracy. However, since the development of high-order boundary conditions is not the primary focus of this work, we restrict our attention to the cases \( M = 1 \) and \( M = 2 \) in this paper.

We present numerical results for the soliton-like propagation obtained using the $(2,2M)$ and $(2M,2M)$ FDTD methods, with $M=1,2$, at time $T= 40$ and $T = 80$ in Fig. \ref{fig:solcomp} and \ref{fig:solb}. We use this example to illustrate the AP and PP properties, to compare the performance of the different schemes, and to examine the influence of the finite relaxation time on the soliton dynamics. In our simulations, we first consider a coarse mesh with $\Delta x = 1/50$ and choose the time step as $\Delta t = \nu \Delta x/c$, where $\nu = 2/\pi$, to emphasize the importance of high-order schemes. We then consider a finer mesh with $\Delta x = 1/250$ and use the same time-stepping relation with $\nu = 0.1$.

\begin{figure}[H]
    \centering
     \includegraphics[scale = 0.45]{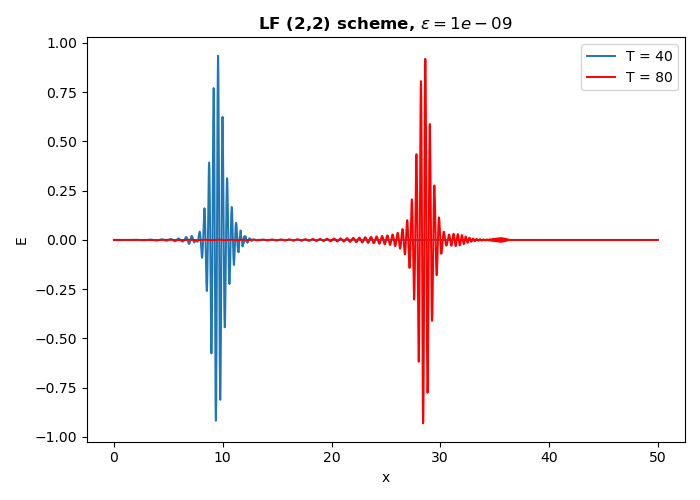}
     \includegraphics[scale = 0.45]{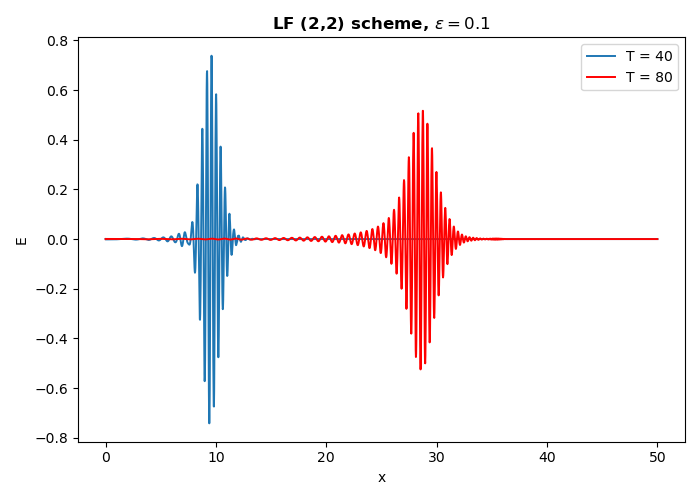}
     \includegraphics[scale = 0.45]{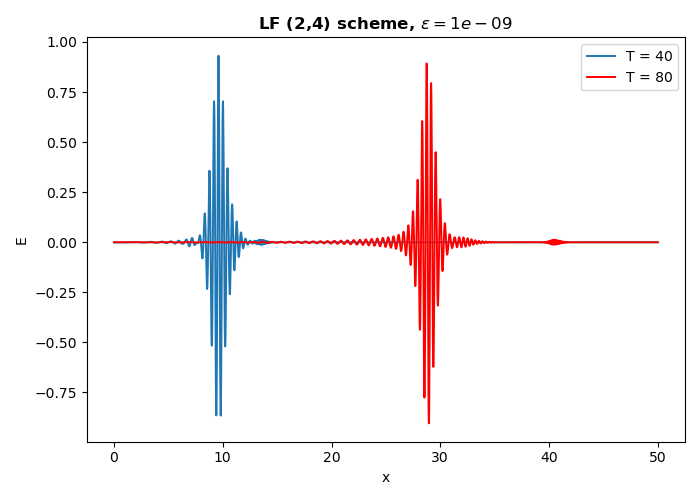}
     \includegraphics[scale = 0.45]{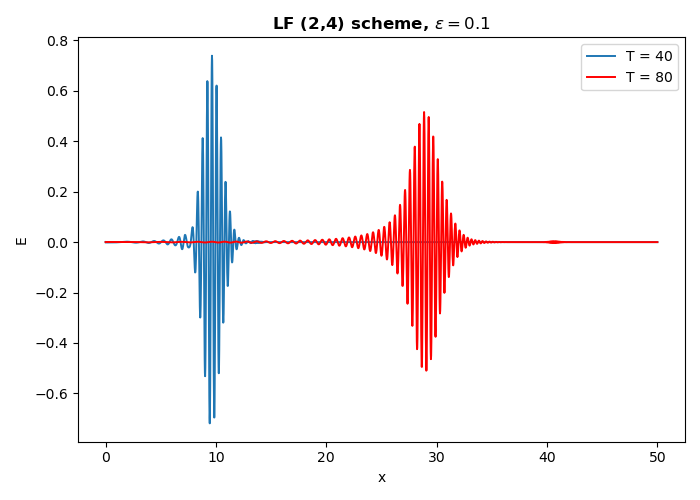}
    \caption{The electric field $E$ in the soliton-like propagation at $T = 40$ and $T=80$ computed using the $(2,2M)$ leapfrog FDTD schemes for $M = 1,2$.}
    \label{fig:solcomp}
\end{figure}
The phenomenon of third-harmonic generation leads to the formation of a daughter pulse that separates from the main soliton pulse and propagates ahead of it with a significantly lower peak amplitude. In Figures~\ref{fig:solcomp} and~\ref{fig:solb}, for the near-instantaneous case $\varepsilon = 10^{-9}$, the main pulse translates almost rigidly from $T=40$ to $T=80$, preserving both its peak amplitude and width. A small leading pulse is observed ahead of the main crest at $T=80$, consistent with Kerr-induced harmonic radiation \cite{Akhmediev1995, Skryabin2010, boyd2019nonlinear} and with earlier results in the instantaneous limit ($\varepsilon=0$) \cite{gilles2000comparison, agrawal2019nonlinear, PengJCP2020, bokil2018high, sorensen2002solitary}. For finite relaxation ($\varepsilon = 0.1$), the main pulse broadens and its peak decreases, while the daughter pulse is strongly suppressed, in agreement with \cite{PengJCP2020}.

\begin{figure}[H]
    \centering
     \includegraphics[scale = 0.45]{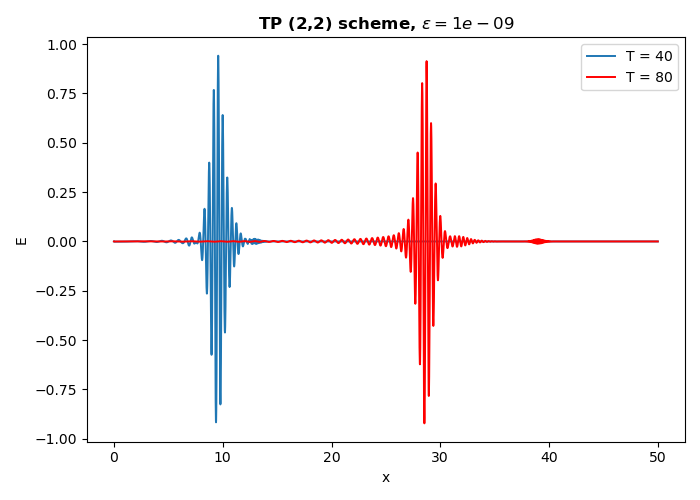}
     \includegraphics[scale = 0.45]{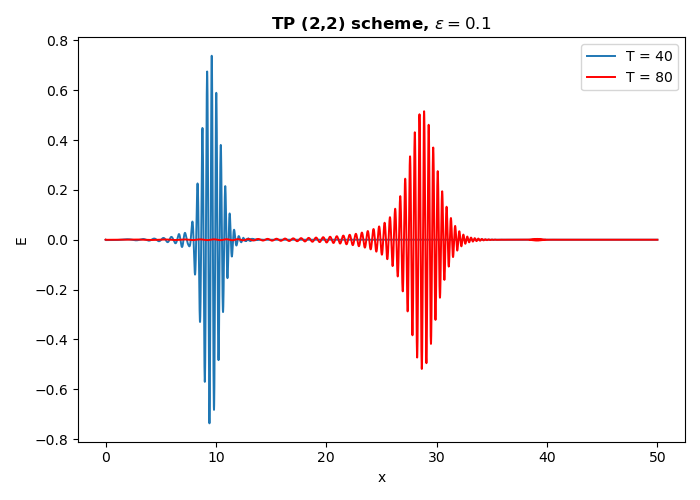}
     \includegraphics[scale = 0.45]{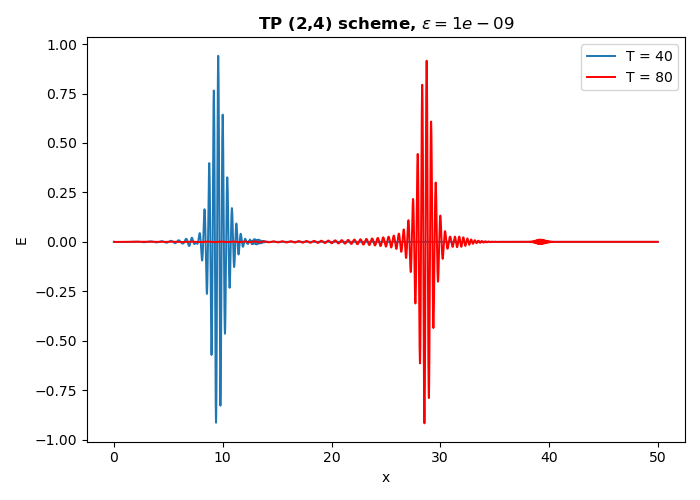}
     \includegraphics[scale = 0.45]{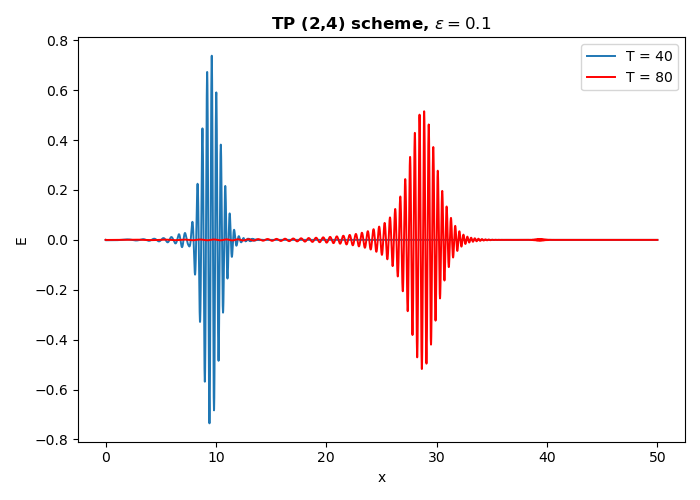}
      \includegraphics[scale = 0.45]{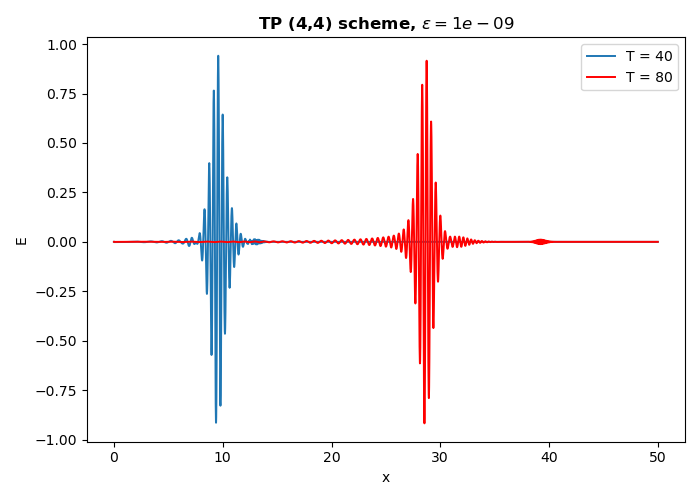}
     \includegraphics[scale = 0.45]{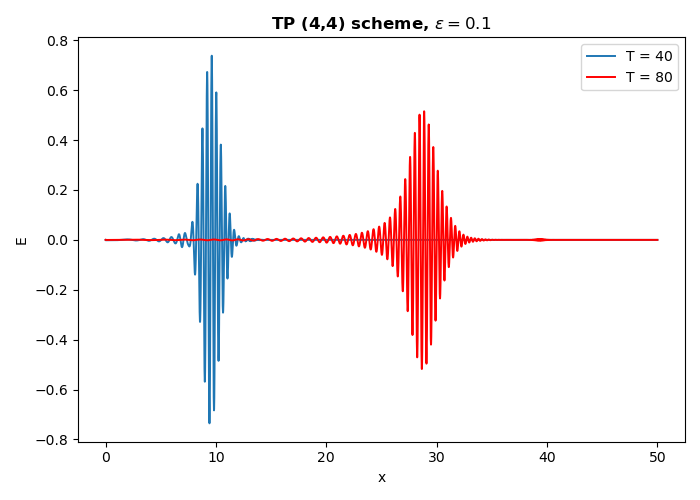}
    \caption{The electric field $E$ in the soliton-like propagation at $T = 40$ and $T=80$ computed using the $(2,2M)$ trapezoidal FDTD schemes and the $(2M,2M)$ extrapolated trapezoidal FDTD schemes}
    \label{fig:solb}
\end{figure}
Figure~\ref{fig:solcomp} further shows that the daughter pulse is sensitive to the order of accuracy of the $(2,2M)$ leapfrog FDTD schemes. For $\varepsilon=10^{-9}$, this pulse appears farther along the propagation direction in the $(2,4)$ scheme than in $(2,2)$, which is due to the stronger dispersive numerical errors in lower-order schemes \cite{bokil2012analysis, FinkelsteinKastner2007}. For $\varepsilon=0.1$, the daughter pulse is not visibly resolved by the $(2,2)$ scheme but remains visible in the $(2,4)$ simulation. When the mesh is refined for the lower-order scheme, for example to $\Delta x = 1/250$, the $(2,2)$ result becomes qualitatively similar to the $(2,4)$ result. The corresponding refined-mesh plots are omitted for brevity, since they mainly confirm the balance between computational cost and accuracy.

Furthermore, additional simulations for intermediate values of $\varepsilon$, not included here, suggest that visible relaxation effects begin near $\varepsilon \approx 4\times 10^{-3}$. Below this threshold, the solution remains essentially indistinguishable from the Kerr limit, whereas above it the pulse progressively departs from Kerr-like solitary-wave behavior. Moreover, the solutions produced by the leapfrog, trapezoidal and extrapolated trapezoidal FDTD schemes, as shown in Fig.~\ref{fig:solb}, are visually very similar for all relaxation times considered.

Finally, we note that no negative values of $\chi_h^n$ are observed in any of the simulations for all schemes, confirming the PP property of the methods. Numerically, the discrete solution $\chi_h^n$ is nonnegative if $\chi_j^n > -1 \times 10^{-14}$ for all grid points in the simulations.

\subsection{Sech signal in fused bulk silica}
\label{sim4}
When ultrashort light pulses propagate in a nonlinear medium, they may experience self-steepening, a nonlinear distortion in which the pulse develops a shock-like sharp edge, typically on the trailing side. This behavior arises from the intensity dependence of the propagation velocity in a nonlinear medium and can lead to optical shock formation \cite{boyd2019nonlinear,Osgood2009}.

We consider an example of a signal propagating through fused bulk silica, following a setup similar to that described in \cite{gilles1999electromagnetic}. Our simulation is based on the dimensionless form for the Kerr-Debye-Lorentz model \eqref{KDLaa} with three Lorentz poles ($N_\omega = 3$). 
\begin{figure}[H]
    \centering
    \includegraphics[width=0.5\linewidth]{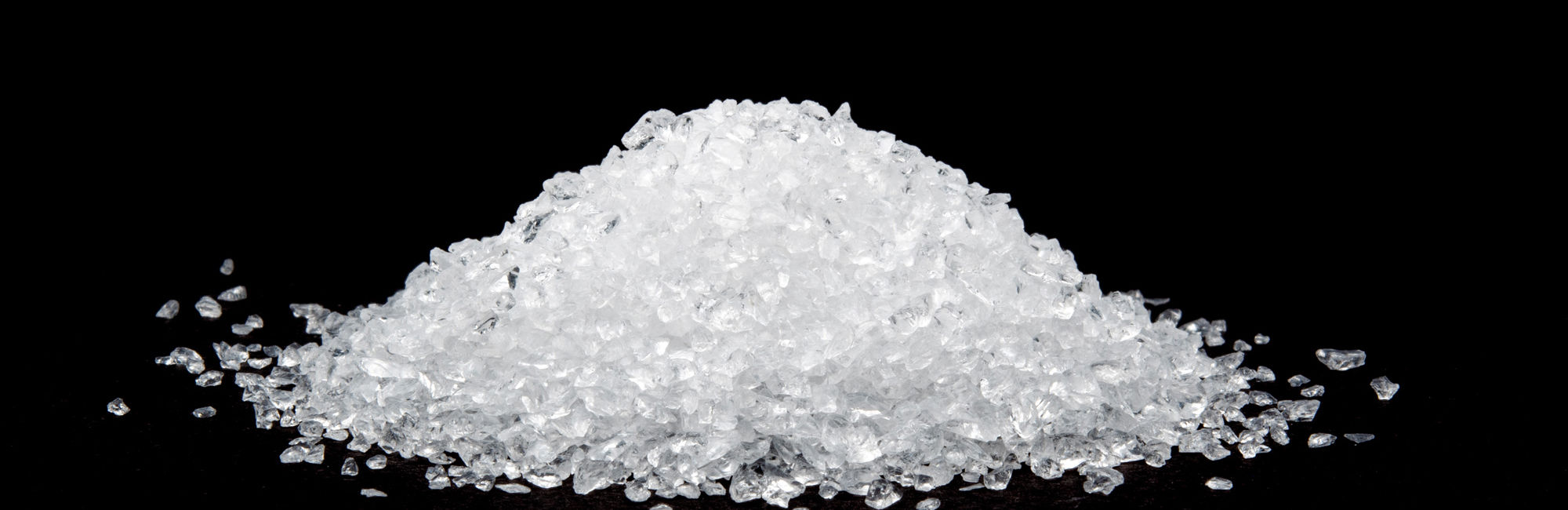}
    \caption{Fused silica, a high-purity amorphous form of silica produced by melting high-purity quartz and rapidly cooling the melt to prevent crystallization. Image reproduced from Elkem \cite{elkem-fused-silica}.}
    \label{fig:silica}
\end{figure}
Prior to non-dimensionalization, the physical parameters in equation \eqref{KDLi} are given by  
\[
\epsilon_\infty = 1, \quad \mu_0 = 4\pi \times 10^{-7} \, \text{H/m}, \quad \epsilon_0 = \frac{1}{\mu_0 c^2}, \quad \tau_1 = \tau_2 = \tau_3 = 10 \, \mu\text{s},
\]
\[
\omega_k = \frac{2\pi c}{\lambda_k}, \quad \omega_p^{(k)} = \sqrt{\beta_k} \, \omega_k, \quad k = 1, 2, 3,
\]
with
\[
\beta_1 = 0.7, \quad \beta_2 = 0.4, \quad \beta_3 = 0.9, \quad 
\lambda_1 = 0.068 \, \mu\text{m}, \quad \lambda_2 = 0.116 \, \mu\text{m}, \quad \lambda_3 = 9.9 \, \mu\text{m}.
\]
Here, \( c \) denotes the speed of light in vacuum. The system is initially quiescent, with all fields identically zero. A \textit{sech}-profiled input signal is imposed at the left boundary \( x = 0 \):
\[
E(x = 0, t) = E_0 \, \text{sech} \left( \frac{t - 50 \, \text{fs}}{t_0} \right) \cos(\Omega_0 t), 
\]
where \( t_0 = 5 \, \text{fs} \), \( \Omega_0 = 1520 \, \text{THz} \), and the amplitude \( E_0 \) is chosen to satisfy \( a E_0^2 = 0.06 \). The final time of the simulation is taken as \( T = 40\pi/\Omega_0 \), and the spatial domain is defined by the interval \([0, 15 c t_0]\). The relaxation time parameter \( \varepsilon \) is varied over the values  
\( \varepsilon = 0.5 \, \text{fs}, \, 0.1 \, \text{fs}, \, 10^{-4} \, \text{fs} \). For the purpose of obtaining a non-dimensional formulation, we employ the reference scales \( \bar{t} = 5 \, \text{fs} = 5 \times 10^{-15} \, \text{s} \) and \( \bar{E} = E_0 \), consistent with the normalization used in equation \eqref{Para}.

The numerical boundary condition for the magnetic field \( H \) at \( x = 0 \) is enforced analogously to the approach described in Section~\ref{soliton}. Since the signal remains localized and does not reach the right boundary within the simulation time window, zero boundary conditions are imposed there for both \( E \) and \( H \). One-sided fourth-order finite difference approximations, as defined in equation \eqref{fourthonesided}, are employed near the boundary to compute the solution with high accuracy. With the physical parameters specified above, the nonlinear effect is dominant in the model (see \cite{gilles1999electromagnetic}). 

We present numerical results for the propagation of a \textit{sech}-shaped signal through fused bulk silica using the proposed numerical methods. Using this example, we want to compare the different schemes developed in this paper and investigate the influence of the finite relaxation time $\varepsilon$. We also want to demonstrate the third and fifth harmonic generation due to the cubic nonlinearity inherent in the model. To this end, we analyze the numerical solutions obtained using the $(2,2M)$ leapfrog FDTD schemes in the frequency domain by applying the discrete Fourier transform to the computed electric field \( E \) over a duration of 20 periods (i.e., over the interval \([0, T]\)) at spatial locations \( x = 2ct_0 \), \( x = 3ct_0 \), and \( x = 4ct_0 \), respectively. In our simulations, we choose the spatial step size as $\Delta x = 1/250$ and time step as $\Delta t = \nu \Delta x$, where $\nu = 0.1$.

\begin{figure}[H]
    \centering
     \includegraphics[scale = 0.45]{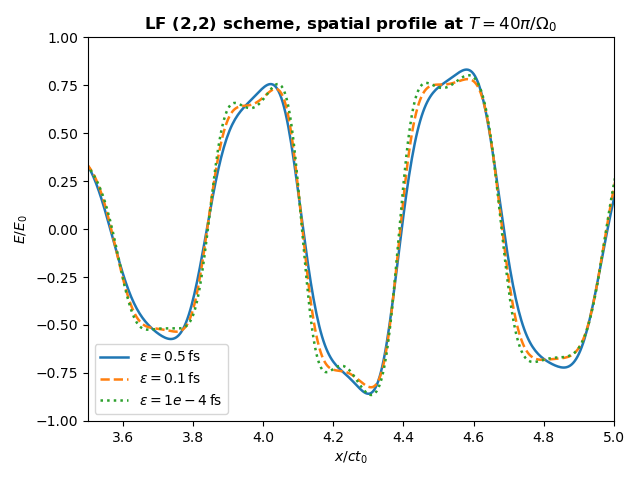}
     \includegraphics[scale = 0.45]{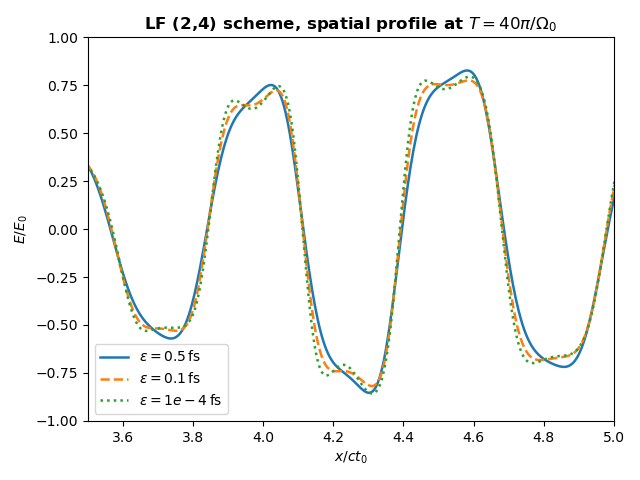}
     \includegraphics[scale = 0.45]{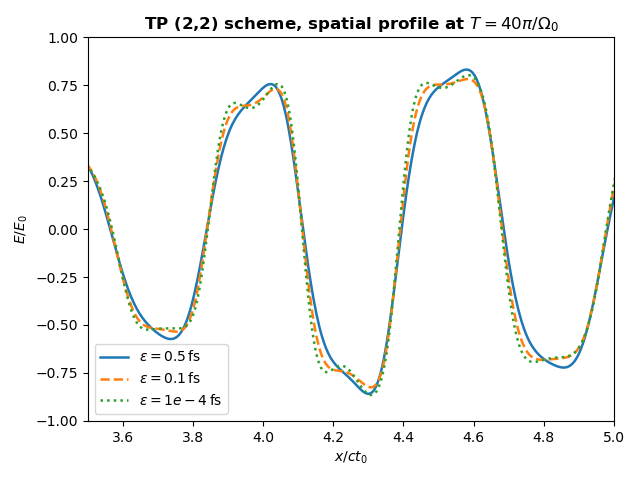}
     \includegraphics[scale = 0.45]{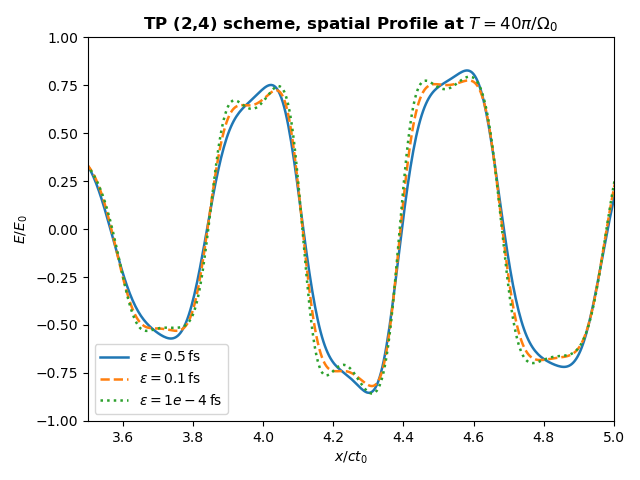}
      \includegraphics[scale = 0.45]{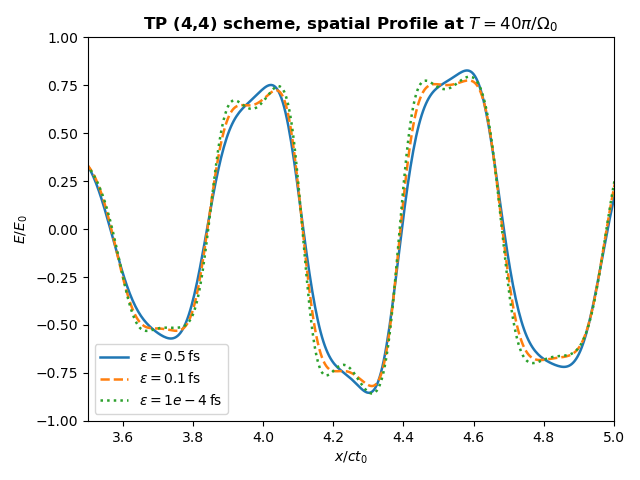}
    \caption{Spatial configuration of the electric field at the final time, $T = 40\pi/\Omega_0$}
    \label{fig: silicaa}
\end{figure}

\begin{figure}[H]
    \centering
     \includegraphics[scale = 0.45]{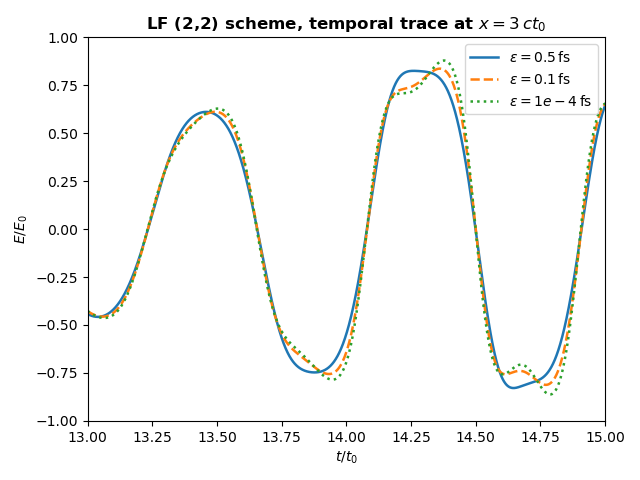}
     \includegraphics[scale = 0.45]{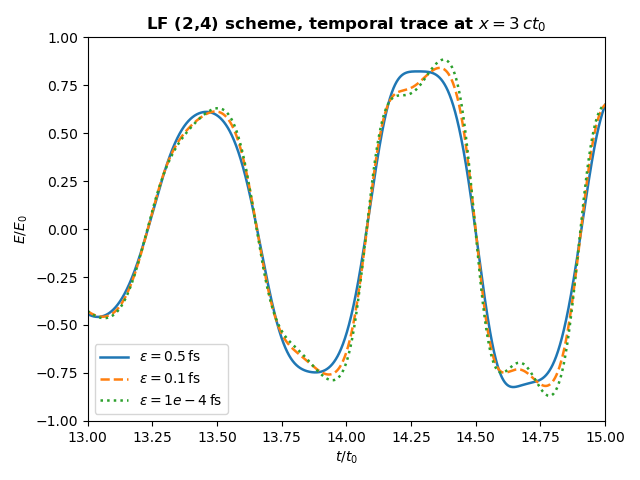}
     \includegraphics[scale = 0.45]{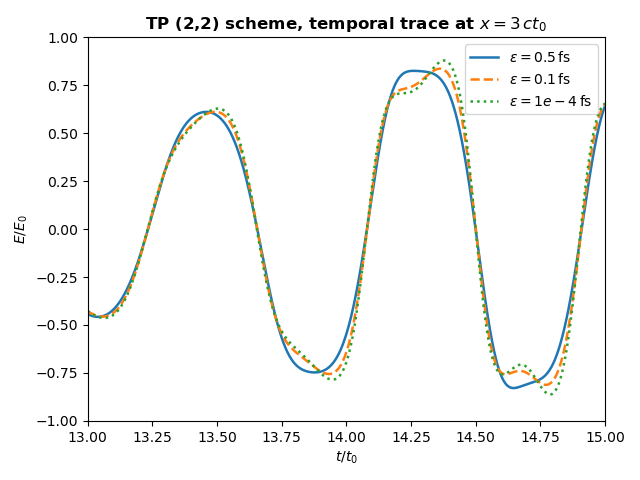}
     \includegraphics[scale = 0.45]{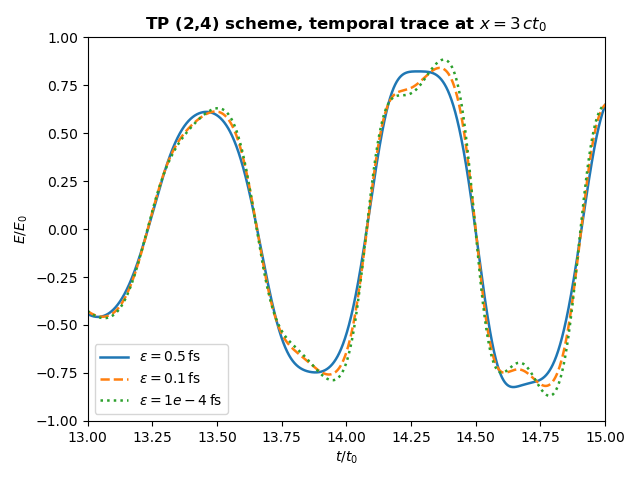}
      \includegraphics[scale = 0.45]{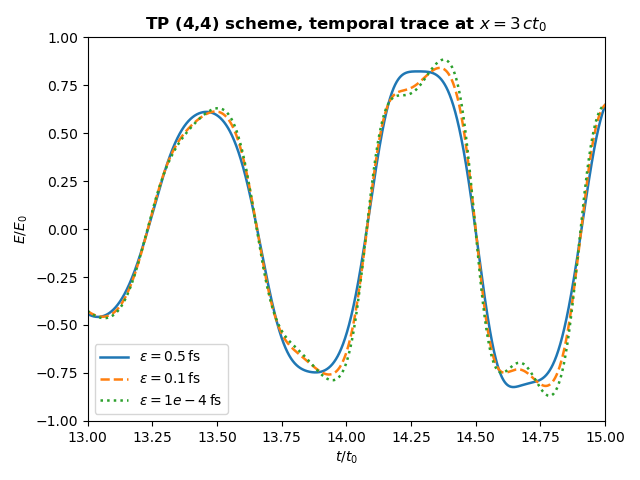}
    \caption{Temporal configuration of $E$ at $x = 3ct_0$ }
    \label{fig: silicab}
\end{figure}
In Fig. \ref{fig: silicaa}-\ref{fig: silicac}, we present the electric field $E$ computed by the leapfrog, trapezoidal and extrapolated trapezoidal FDTD schemes in the time and frequency domain for the model \eqref{KDLaa} with different relaxation times. Self-steepening effects are evident in our simulations, manifesting as sharp features in the pulse profile. This steepening becomes more pronounced as the relaxation time $\varepsilon$ decreases, consistent with prior observations in \cite{PengJCP2020, gilles1999electromagnetic}. In contrast, a finite relaxation time moderates this effect relative to the instantaneous Kerr model. Fig.  \ref{fig: silicac} illustrates the presence of third and fifth harmonic generation, which plays a key role in the formation of daughter pulses. The strength of these harmonic components diminishes with increasing $\varepsilon$, resulting in weaker daughter waves. At position $x = 4ct_0$, their frequencies are observed to be slightly below $3$ and $5$, in agreement with previous findings \cite{PengJCP2020, gilles1999electromagnetic}.

In summary, increasing the relaxation time in the Kerr nonlinearity reduces both the self-steepening and the strength of harmonic generation. Finally, as part of verifying the preservation of physical properties in our simulations, we note that no negative values of $\chi_h^n$ were encountered in any case.

\begin{figure}[H]
    \centering
     \includegraphics[scale = 0.45]{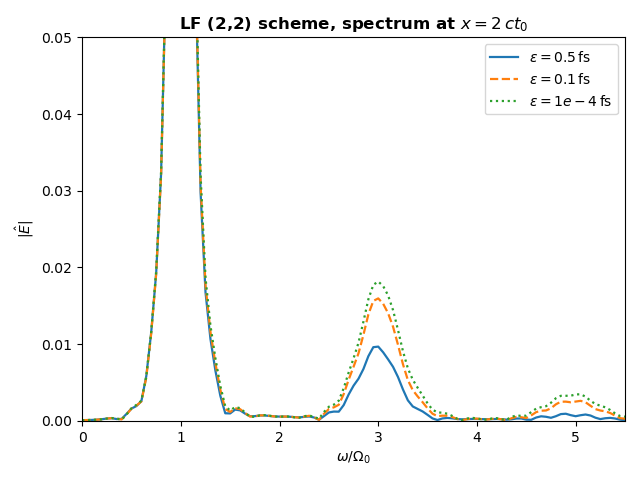}
     \includegraphics[scale = 0.45]{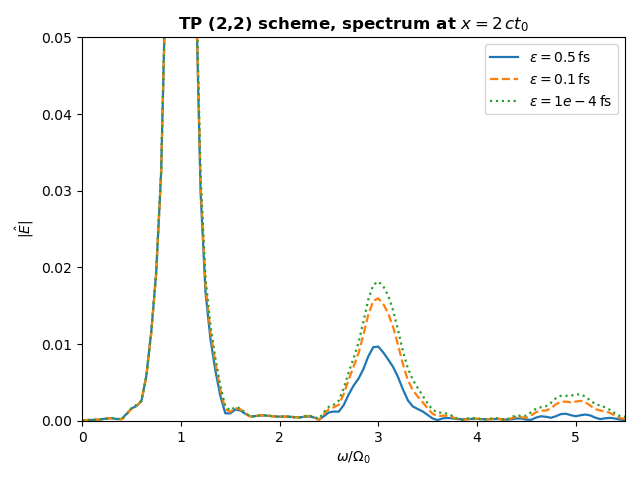}
     \includegraphics[scale = 0.45]{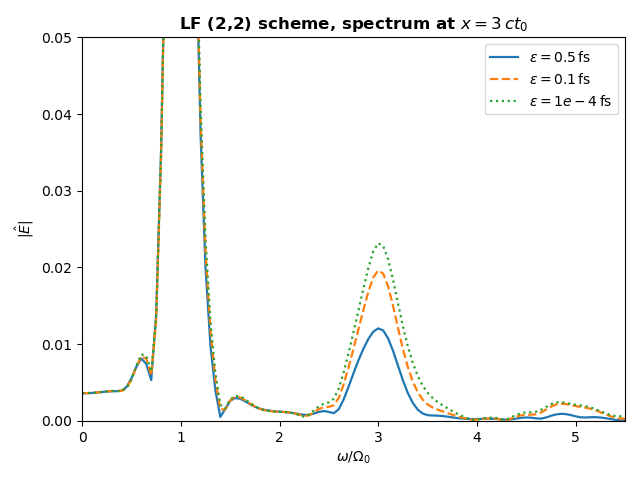}
      \includegraphics[scale = 0.45]{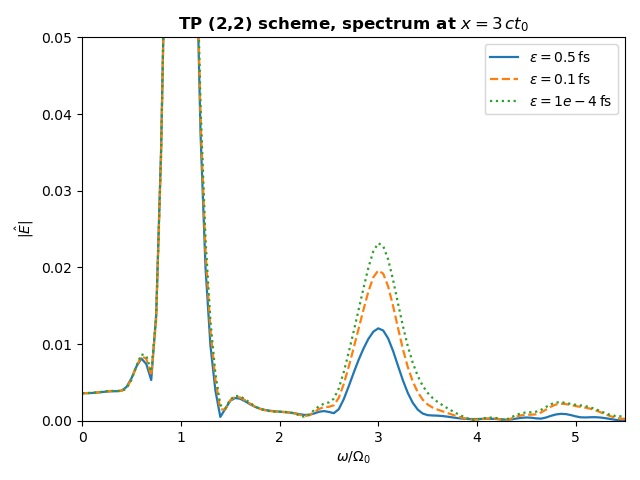}
     \includegraphics[scale = 0.45]{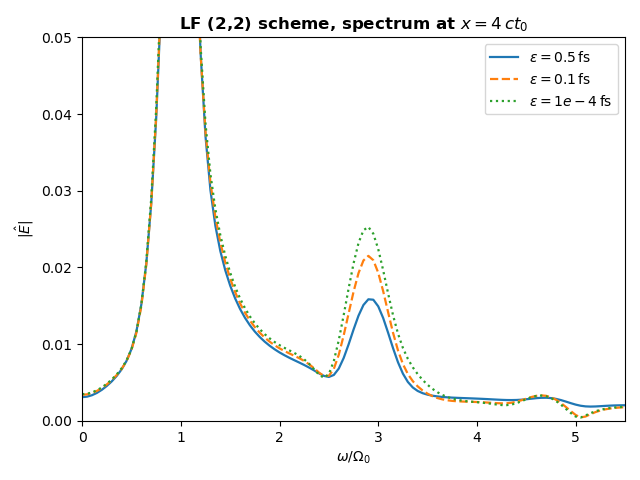}
     \includegraphics[scale = 0.45]{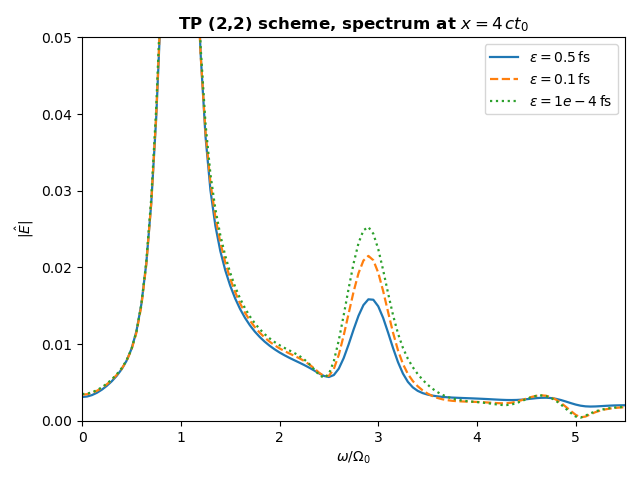}
    \caption{The computed electric field in the frequency domain.}
    \label{fig: silicac}
\end{figure}

\section{Conclusion and future work}\label{sec:conclusion}

In this work, we developed several families of high-order FDTD methods for the one-dimensional Maxwell--Kerr--Debye--Lorentz system, a nonlinear and dispersive model that is widely used for describing ultrashort pulse propagation in optical media. Our primary goal was to design structure preserving numerical schemes, i.e., methods that remain faithful to key properties of the underlying physics such as the positivity of the nonlinear susceptibility, asymptotic  in the stiff relaxation limit, energy dissipation, and robustness across stiff relaxation regimes. To this end, we introduced staggered
$2M$th-order spatial discretizations together with carefully designed seond order time integrators, and we established
rigorous discrete energy estimates that mimic the continuous energy laws. 

More precisely, we constructed and analyzed two classes of $(2,2M)$ finite difference schemes using staggered $2M$th-order spatial discretizations and either explicit leapfrog (FDTD) or implicit trapezoidal time integrators. For both types of schemes, we proved that the discrete susceptibility remains nonnegative provided it is initially nonnegative, and we derived a discrete energy identity showing energy stability with the correct dissipation mechanisms arising from Debye relaxation and Lorentz damping. Moreover, the proposed schemes are asymptotic-preserving in the stiff relaxation limit $\varepsilon\to 0$, so that stability and accuracy are maintained uniformly as the constitutive model transitions from Kerr--Debye--Lorentz to the Kerr--Lorentz regime. While the family of leapfrog schemes requires a CFL restriction leading to conditionally stable schemes, the trapezoidal schemes enjoy unconditional energy stability and thus offer additional flexibility for practical simulations.

To further improve efficiency, we extended the approach to high-order temporal accuracy by applying Richardson extrapolation to the $(2,2M)$ trapezoidal FDTD schemes, resulting in $(2M,2M)$ methods that achieve high-order accuracy in both time and space under the mild scaling $\Delta t=\mathcal O(h)$ and have bounded energy estimates. This substantially reduces the cost relative to standard high-order strategies that require $\Delta t=\mathcal O(h^M)$ to balance temporal and spatial errors. Rigorously proving structure preservation, i.e., positivity preservation, asymptotic preservation and energy decay, in the extrapolated scheme remains an open area for further investigation. 

Numerical experiments confirmed the convergence rates for multiple values of the relaxation time $\varepsilon$, and demonstrated the ability of the proposed schemes to capture nonlinear optical phenomena such as soliton-like propagation, self-steepening, and odd harmonic generation in dispersive Kerr media. 

Our work extends the energy-stable and asymptotic-preserving framework developed in \cite{PengJCP2020} in two important directions. First, while \cite{PengJCP2020} focused on Galerkin-based spatial discretizations, we developed a high-order staggered finite-difference framework that is naturally compatible with FDTD methodology and retains the desirable structure-preserving features. Second, beyond second-order time-stepping, we introduced systematic high-order temporal discretizations through extrapolation, together with uniform-in-time energy bounds for the resulting $(2M,2M)$ schemes. These advances provide a flexible and practical toolkit for high-fidelity simulations of nonlinear wave propagation in optical materials. 

Several directions merit further investigation. An immediate extension is to generalize the proposed structure-preserving schemes to two and three dimensions and to incorporate more realistic geometries and boundary treatments. It would also be of interest to adapt the framework to other nonlinear optical models beyond the Kerr--Debye--Lorentz system, and to develop uncertainty-aware variants that account for variability in material parameters. Finally, these solvers could serve as accurate forward models for inverse problems in parameter identification and computational material design.

\section*{Acknowledgments and Author Statement}
Dr. Bokil's research was partially supported by NSF grants \#1720116 and \#2012882. Dr. Gibson's research was partially supported by NSF grant \#2012882. The authors report that there are no competing interests to declare.

\appendix

\section{Derivation of the Modified Richardson Extrapolation Coefficients}
\label{app:RE-coefficients}

We provide the complete derivation of the Richardson extrapolation coefficients used in \eqref{eq:general-alphas}. Recall that
\[
\tilde{\alpha}_i=P_i(0),
\]
where
\[
P_i(z)
=
L_i(z)
-
\frac{L_i'(0)}{Q'(0)}Q(z).
\]
Therefore,
\begin{equation}
\tilde{\alpha}_i
=
L_i(0)
-
\frac{L_i'(0)}{Q'(0)}Q(0).
\label{eq:app-alpha-start}
\end{equation}
We now derive an explicit expression for these coefficients. From \eqref{Llang},
\[
L_i(0)
=
\prod_{\substack{m=1\\m\neq i}}^{2M-1}
\frac{-r_m}{r_i-r_m}.
\]
Multiplying both the numerator and denominator of each factor by $-1$ gives
\begin{equation}
L_i(0)
=
\prod_{\substack{m=1\\m\neq i}}^{2M-1}
\frac{r_m}{r_m-r_i}.
\label{eq:app-Li0}
\end{equation}
Moreover, logarithmic differentiation of $L_i$ in \eqref{Llang} gives
\[
\frac{L_i'(z)}{L_i(z)}
=
\sum_{\substack{m=1\\m\neq i}}^{2M-1}
\frac{1}{z-r_m}.
\]
Setting $z=0$, we obtain
\[
\frac{L_i'(0)}{L_i(0)}
=
-\sum_{\substack{m=1\\m\neq i}}^{2M-1}
\frac{1}{r_m},
\]
and therefore
\begin{equation}
L_i'(0)
=
-L_i(0)
\sum_{\substack{m=1\\m\neq i}}^{2M-1}
\frac{1}{r_m}.
\label{eq:app-Liprime0}
\end{equation}
Similarly, logarithmic differentiation of $Q$ in \eqref{Qlang} gives
\[
\frac{Q'(z)}{Q(z)}
=
\sum_{m=1}^{2M-1}
\frac{1}{z-r_m}.
\]
Evaluating at $z=0$ yields
\[
\frac{Q'(0)}{Q(0)}
=
-\sum_{m=1}^{2M-1}\frac{1}{r_m}.
\]
Hence,
\begin{equation}
\frac{Q(0)}{Q'(0)}
=
-\frac{1}
{\displaystyle\sum_{m=1}^{2M-1}r_m^{-1}}.
\label{eq:app-Qratio}
\end{equation}
Substituting \eqref{eq:app-Liprime0} and \eqref{eq:app-Qratio} into \eqref{eq:app-alpha-start}, we obtain
\begin{align}
\tilde{\alpha}_i
&=
L_i(0)
-
L_i'(0)\frac{Q(0)}{Q'(0)}
\notag\\
&=
L_i(0)
-
\left(
-L_i(0)
\sum_{\substack{m=1\\m\neq i}}^{2M-1}r_m^{-1}
\right)
\left(
-\frac{1}
{\displaystyle\sum_{m=1}^{2M-1}r_m^{-1}}
\right)
\notag\\
&=
L_i(0)
-
L_i(0)
\frac{
\displaystyle\sum_{\substack{m=1\\m\neq i}}^{2M-1}r_m^{-1}
}{
\displaystyle\sum_{m=1}^{2M-1}r_m^{-1}
}
\notag\\
&=
L_i(0)
\left[
1-
\frac{
\displaystyle\sum_{\substack{m=1\\m\neq i}}^{2M-1}r_m^{-1}
}{
\displaystyle\sum_{m=1}^{2M-1}r_m^{-1}
}
\right].
\label{eq:app-alpha-intermediate}
\end{align}
Writing the first term inside the brackets over the same denominator gives
\begin{align}
1-
\frac{
\displaystyle\sum_{\substack{m=1\\m\neq i}}^{2M-1}r_m^{-1}
}{
\displaystyle\sum_{m=1}^{2M-1}r_m^{-1}
}
&=
\frac{
\displaystyle\sum_{m=1}^{2M-1}r_m^{-1}
-
\displaystyle\sum_{\substack{m=1\\m\neq i}}^{2M-1}r_m^{-1}
}{
\displaystyle\sum_{m=1}^{2M-1}r_m^{-1}
}
=
\frac{r_i^{-1}}
{\displaystyle\sum_{m=1}^{2M-1}r_m^{-1}}.
\end{align}
Therefore,
\begin{equation}
\tilde{\alpha}_i
=
\frac{L_i(0)}
{r_i\displaystyle\sum_{m=1}^{2M-1}r_m^{-1}}.
\label{eq:app-alpha-Li}
\end{equation}
Using \eqref{eq:app-Li0}, we obtain
\begin{equation}
\tilde{\alpha}_i
=
\frac{1}
{r_i\displaystyle\sum_{j=1}^{2M-1}r_j^{-1}}
\prod_{\substack{m=1\\m\neq i}}^{2M-1}
\frac{r_m}{r_m-r_i},
\qquad
i=1,\ldots,2M-1.
\label{eq:app-general-alphas}
\end{equation}

For the dyadic refinement factors,
\[
r_i=2^{-(i-1)},
\]
and therefore
\[
r_i^{-1}=2^{i-1}.
\]
Moreover,
\[
\sum_{j=1}^{2M-1}r_j^{-1}
=
\sum_{j=1}^{2M-1}2^{j-1}
=
\sum_{j=0}^{2M-2}2^j
=
2^{2M-1}-1.
\]
Consequently,
\begin{equation}
\frac{1}
{r_i\displaystyle\sum_{j=1}^{2M-1}r_j^{-1}}
=
\frac{2^{i-1}}{2^{2M-1}-1}.
\label{eq:app-dyadic-factor}
\end{equation}
The product in \eqref{eq:app-general-alphas} can also be expressed in terms of the dyadic refinement factors. For $m\neq i$,
\[
\frac{r_m}{r_m-r_i}
=
\frac{1}{1-r_i/r_m}.
\]
Since
\[
\frac{r_i}{r_m}
=
\frac{2^{-(i-1)}}{2^{-(m-1)}}
=
2^{m-i},
\]
it follows that
\[
\frac{r_m}{r_m-r_i}
=
\frac{1}{1-2^{m-i}}.
\]
Thus,
\begin{equation}
\prod_{\substack{m=1\\m\neq i}}^{2M-1}
\frac{r_m}{r_m-r_i}
=
\prod_{\substack{m=1\\m\neq i}}^{2M-1}
\frac{1}{1-2^{m-i}}.
\label{eq:app-dyadic-product}
\end{equation}
Finally, substituting \eqref{eq:app-dyadic-factor} and \eqref{eq:app-dyadic-product} into \eqref{eq:app-general-alphas} gives
\begin{equation}
\tilde{\alpha}_i
=
\frac{2^{i-1}}{2^{2M-1}-1}
\prod_{\substack{m=1\\m\neq i}}^{2M-1}
\frac{1}{1-2^{m-i}}
,
\qquad
i=1,\ldots,2M-1,
\end{equation}
which is the expression stated in \eqref{eq:general-alphas}.

\bibliographystyle{siam}
\bibliography{references}

\end{document}